\documentclass[11pt]{article}

\usepackage[utf8]{inputenc}
\usepackage[T1]{fontenc}
\usepackage{lmodern}
\usepackage[a4paper,top=3cm, bottom=3cm, left=2.5cm, right=2.5cm]{geometry} %
\usepackage{amsmath,amsthm,amsfonts,amssymb, stmaryrd,bbm,mathrsfs, dsfont}
\usepackage{mathtools}
\usepackage{graphicx}
\usepackage{enumitem}
\usepackage{url}
\usepackage{appendix}
\usepackage{enumitem} 
\newcounter{dummy}
\makeatletter
\newcommand\myitem[1][]{\item[#1]\refstepcounter{dummy}\def\@currentlabel{#1}}
\makeatother 

\usepackage[english]{babel}

\usepackage[usenames,dvipsnames]{xcolor}
\usepackage[labelfont={bf,sf}, textfont={sf}]{caption}
\usepackage{subfigure}
\usepackage
[final]
{changes}

\definecolor{darkblue}{rgb}{0.0, 0.2, 0.6}
\usepackage[pdftitle={},
  pdfauthor={},
colorlinks=true,linkcolor=RoyalBlue,urlcolor=MidnightBlue,citecolor=darkblue,bookmarks=true,bookmarksopen=true,bookmarksopenlevel=2,unicode=true,linktocpage]{hyperref}

\numberwithin{equation}{section}

\usepackage[capitalise,noabbrev]{cleveref}

\makeatletter
\renewenvironment{abstract}{%
	\small
	\begin{center}%
		{\bfseries Abstract\vspace{-.5em}\vspace{\z@}}%
	\end{center}%
	\quotation}
{\endquotation}
\makeatother

\newcommand\bb{\mathbb}
\renewcommand\cal{\mathcal}

\renewcommand\frak{\mathfrak}
\renewcommand\bf{\boldsymbol}
\renewcommand\tt{\mathtt}

\newcommand\eps{\varepsilon}

\newcommand{\cle}{\mathrm{CLE}}

\newcommand{\Deg}{\mathsf{Deg}}

\newcommand{\abs}[1]{\left\lvert#1\right\rvert}

\theoremstyle{plain}
\newtheorem{thm}{Theorem}[section]
\newtheorem*{thm*}{Theorem}
\newtheorem{prop}[thm]{Proposition}
\newtheorem*{prop*}{Proposition}

\newtheorem{defn}[thm]{Definition}
\newtheorem*{Def*}{Definition}
\newtheorem{lem}[thm]{Lemma}
\newtheorem*{Cor*}{Corollary}
\newtheorem{Cor}[thm]{Corollary}

\newtheorem{Thmx}{Theorem}

\theoremstyle{definition}

\newtheorem*{Ex*}{Example}
\newtheorem{rem}[thm]{Remark}
\newtheorem*{rem*}{Remarques}

\title{\Large{\textbf{The scaling limit of the volume of loop--$O(n)$ quadrangulations}}}
\author{Élie Aïdékon\footnote{\scriptsize SMS, Fudan University, China, \texttt{aidekon@fudan.edu.cn}} \and William Da Silva\footnote{\scriptsize University of Vienna, Austria, \texttt{william.da.silva@univie.ac.at}} \and Xingjian Hu\footnote{\scriptsize SMS, Fudan University, China, \texttt{22110180020@m.fudan.edu.cn}}}
\date{}

\begin{document}

\maketitle

\begin{abstract}
We study the volume of rigid loop--$O(n)$ quadrangulations with a boundary of length $2p$ in the non-generic critical regime, {for all $n\in (0,2]$}. We prove that, as the half-perimeter $p$ goes to infinity, the volume scales in distribution to an explicit random variable. This limiting random variable is described in terms of the multiplicative cascades of Chen, Curien and Maillard \cite{chen2020perimeter}, or alternatively (in the dilute case) as the law of the area of a unit-boundary $\gamma$--quantum disc, as determined by Ang and Gwynne \cite{ang2021liouville}, for suitable $\gamma$. 

Our arguments go through a classification of the map into several regions, where we rule out the contribution of bad regions to be left with a tractable portion of the map. One key observable for this classification is a Markov chain which explores the nested loops around a size-biased vertex pick in the map, making explicit the spinal structure of the discrete multiplicative cascade.

{We stress that our techniques enable us to include the boundary case $n=2$, that we define rigorously, and where the nested cascade structure is that of a critical branching random walk. In that case the scaling limit is given by the limit of the derivative martingale and is inverse-exponentially distributed, which answers a conjecture of \cite{aidekon2022growth}.}
\end{abstract}

\tableofcontents

%
%

\bigskip
\noindent 
\textbf{Acknowledgments.} We thank Nicolas Curien for enlightening discussions on loop--$O(n)$ planar maps and for sharing with us the general argument leading to the proof of Theorem \ref{thm: main} assuming uniform integrability. We also thank Joonas Turunen for discussions pertaining to the phase diagram of the $O(2)$ model and Emmanuel Kammerer for sharing with us his recent work on that phase diagram in the case of bipartite maps. 
{ Finally, we are grateful to two anonymous referees for their very careful reading and detailed comments, which substantially improved the presentation.}
E.A. was supported by NSFC grant QXH1411004. W.D.S. acknowledges the support of the Austrian Science Fund (FWF) grant on  “Emergent branching structures in random geometry” (DOI: 10.55776/ESP534).

%
%

\section{Introduction}
\label{sec: intro}

Loop--$O(n)$ planar maps form one of the classical models of statistical physics. They consist in sampling a planar map together with a collection of self-avoiding and non-touching nested loops. Upon driving the parameters of the model to criticality, it is predicted by physics { \cite{kostov1989n,eynard1992n,kostov1992multicritical,di19952d,eynard1995exact}} that a variety of universality classes of \emph{two-dimensional quantum gravity} may be reached in a suitable scaling limit of the maps. One key observable for the geometric features of these  maps is their \textbf{volume}, defined as the number of vertices.

In this paper, we investigate the question of the volume for the \textbf{rigid model on quadrangulations} in the non-generic critical regime, in the framework of Borot, Bouttier and Guitter \cite{borot2011recursive}. Precisely, we prove that the volume of these quadrangulations has an explicit scaling limit as their perimeter goes to infinity. The limit is described as the limit of the Malthusian martingale of the Chen--Curien--Maillard multiplicative cascade \cite{chen2020perimeter}. In the dilute regime, it also matches for suitable $\gamma$ the area of a unit-boundary $\gamma$--quantum disc, as determined by Ang and Gwynne \cite{ang2021liouville}. Our proof builds on the gasket decomposition of Borot, Bouttier and Guitter \cite{borot2011recursive} and the multiplicative cascades of Chen, Curien and Maillard \cite{chen2020perimeter}.

\subsection{Rigid loop--$O(n)$ quadrangulations}

\medskip
\noindent \textbf{Definition of the loop model.}
\added{To introduce the model, we follow the exposition given in \cite{chen2020perimeter}.} All planar maps considered in this paper are rooted, \added{\textit{i.e.} \ come with a distinguished oriented root edge}. The root face (or external face) $\frak f_r$ is defined to be the face of the map lying to the right of the oriented root edge. The other faces are called internal faces. 
{We say that a planar map $\frak q$ is a \textbf{quadrangulation with a boundary} if all its internal faces are of degree $4$, while the external face may have a different (even) degree, called the \textbf{perimeter} of $\frak q$.}
A (rigid) loop configuration on $\frak q$ is a collection \added{$\bf \ell := \{\ell_1, \ell_2, \cdots, \ell_k\}$} of nested disjoint and self-avoiding loops in the dual map of $\frak q$ (\textit{i.e.} crossing faces of $\frak q$), avoiding the external face $\frak f_r$, and with the rigidity constraint that loops must exit faces through opposite edges. For $p\ge 0$, let $\cal O_p$ be the set of loop-decorated quadrangulations $(\frak q, \bf \ell)$ with perimeter $2p$.

\bigskip
\begin{figure}[h]
\begin{center}
\includegraphics[scale=0.65]{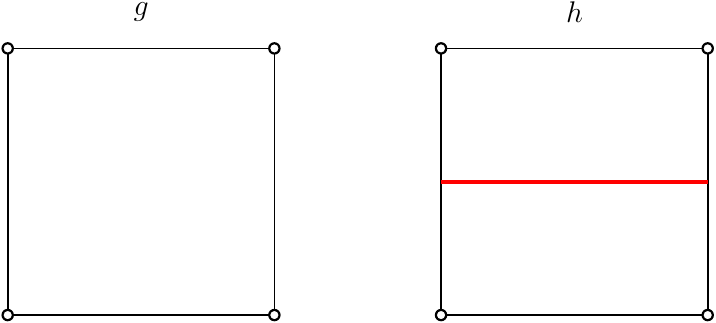}
\end{center}
\caption{The two types of faces of a loop-decorated quadrangulation $(\frak q,\bf \ell)$, with corresponding weights $g$ and $h$. Each loop receives additional (global) weight $n$.}
\label{fig: faces loop-decorated quadrangulation}
\end{figure}

The rigidity condition imposes that faces are of two possible types: either they are empty, or crossed by a loop through opposite edges (see \cref{fig: faces loop-decorated quadrangulation}). We now introduce the loop--$O(n)$ measure, depending on parameters $g,h\ge 0$ and $n\in (0,2]$. Each quadrangle receives local weight $g\ge 0$ or $h\ge 0$ according to its type (empty or crossed), whereas each loop gets an extra global weight $n$, so that the total weight of a loop-decorated quadrangulation $(\frak q,\bf\ell)$ is
\begin{equation} \label{eq: weights O(n)}
\mathsf{w}_{(n;g,h)}(\frak q,\bf\ell) := g^{|\frak q| -|\bf \ell |} h^{|\bf \ell |} n^{\#\bf\ell},
\end{equation}
where $|\frak q|$ is the number of internal faces of $\frak q$, $|\bf \ell |$ is the total length (\textit{i.e.} number of faces crossed) of the loops in $\bf\ell$, and $\#\bf\ell$ is the number of loops. When the partition function
\begin{equation} \label{eq: partition function O(n)}
F_p(n;g,h) := \sum_{(\frak q,\bf\ell)\in \cal{O}_p} \mathsf{w}_{(n;g,h)}(\frak q,\bf\ell),
\end{equation}
is finite, we say that $(n;g,h)$ is \emph{admissible}, and we introduce the loop--$O(n)$ probability measure
\[
\bb P^{(p)}_{(n;g,h)}(\cdot) := \frac{\mathsf{w}_{(n;g,h)}(\cdot)}{F_p(n;g,h)},
\]
on the set $\cal O_p$, associated with the weights \eqref{eq: weights O(n)}. In the sequel we shall often drop the subscript and write $\bb P^{(p)}$ since the weight sequence will be fixed once and for all. As an illustration, we represented in \cref{fig: loop-decorated quadrangulation} a loop-decorated quadrangulation with half-perimeter $p=20$.

\bigskip
\begin{figure}[h]
\begin{center}
\includegraphics[scale=0.7]{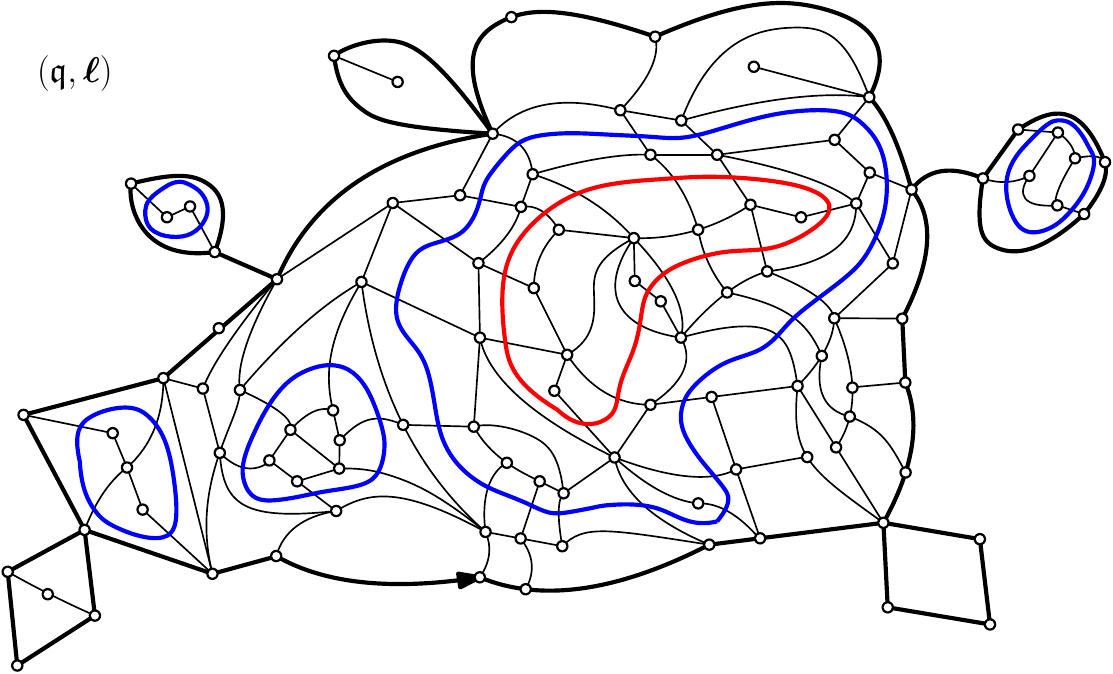}
\end{center}
\caption{A loop-decorated quadrangulation $(\frak q,\bf \ell)$. The boundary of the map is shown in bold, with a distinguished oriented root edge. The outermost loops are shown in blue, one interior loop is shown in red.}
\label{fig: loop-decorated quadrangulation}
\end{figure}

\bigskip
\noindent \textbf{Phase diagram.} In the case $n\in (0,2)$, Borot, Bouttier and Guitter \cite{borot2011recursive} used the gasket decomposition to classify the parameters into three phases (\emph{subcritical}, \emph{generic critical}, \emph{non-generic critical}) where the model exhibits very different large-scale geometry features (see \cref{fig: O(n) phase diagram}). Subcritical planar maps are expected to scale to Aldous’ continuum random tree (CRT). The generic critical regime is expected to lead to the same geometry as for regular quadrangulations, with scaling to the Brownian disk \cite{bettinelli2017compact}. Our paper deals with the \textbf{non-generic critical phase}. In this case the geometry is still not well understood, see \cite{legall2011scaling} for information on the gaskets of the maps. 
When $n\in (0,2)$, this regime is defined by the equations
\begin{equation}\label{non-generic line}
g  = \frac{3}{2+b^2}\bigg(h - \frac{2-n}{2b^2}h^2\bigg), \qquad
g  \leq \frac{3h}{2(b^2-2b+3)},
\end{equation}
where  $b:=\frac1\pi \arccos \frac{n}2$ (see \cite[Equations (6.15) and (6.17)]{borot2011recursive}). The phase is called \textbf{dense} when the above inequality is strict, and \textbf{dilute} when it is an equality (see \cref{fig: O(n) phase diagram}). {It is believed that in each phase respectively, loops are simple and mutually avoiding or touch themselves and each other, in a suitable scaling limit. For two functions $f$ and $g$ defined on an unbounded set $D \subset \bb R_+$, we write $f(x) \sim g(x)$ as $x\to\infty$ if the ratio $f/g$ tends to $1$ as $x\to\infty$ (with $x\in D$).} The partition function \eqref{eq: partition function O(n)} satisfies 
\begin{equation}\label{eq: asymptotics partition function}
F_p(n;g,h) \sim C h^{-p} p^{-\alpha-1/2}, \quad \text{as } p\to \infty,
\end{equation}
for some $C>0$, and
\begin{equation} \label{eq: relation alpha and n}
\alpha := \frac32 \pm \frac{1}{\pi} \arccos(n/2) \in (1,2) \setminus \Big\{\frac32\Big\},
\end{equation}

\noindent the signs $+$ and $-$ corresponding respectively to the dilute and dense phase. 
Much more is known on the geometry of Boltzmann planar maps \cite{budd2016peeling,budd2017geometry}, which describe the gasket of loop--$O(n)$ quadrangulations -- see \cref{sec: gasket decomposition} for a definition.  One key observable in the description of the intricate geometry of these maps is the so-called peeling exploration, which was introduced by Watabiki for triangulations \cite{watabiki1995construction} (see also \cite{angel2003growth,curien2017scaling}), extended to Boltzmann maps by Budd \cite{budd2016peeling} and describes a Markovian exploration of the planar map obtained by discovering the faces given by gradually peeling the edges of the map. 

The case $n=2$ is a boundary case. To the best of our knowledge, the phase diagram of this case has so far not been covered in the mathematical literature. Still, in \cref{s:partition function O(2)}, we show that the regime 
\begin{equation}\label{n=2 line}
g = \frac{3}{2}\bigg(h - \frac{\pi^2}{2}h^2\bigg), \quad g\leq \frac{h}{2},
\end{equation}

\noindent gives rise to a well-defined non-generic critical loop--$O(2)$ quadrangulation. Notice that \eqref{n=2 line} is simply the limit as $n\to 2$ of \eqref{non-generic line}. The partition function \eqref{eq: partition function O(n)}  satisfies 
\begin{equation}\label{asymptotics}
F_p(2;g,h) \sim\begin{cases}
 {C} h^{-p} p^{-2}, & \text{if } g=\frac{h}{2}, \\
 {C} h^{-p} p^{-2}\ln p, & \text{if }g<\frac{h}{2}.
\end{cases}
\end{equation}

\noindent {We refrain from using the terms \textit{dense} and \textit{dilute} to distinguish the phases $g=h/2$ and $g<h/2$ as we expect that in both cases, after suitable embedding and under appropriate (possibly different) scaling, the model converges to $\cle_4$ on critical Liouville quantum gravity.} 
In a paper that appeared on the same day on the arXiv, Kammerer used a more systematic approach to establish asymptotics for $F_p$ for the rigid loop--$O(2)$ model on bipartite planar maps (not necessarily quadrangulations). These asymptotics {involve, in general, a} slowly varying function $L(p)$ which is between a constant and $\ln(p)$. We still decided to include \eqref{n=2 line} and \eqref{asymptotics} in the appendix since we follow a more elementary approach by taking limits (that is specific to our model).

There has been very little work on the loop--$O(2)$ model. First, we mention that in \cite{budd2018peeling} and \cite{KammererEmmanuel2023Ol3m}, Boltzmann maps with face degree decay of order $k^{-2}$ are discussed, which apply to the gasket of loop--$O(2)$ quadrangulations when $g=\frac{h}{2}$ in our setting. Furthermore, Kammerer proved in particular in \cite{KammererEmmanuel2023Dotm} a scaling limit for some distances from the loops to the boundary, expressed in terms of $\cle_4$ on the critical quantum disc.

\begin{figure}[h]
\medskip
	\begin{center}
		\includegraphics[width=0.75\textwidth]{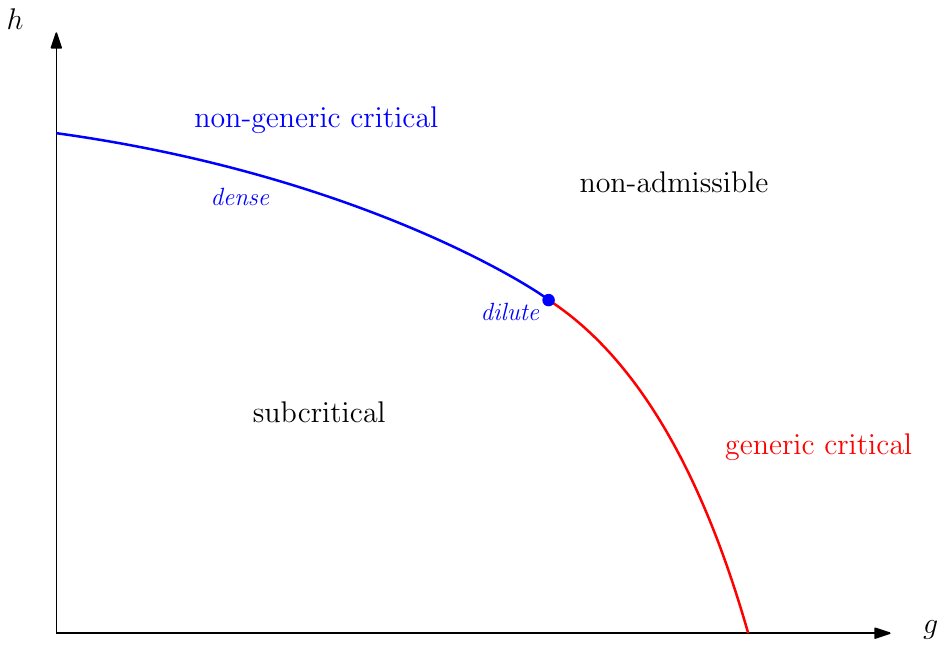}
	\end{center}
		\caption{Phase diagram of the $O(n)$ model on quadrangulations \cite{borot2011recursive,budd2018peeling} for fixed $n\in(0,2)$ (the diagram looks roughly the same for all $n$). The critical line separates the \emph{non-admissible} region (where the partition function \eqref{eq: partition function O(n)} blows up) from the subcritical regime, where the maps are believed to yield trees in the limit. On the critical line, interesting behaviours are expected, with convergence either to the Brownian disk in the generic critical regime (red), or other objects in the non-generic critical regime (blue). The latter can be further split into dense (blue line) or dilute (blue point) phases, where $\alpha := \frac32 - \frac{1}{\pi} \arccos(n/2)$ or $\alpha := \frac32 + \frac{1}{\pi} \arccos(n/2)$ \added{in \eqref{eq: relation alpha and n}} respectively.}
		\label{fig: O(n) phase diagram}
\end{figure}

\bigskip

\fbox{
\begin{minipage}{0.9\textwidth}
\emph{From now on, we fix a set of parameters $(n;g,h)$ in the \textbf{non-generic critical region}, \emph{i.e.}\ satisfying either \eqref{non-generic line} when $n\in (0,2)$ or \eqref{n=2 line} when $n=2$. The index $\alpha$ is then fixed as in \eqref{eq: relation alpha and n}, setting $\alpha=\frac32$ when $n=2$.}
\smallskip
\end{minipage}
}

\subsection{Previously known results on the volume}
\label{sec: intro prev known}

\noindent \textbf{Volume of Boltzmann maps.} The volume of a planar map, defined as its number of vertices, plays a central role in the development of planar maps, both from the combinatorial and probabilistic perspectives. In the 1960s, Tutte \cite{tutte1962census} first enumerated planar triangulations or quadrangulations of a polygon with fixed volume, relying on a recursive procedure known as Tutte's equation. 

In the case of loop--$O(n)$ quadrangulations {($n\in(0,2)$)}, the volume of the gasket is fairly well understood (recall that the gasket corresponds to a Boltzmann planar map -- see \cref{sec: gasket decomposition} for more details). In fact, the scaling limit of the volume of Boltzmann planar maps as the perimeter goes to infinity was established by Budd \cite{budd2016peeling}, and Budd and Curien \cite{budd2017geometry}. To clarify, let us consider under $\bb P$ a Boltzmann planar map $\frak B^{(p)}_{\hat{\bf{g}}}$ with non-generic critical weight sequence $\hat{\bf{g}}$ (as made explicit by \eqref{non-generic line}--\eqref{n=2 line} and \eqref{eq: fixed point equation}), and fixed perimeter $2p$. Then the \emph{expectation} of the volume $V(\frak B^{(p)}_{\hat{\bf{g}}})$ of $\frak B^{(p)}_{\hat{\bf{g}}}$ is easily determined as a ratio of (pointed/non-pointed) Boltzmann partition functions.\footnote{{Similarly to \eqref{eq: partition function O(n)}, the (non-pointed) Boltzmann partition function is defined by summing the corresponding Boltzmann weights over all possible bipartite maps $\mathfrak{m}$ with perimeter $2p$. The \emph{pointed} Boltzmann partition function corresponds to summing the same weights over all pairs $(\mathfrak{m}, v)$, where $\mathfrak{m}$ is as above and $v$ is a vertex of $\mathfrak{m}$.}}
The analysis of these partition functions yields the {following} estimate: {for some constant $\Gamma>0$ depending on the weight sequence, and with $\alpha$ as in \eqref{eq: relation alpha and n},}
\begin{equation} \label{eq: mean volume boltzmann}
\bb E \big[V(\frak B^{(p)}_{\hat{\bf{g}}})\big] \sim \Gamma p^{\alpha}, \quad \text{as } p\to\infty.
\end{equation}
Budd and Curien further proved \cite[Proposition 3.4]{budd2017geometry} that, as a matter of fact,
\begin{equation} \label{eq: scaling volume boltzmann}
p^{-\alpha}V(\frak B^{(p)}_{\hat{\bf{g}}}) \overset{(\mathrm{d})}{\longrightarrow} V_\infty, \quad \text{as } p\to \infty,
\end{equation}
for some explicit random variable $V_{\infty}$. The proof expresses the Laplace transform of the volume in terms of Boltzmann partition functions, carrying out the analysis on these (pointed/non-pointed) partition functions. 

Much less is known about the volume of the loop--$O(n)$ quadrangulations themselves. One reason is that similar techniques do not apply since, contrary to Boltzmann maps, there does not seem to be any tractable expression of the pointed $O(n)$ partition function. Instead, Budd \cite{budd2018peeling} takes a different, more probabilistic, route going through \emph{ricocheted random walks} to work out the asymptotics of the pointed $O(n)$ partition function in the case $n\in (0,2)$. His result in \cite[Proposition 9]{budd2018peeling} is the following. Let $(\frak q, \bf \ell)$ be a loop--$O(n)$ quadrangulation under $\bb P^{(p)}_{(n;g,h)}$, with $n\in (0, 2)$ (recall that we fixed a non-generic critical set of parameters $(n;g,h)$). Then {there exists a constant $\Lambda>0$ such that} the expected volume satisfies
\begin{equation} \label{eq: mean volume}
\overline{V}(p) := \bb E^{(p)}_{(n;g,h)} [V(\frak q)] \sim \Lambda p^{\theta_\alpha}, \quad \text{as } p\to \infty,
\end{equation}
where
\begin{equation} \label{eq: def theta_alpha}
\theta_\alpha := \min(2, 2\alpha-1).
\end{equation}
Note that $\theta_\alpha$ separates the dilute and dense phases described after \eqref{eq: relation alpha and n}, in the sense that in the dilute phase $\theta_\alpha = 2$, whereas in the dense phase $\theta_\alpha = 2\alpha-1$. By convention, we also set $\overline{V}(0) = 0$.

In the case $n=2$, we will establish in \cref{s:construction O(2)} the following volume estimates. For the gasket viewed as a Boltzmann map $\frak B^{(p)}_{\hat{\bf{g}}}$ with weight sequence $\hat{\bf{g}}$ as in \cref{eq: fixed point equation}, 
\begin{equation}\label{eq: mean volume boltzmann n=2}
\bb E \big[V(\frak B^{(p)}_{\hat{\bf{g}}})\big] \sim 
\begin{cases}
p^{\frac32}, & \text{ if } g = \frac{h}{2},\\
\frac{p^{\frac32}}{\ln p}, & \text{ if } g < \frac{h}{2}. 
\end{cases}
\end{equation}
\noindent In addition, the {mean} volume itself satisfies
\begin{equation}\label{eq: mean volume 2}
\overline{V}(p) \sim \Lambda \begin{cases}
  p^2, & \text{if } g=\frac{h}{2}, \\
 \frac{p^2}{\ln(p)}, & \text{if }g< \frac{h}{2}.
\end{cases}
\end{equation}

Despite the information on the mean volume \eqref{eq: mean volume} provided by Budd, no scaling limit in the spirit of \eqref{eq: scaling volume boltzmann} is known for the volume of loop--$O(n)$ quadrangulations. The main result of the present paper addresses this question by providing an explicit scaling limit.

\bigskip
\noindent \textbf{The Chen--Curien--Maillard multiplicative cascades.} In \cite{chen2020perimeter}, Chen, Curien and Maillard give an explicit and very convincing conjecture on the scaling limit of the volume of loop--$O(n)$ quadrangulations when $n\in(0, 2)$. They introduce a discrete cascade $(\chi^{(p)}(u), u\in \cal U)$ {indexed by the Ulam tree $\cal U$}, which informally records the half-perimeters of the $O(n)$ loops at each \emph{generation}, starting from the outermost loops (generation $1$) and exploring each loop in an inductive way. They prove that as the perimeter of the quadrangulation goes to infinity, the discrete cascade scales to a continuous \textbf{multiplicative cascade} $(Z_{\alpha}(u),u\in\cal U)$. This limiting branching process is related to the jumps of a spectrally positive $\alpha$--stable process, with $\alpha \in \big(1, \frac32\big)\cup \big(\frac32, 2\big)$ given by \eqref{eq: relation alpha and n}. More details are provided in \cref{sec: CCM results}.

One key feature of their analysis is the (additive) Malthusian martingale $(W_{\ell}, \ell\ge 0)$ of the multiplicative cascade, which sums over loops at generation $\ell\ge 0$:
\begin{equation} \label{eq: intro def martingale W_ell}
W_{\ell} := \sum_{|u|=\ell} (Z_{\alpha}(u))^{\theta_{\alpha}}, \quad \ell\ge 0,
\end{equation}
with $\theta_\alpha$ as in \eqref{eq: def theta_alpha}. We summarise part of their results in the following statement. Set
\begin{equation} \label{eq: intro psi}
\psi_{\alpha,\theta}(q)
:=
\frac{1}{\Gamma(\alpha-1/2)} \int_0^{\infty} \mathrm{e}^{-q^{2/\theta}y-1/y} y^{-(\alpha+1/2)} \mathrm{d}y.
\end{equation}

\begin{thm} \label{thm: CCM convergence martingales}
\emph{(\cite[Theorem 9]{chen2020perimeter}.)}

\noindent The martingale $(W_{\ell})_{\ell\ge 0}$ of \eqref{eq: intro def martingale W_ell} converges in $L^1$ as $\ell\to\infty$ towards a positive limit $W_{\infty}$. Moreover, the law of $W_{\infty}$ is determined by its Laplace transform as follows:
\begin{itemize}
\item in the dilute case ($\alpha>3/2$),
\begin{equation} \label{eq: dilute W_infty}
\bb E[\mathrm{e}^{-q W_{\infty}} ] =
\psi_{\alpha,\theta_{\alpha}} ((\alpha-3/2)q),
\end{equation}
\item in the dense case ($\alpha<3/2$),
\begin{equation} \label{eq: dense W_infty}
\bb E[\mathrm{e}^{-q W_{\infty}} ] =
\psi_{\alpha,\theta_{\alpha}} \bigg(\frac{\Gamma(\alpha+1/2)}{\Gamma(3/2-\alpha)} q\bigg).
\end{equation}
\end{itemize}
\end{thm}
In particular, in the dilute case, $W_\infty$ follows the inverse-Gamma distribution with parameters $(\alpha-1/2, \alpha-3/2)$. In the same paper, the authors further make the conjecture that $W_\infty$ describes the scaling limit of the volume $V(\frak q)$ appropriately normalised by \eqref{eq: mean volume}. The arguments supporting this conjecture are the following. Consider the conditional expectation $V_{\ell}$ of $V(\frak q)$ given the information \emph{outside} all the loops at generation $\ell$. It is plain that $V_{\ell}$ is a uniformly integrable martingale that converges to $V(\frak q)$ as $\ell \to \infty$. On the other hand, one should expect $V_\ell$ to be \emph{close} to $\Lambda p^{\theta_\alpha} W_{\ell}$ as the half-perimeter $p$ of $\frak q$ goes to infinity. Indeed, one can see that the portion outside loops at generation $\ell$ is negligible.\footnote{{Heuristically, this is due to the fact that, in expectation (recall \eqref{eq: mean volume boltzmann} and \eqref{eq: mean volume}), the volume of the gasket is of order $p^{\alpha}$, while that of $\mathfrak{q}$ is of order $p^{\theta_{\alpha}}$, with $\theta_{\alpha}>\alpha$. A rigorous statement in this direction will be obtained in \cref{thm: gasket estimate}.}} 
Hence using the gasket decomposition, and Budd's asymptotics \eqref{eq: mean volume} on the mean volume, $V_\ell$ should be close to
\[
\widetilde{V}_\ell = \Lambda \sum_{|u|=\ell} (\chi^{(p)}(u))^{\theta_{\alpha}},
\]
which is the discrete analogue of $W_\ell$. The scaling of $(\chi^{(p)}(u), u\in \cal U)$ towards $(Z_{\alpha}(u),u\in\cal U)$ then provides the heuristics. In a nutshell (taking $\ell$ large), we end up with the conjecture that $V(\frak q)$ should be close to $\Lambda p^{\theta_\alpha} W_{\infty}$ as $p\to\infty$. These arguments can be turned into a proof provided some uniform integrability is known on the volumes. We emphasise that we do not know any direct way to establish the desired uniform integrability, except as a consequence of the results of the present paper. 

In the case $n=2$, however, $W_{\infty}=0$  a.s. and the volumes are not uniformly integrable. We will see that, in spite of the mean asymptotics \eqref{eq: mean volume}, the correct renormalization is actually $\frac{p^2}{\ln(p)}$ when $g = \frac{h}{2}$ and $\frac{p^2}{(\ln(p))^2}$ when $g < \frac{h}{2}$ (note that $\theta_{\alpha} = 2$ {when $n=2$}), and $W_\ell$ should be replaced with the so-called derivative martingale 

\begin{equation} \label{eq: intro def martingale D_ell}
D_{\ell} := - 2\sum_{|u|=\ell} (Z_{\alpha}(u))^{2}\ln (Z_{\alpha}(u)), \quad \ell\ge 0.
\end{equation}

\noindent This martingale is well-known in the context of branching Brownian motion \cite{lalleysellke, neveuderivative, kyprianouKPP} and branching random walks \cite{Liuderivative, Kyprianouderivative, biggins2004measure}.  
Although it is a signed martingale, its almost sure limit $D_\infty$ exists and is positive. Moreover $1/D_\infty$ is exponentially distributed with parameter $1$, see \cref{p:derivative martingale}. 

{
We provide more context on branching random walks, the \emph{boundary} case and the role of the derivative martingale in \cref{sec:bdry case derivative}.
}

\bigskip
\noindent \textbf{CLE on LQG.} A tantalising conjecture from physics, dating back to Nienhuis (see the survey \cite{kager2004guide}), is that after appropriate embedding, the aforementioned loop--$O(n)$ quadrangulations (with a boundary) are described in the scaling limit by a suitable $\gamma$--quantum disc, on top of which is drawn an independent $\text{CLE}_\kappa$, where the parameters are given in terms of \eqref{eq: relation alpha and n} by
\begin{equation} \label{eq: gamma and kappa}
\gamma = \sqrt{\min(\kappa, 16/\kappa)},
\quad \text{and} \quad
\frac{4}{\kappa} =\alpha -\frac12.
\end{equation}
Similar conjectures hold for many other important models of planar maps, such as the Fortuin--Kasteleyn model, and remain today the main challenge in random conformal geometry.

When $n\in (0,2)$, our main result on the volume of loop--$O(n)$ quadrangulations relates to this conjecture in that, as proved by Ang and Gwynne \cite{ang2021liouville}, in the dilute case, the limiting random variable $W_\infty$ describes the law of the \textbf{area} of a \emph{unit-boundary} $\gamma$--quantum disc (with $\gamma$ as in \eqref{eq: gamma and kappa}). As a consequence, our result can be rephrased as a scaling limit result for the volume of loop--$O(n)$ quadrangulations towards the area of its quantum analogue. As pointed out in \cite{chen2020perimeter}, the multiplicative cascade $Z_\alpha$ can also be constructed directly in the continuum by recording the perimeters of the nested loops in the $\text{CLE}_{\kappa}$ drawn on top of the unit-boundary quantum disc. This construction is essentially a consequence of \cite{miller2017cle}.

When $n=2$, the model is expected to converge to $\cle_4$ on critical Liouville quantum gravity ($\gamma=2, \kappa=4$). In this case it was first conjectured in \cite{aidekon2022growth} that the volume of loop--$O(2)$ quadrangulations scales to the duration of a Brownian half-plane excursion from $(0,0)$ to $(1,0)$, which is inverse-exponentially distributed. This conjecture is bolstered by the convergence of the derivative martingale related to the growth-fragmentation embedded in the half-plane excursion (see \cite[Section 5]{aidekon2022growth}), and the connection with the \emph{intrinsic areas} defined in \cite{bertoin2018martingales}. On the other hand, the critical mating of trees established by Aru, Holden, Powell and Sun \cite{aru2023brownian} shows that the duration of the Brownian half-plane excursion describes the law of the area of a critical quantum disc. Therefore, the conjecture of \cite{aidekon2022growth} can be translated into a convergence statement for the volume of loop--$O(2)$ quadrangulations towards their quantum area analogue. Our paper solves this conjecture and provides the explicit scaling in $\ln(p)/\overline{V}(p)$, revealing a logarithmic correction to the mean behaviour.

\subsection{Main result and outline}
\label{sec: main result + outline}

\noindent \textbf{Statement of the main result.} The main result of this paper is an explicit scaling limit for the volume of (non-generic critical) loop--$O(n)$ quadrangulations, proving a conjecture of Chen, Curien and Maillard \cite{chen2020perimeter} when $n\in(0, 2)$ and giving the analogous result in the boundary case $n=2$. 
The scaling limit is described in terms of the limit $W_\infty$ of the Malthusian martingale of the multiplicative cascades (see \cref{thm: CCM convergence martingales}) when $n\in(0, 2)$ and in terms of the limit $D_{\infty}$ of the derivative martingale when $n = 2$. The law of $W_\infty$, \textit{resp.}\ $D_\infty$, is explicitly given by \eqref{eq: dilute W_infty}--\eqref{eq: dense W_infty},  \textit{resp.}\ \cref{p:derivative martingale}. We denote by $V = V(\frak q)$ the volume (\textit{i.e.} the number of vertices) of the quadrangulation $\frak q$. Recall from \cref{eq: mean volume} and \cref{eq: mean volume 2} the notation and asymptotics of the expected volume $\overline{V}(p)$.

\begin{Thmx}\label{thm: main}
The following convergence in distribution holds for the volume of rigid loop--$O(n)$ quadrangulations: as $p\to\infty$, when $n \in(0, 2)$, 
\[
\frac{1}{\overline{V}(p)} V\overset{(\mathrm{d})}{\longrightarrow} W_{\infty},
\]
when $n = 2$, 
\[
\frac{\ln p}{\overline{V}(p)} V\overset{(\mathrm{d})}{\longrightarrow} D_{\infty}.
\]
\end{Thmx}

The proof uses a two-step procedure. First, we fix some large threshold  $M>0$ and show that for $M$ large enough, one can neglect 
{as $p\to\infty$ the contribution of vertices that are not included in any loop of perimeter smaller than $2M$.}
By the gasket decomposition, the remaining portion of the map then consists in many independent loop-decorated maps with perimeter less than $2M$. In the second step, we establish a classification {in \cref{def: bad vertices}} of this remaining portion of the map into good or bad regions. 
We prove that one can rule out the contribution of bad regions to the volume, and estimate the size of the good region, where the volume is square integrable. Our estimates mainly rely on the analysis on a Markov chain $\tt S$ {introduced in \cref{sec: markov chain}}, which records the half-perimeters of the nested loops around a typical vertex of the map.

{The classification alluded to above is quite involved, but at a very high level, it can be understood as follows. We will say that a loop (or the region inside it) is good if:
\begin{itemize}
    \item[(a)] The sequence of perimeters of the nested loops around it (in $(\mathfrak{q},\boldsymbol{\ell})$) is tamed;
    \item[(b)] These nested loops do not carry too many loops, whose perimeters conspire to inflate their volume.
\end{itemize}
We will express this in terms of the Chen--Curien--Maillard cascade $(\chi^{(p)}(u), u\in \cal U)$. It will result in a set of constraints on branches of the tree $\cal U$ up to good labels $v$, that roughly ensure, respectively, that:
\begin{itemize}
    \item[(a)] The values of $\chi^{(p)}$ along these branches stay below a barrier;
    \item[(b)] At each generation along these branches, the offspring particles do not carry ``unusually large'' values of $\chi^{(p)}$.\footnote{{To be more precise, we stress that what will matter is not the value of the (half-)perimeters $\chi^{(p)}$ \textit{per se}, but rather the ``expected volume'' they should carry. It would be more accurate to say that this expected volume should be small compared to that of the parent loop.}}
\end{itemize}
The classification induced by this set of constraints will be estimated using branching arguments.}
\added{
{In that respect}, the structure of the proof bears close connections to that of martingale convergence for branching random walks. The key feature in our approach is the \emph{branching Markov chain} $(\chi^{(p)}(u), u\in\cal U)$ recording the half-perimeters of the loops. In order to prove the convergence of the renormalised volume, we leverage information on the \emph{spine} $\tt S$ to shed light on the behaviour of typical loops in the branching Markov chain. In particular, we establish moment estimates on $\tt S$, and control the offspring of the loops. This idea is similar to the setting of branching random walks (see for instance \cite{lyons1997simple,biggins2004measure,shi2016branching}).
}

\bigskip
\noindent \textbf{Organisation of the paper.} The paper is organised as follows. We start \cref{sec: prelim} with some preliminaries, recalling the setup of Chen, Curien and Maillard \cite{chen2020perimeter} and explaining the tree encoding of Boltzmann planar maps \textit{via} the Bouttier--Di Francesco--Guitter and Janson--Stefánsson transformations. Then we introduce the perimeter cascade of Chen, Curien and Maillard and its scaling limit.

In \cref{sec: discrete Biggins transform}, we discuss the discrete Biggins transform, for which we provide uniform tail estimates. This will enable us to control the offspring of typical loops in the nested cascade, {so as to estimate the constraint (b) outlined above}. We will also prove its convergence to the Biggins transform of the continuous cascade in an $L^p$ sense. 

\cref{sec: markov chain estimates} is the main technical part of the paper. We introduce the Markov chain $\tt S$, making explicit the spine structure of the discrete multiplicative cascade. 
{As already mentioned, this Markov chain describes the sequence of half-perimeters of nested loops around a \emph{typical} vertex of the quadrangulation. This will allow us to translate a few estimates from the whole perimeter cascade $(\chi^{(p)}(u),u\in\cal U)$ to the single trajectory $\tt S$ (such a procedure is referred to as \emph{many-to-one} in the branching literature). The gasket threshold $M$ of the discussion following \cref{thm: main} and the constraint (a) that is sketched above suggest that we will need to understand the behaviour of the perimeter cascade confined within two barriers. This will thus translate into estimating}
the hitting times and Green function of $\tt S$, in both cases $n\in (0,2)$ and $n=2$. 
{In the latter case, we use a new coupling argument between the discrete and continuous cascade which is of independent interest. We will see that, because the cascade is critical in this case, logarithmic corrections emerge as a cost for introducing the upper barrier.}

In \cref{sec: classification}, we present our {rigorous} classification of the map into \emph{good} or \emph{bad} regions. 
On the one hand, we shall see that the estimates in \cref{sec: markov chain estimates} are tailored to rule out the contribution of the bad regions to the volume. On the other hand, we prove that the remaining good volume is square integrable and establish a second moment estimate for this good region. 

We conclude the proof of \cref{thm: main} in \cref{sec: proof main result} by tuning all the parameters of our estimates. This allows to neglect the bad regions, and to approximate the rescaled volume by the additive martingale or the derivative martingale of the \emph{continuous} cascade.
{
We stress that \cref{sec: proof main result} also presents three important \textbf{flowcharts} (Figures~\ref{fig: proof diagram 1}, \ref{fig: proof diagram 2} and \ref{fig: proof diagram 3}), which should help the reader understand how the estimates of Sections~\ref{sec: markov chain estimates} and \ref{sec: classification} come together.
}

Finally, the Appendix is devoted to the rigorous definition and basic properties of the rigid loop--$O(2)$ model on quadrangulations.
We first present in \cref{s:partition function O(2)} standard general results in the $O(n)$ case for $n\in (0,2)$. We then use these results in \cref{s:construction O(2)} to construct the rigid loop--$O(2)$ model and derive some of its basic properties. In \cref{s:tail distribution S}, we prove a key estimate for the Markov chain $\tt S$ in the case $n = 2$.

%
%
\section{Preliminaries}
\label{sec: prelim}

\subsection{The gasket decomposition}
\label{sec: gasket decomposition}
One fruitful approach towards understanding the geometry of $O(n)$--decorated planar maps relies on the \textbf{gasket decomposition} that was formalised by Borot, Bouttier and Guitter in \cite{borot2011recursive} (see also \cite[Section 8]{legall2011scaling}). This decomposition reveals a spatial Markov property that has been used to provide a lot of geometric information on the planar maps coupled to $O(n)$ models as well as their gaskets, see \cite{budd2017geometry,budd2018infinite,bertoin2018martingales,budd2018peeling,chen2020perimeter,albenque2022geometric} to name but a few. We follow the exposition given in \cite{chen2020perimeter}, which extends to the case when $n=2$.

Let $(\frak q, \bf \ell)$ be a fixed loop-decorated quadrangulation, and let $l\ge 1$ be the number of \emph{outermost} loops in $\bf \ell$, that is, loops that can be reached from the boundary of $\frak q$ without crossing another loop. Erasing these loops together with the edges they cross, one ends up disconnecting the map into $l+1$ connected components (see \cref{fig: gasket decomposition}):
\begin{itemize}
\item The connected component containing the external face is called the \textbf{gasket}. The gasket is a (rooted bipartite) planar map whose internal faces are either quadrangles from the original map $\frak q$, or the $l$ \textbf{holes} left by the removal of outermost loops.  The perimeter of the gasket is the same as that of $\frak q$. Moreover, one may label the holes $\frak h_1,\ldots, \frak h_l$ of $\frak g$ in some deterministic way that we fix from now on.
\item The other $l$ components consist of the loop-decorated quadrangulations inside the holes (rooted in some deterministic fashion), with perimeters given by the degrees of the holes.
\end{itemize}

In other words, the loop-decorated quadrangulation $(\frak q, \bf\ell)$ decomposes in the following way: given its gasket $\frak g$, one may recover $(\frak q, \bf\ell)$ by gluing into each face of $\frak g$ of degree $2k$, $k\ge 1$, a loop-decorated quadrangulation with perimeter $2k$, surrounded by a necklace of $2k$ quadrangles crossed by a loop -- with the caveat that for $k=2$ one might as well glue a plain quadrangle (with no loop).

By definition of the model (see \eqref{eq: partition function O(n)}), it is easily seen that the gasket $\frak g$ of a loop--$O(n)$ decorated quadrangulation $(\frak q, \bf \ell)$ with $(n,g,h)$--weights is Boltzmann distributed \cite{borot2011recursive,legall2011scaling}. More precisely, under \added{$\bb P^{(p)} = \bb P^{(p)}_{(n;g,h)}$}, the probability for the gasket to be a given map $\frak m$ with perimeter $2p$ is proportional to the weight
\begin{equation}\label{eq: Boltzmann weights}
w_{\bf{\hat{g}}} :=
\prod_{\frak f \in \textsf{Faces}(\frak m) \setminus \frak f_r} \hat{g}_{\mathsf{deg}(\frak f)/2},
\end{equation}
where the weight sequence $\bf{\hat{g}} = (\hat{g}_k,k\ge 1)$ is given in terms of the $O(n)$ partition function \eqref{eq: partition function O(n)} as
\begin{equation} \label{eq: fixed point equation}
\forall k\ge 1, \quad  \hat{g}_k := g\delta_{k,2} + n h^{2k} F_k(n;g,h).
\end{equation}
Finally, it is known that the weights in \eqref{eq: Boltzmann weights} can be normalised to form a probability measure, by equivalence of admissibility of sequences related through \eqref{eq: fixed point equation} (see \cite{budd2018peeling}). We comment on that point when $n=2$ in \cref{s:partition function O(2)}.

\begin{figure}[h]
\begin{center}
\includegraphics[scale=0.8]{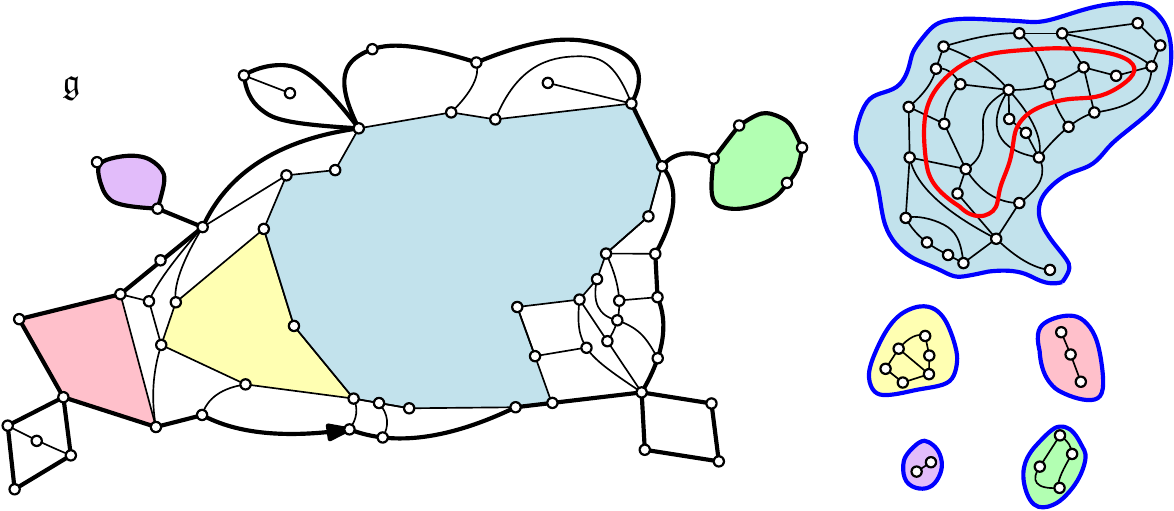}
\end{center}
\caption{The gasket decomposition of the loop-decorated quadrangulation $(\frak q,\bf\ell)$ in \cref{fig: loop-decorated quadrangulation}. On the left is the gasket $\frak g$, whereas on the right are the other connected components. All of these come with a deterministic choice of root edge, that we did not represent. \added{To be precise, we stress that the gluing operation also requires \emph{rooted} holes (this is discussed thoroughly in \cite[Section 2.2]{borot2011recursive}).}}
\label{fig: gasket decomposition}
\end{figure}

This property of the gasket, together with the following spatial Markov property (borrowed from \cite[Lemma 1]{budd2018peeling}), is the cornerstone of the gasket decomposition approach, relating properties of the loop-decorated maps to that of their gaskets. It is also the starting point of the multiplicative cascades of Chen, Curien and Maillard \cite{chen2020perimeter}.

\begin{prop} \label{prop: spatial markov gasket}
\emph{(Spatial Markov property of the gasket decompositon.)}

\noindent Under \added{$\bb P^{(p)}$}, let $(\frak q, \bf \ell)$ be a loop--$O(n)$ decorated quadrangulation with perimeter $2p$. Conditionally on the gasket $\frak g$ and its holes $\frak h_1,\ldots,\frak h_l$ of degrees $2h_1,\ldots,2h_l$, the loop-decorated maps $(\frak q_1, \bf \ell_1),\ldots, (\frak q_l, \bf \ell_l)$ filling in the holes $\frak h_1,\ldots, \frak h_l$ are independent with respective laws given by \added{$\bb P^{(h_1)}, \ldots, \bb P^{(h_l)}$}.
\end{prop}

\begin{rem} We shall often need a \emph{stopping line}\footnote{{Stopping lines are special random collections of particles (in this case, loops) that play the role of stopping times for branching processes. An example of stopping line that is relevant to the present paper is the collection of loops whose perimeters have dropped below $2M$ for the first time (\textit{i.e.}\ loops with perimeter smaller than $2M$ which are not strictly contained in a loop with perimeter smaller than $2M$). We refer to \cite{jagers1989general,chauvin1991prod} for general background on the notion.}} 
version of the above spatial Markov property.
We feel free to simply mention it as we go without providing the details.
\end{rem}

\subsection{The Bouttier--Di Francesco--Guitter bijection and the Janson--Stefánsson trick}
\label{sec: BDG+JS}

\noindent \textbf{The Bouttier--Di Francesco--Guitter (BDG) bijection.} The BDG bijection \cite{bouttier2004planar} gives a way to encode (pointed) bipartite planar maps \emph{via} trees that is particularly suited to Boltzmann planar maps \cite{marckert2005invariance,bettinelli2017compact}. It will be convenient to work with the slight modification of \cite{chen2020perimeter}. Let $(\frak m,\rho)$ be a \textbf{pointed} bipartite planar map of perimeter $2p$, that is a planar map together with a distinguished vertex $\rho$. The following algorithm uses a four-step procedure to build a forest out of $(\frak m,\rho)$ -- see \cref{fig: BDG}:

\bigskip
\begin{minipage}{0.91\textwidth}
\noindent \textsf{Step $1$.} \emph{Draw a dual vertex inside each face of $\frak m$, including the external face. Dual vertices are coloured black, while primal vertices are white. Label the white vertices with their distances with respect to $\rho$. As $\frak m$ is bipartite, the labels of any two adjacent vertices differ by exactly $1$.}

\smallskip
\noindent \textsf{Step $2$.} \emph{Connect a white vertex to a black vertex (corresponding to some face) if the next clockwise white vertex around that face has a smaller label.}

\smallskip
\noindent \textsf{Step $3$.} \emph{Remove the edges of $\frak m$ and the marked vertex $\rho$. This produces a tree \cite{bouttier2004planar}.}

\smallskip
\noindent \textsf{Step $4$.} \emph{Remove the external black vertex $v_{\textnormal{ext}}$ and its neighbouring edges. This results in $p$ trees, rooted at the neighbours of $v_{\textnormal{ext}}$. We choose one of them, $\frak t_1$, to be the first one uniformly at random, and shift all the labels in the trees so that the root of $\frak t_1$ has label $0$. The output is actually a forest of \emph{mobiles}.}
\end{minipage}

\bigskip
\noindent We refer to \cite{bouttier2004planar,bettinelli2017compact} for more details on mobiles and for a description of the reverse construction. From now on we forget the labels in the trees.

\begin{figure}[h]
\bigskip
\begin{center}
\includegraphics[scale=0.78]{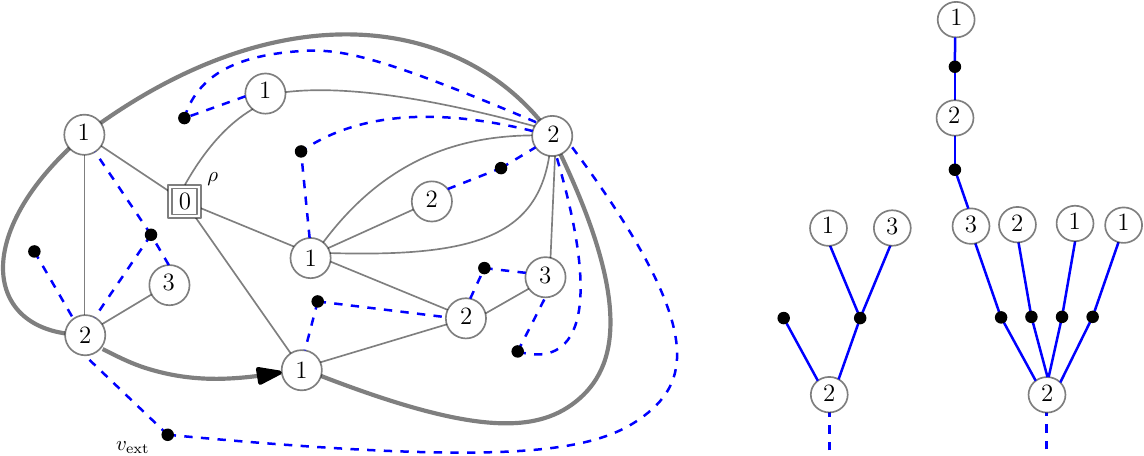}
\end{center}
\caption{The Bouttier--Di Francesco--Guitter bijection. Left: the planar map $\frak m$ is represented in blue, with its marked (square) vertex $\rho$. We draw additional (dashed) edges between a vertex and a face according to Step 2. Right: The forest obtained by disconnecting the external vertex $v_{\textnormal{ext}}$ -- the number of trees corresponds to the half-perimeter of $\frak m$.}
\label{fig: BDG}
\end{figure}

The BDG bijection works particularly well with Boltzmann planar maps \cite{marckert2005invariance}. It is possible to define \emph{pointed} Boltzmann planar maps $\frak m_{\bullet}:=(\frak m,\rho)$ with perimeter $2p$, where $\frak m$ is Boltzmann distributed, and conditionally on $\frak m$, $\rho$ is a uniformly chosen vertex. In other words, the pointed measure is given by the same weights as in \eqref{eq: Boltzmann weights}, but different normalising constant taking care of the uniformly chosen vertex (one can prove that such a normalisation is possible, see \cite[Corollary 23]{curien2019peeling}). Let $\frak m_{\bullet}$ be a pointed Boltzmann planar map with perimeter $2p$, and $\mathscr F^{(p)}_{\text{BDG}}$ the forest of mobiles constructed by running the BDG algorithm. Marckert and Miermont \cite[Proposition 7]{marckert2005invariance} proved that $\mathscr F^{(p)}_{\text{BDG}}$ is a two--type Galton--Watson forest with reproduction law given explicitly in terms of $\bf{\hat{g}}$.

\medskip
\noindent \textbf{The Janson--Stefánsson (JS) trick.} Using the Janson--Stefánsson trick \cite{janson2015scaling}, one can transform the two-type Galton--Watson trees output of the BDG bijection into Galton--Watson trees (with no types). The transformation keeps the same vertices but changes the set of edges, with the interesting feature that white vertices are mapped to leaves, whereas black vertices are mapped to internal vertices. It goes recursively as follows (see \cref{fig: JS}):

\bigskip
\begin{minipage}{0.91\textwidth}
\emph{Suppose a plane bipartite tree $\frak t$ is given (the root $r$ is taken to be white). Consider the children of $r$, say $r_1,\ldots, r_J$ \added{(labelled from left to right)}. Set $r_1$ as the new root. Draw an edge between $r_j$ and $r_{j+1}$ $(1\le j\le J-1)$, and finally an edge between $r_J$ and $r$. Then carry on for the next generations: if $w\ne r$ is a white vertex, denote its offspring by $b_1,\ldots,b_K$ and its parent by $b_0$. Connect $b_0$ to $b_1$, then $b_k$ to $b_{k+1}$ $(1\le k\le K-1)$ and finally $b_K$ to $w$ (if $w$ has no children, we just connect $b_0$ to $w$).
}
\end{minipage}

\begin{figure}[h]
\bigskip
\begin{center}
\includegraphics[scale=0.85]{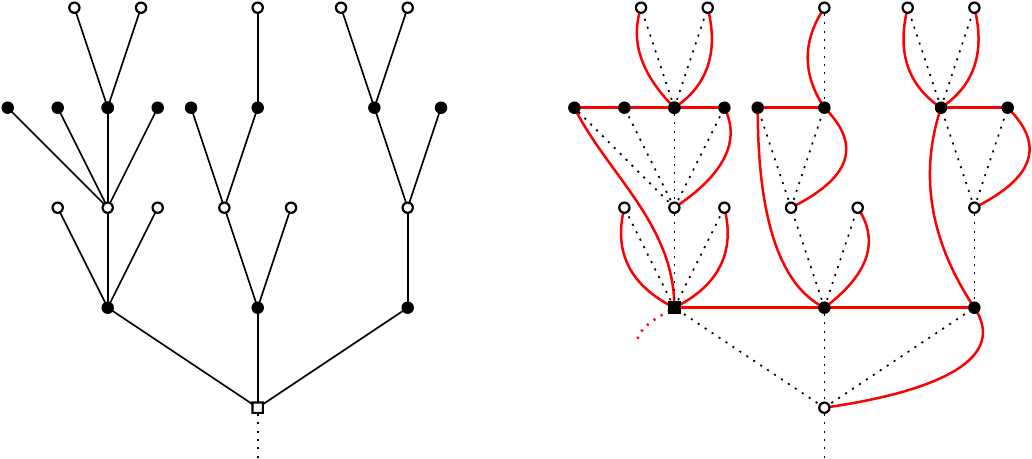}
\end{center}
\caption{The Janson--Stefánsson trick. On the left is the original tree $\frak t$, and on the right the new tree in red (with the old tree $\frak t$ in dotted line). The root vertex is represented by a square. Observe that the white vertices are the leaves of the new tree.}
\label{fig: JS}
\end{figure}

\bigskip
\noindent Then Janson and Stefánsson \cite{janson2015scaling} derived the law of the image $\mathscr F^{(p)}_{\textnormal{JS}}$ under this mapping of the above forest $\mathscr F^{(p)}_{\text{BDG}}$ (obtained after running the BDG algorithm on a pointed Boltzmann map). Explicitly, $\mathscr F^{(p)}_{\textnormal{JS}}$ is a forest of $p$ i.i.d. Galton-Watson trees with offspring distribution
\begin{equation}\label{def:muJS}
\mu_{\textnormal{JS}}(0) =Z_{\bf{\hat{g}}}^{-1}, \quad \text{and} \quad \mu_{\textnormal{JS}}(k) = \dbinom{2k-1}{k-1} \hat{g}_k Z_{\bf{\hat{g}}}^{k-1}, \quad k\ge 1,
\end{equation}
where $Z_{\bf{\hat{g}}}$ is the smallest positive root of the equation
\begin{equation}\label{eq:admissible}
1 + \sum_{k\ge 1} \dbinom{2k-1}{k-1} \hat{g}_k z^{k} = z.
\end{equation}

\noindent See also \cite{curien2015percolation,chen2020perimeter}. {For Boltzmann maps corresponding to the gaskets of loop--$O(n)$ quadrangulations in the non-generic critical regime,} 
it is known that $Z_{\bf{\hat{g}}} = \frac{1}{4h}$ (see \cite[Section 4.1]{borot2011recursive} and \cite[Section 2.3]{curien2019peeling}). In this context, the non-generic criticality condition can be recast in terms of $\mu_{\textnormal{JS}}$. 
More precisely, fix a set of parameters $(n; g, h)$ in the non-generic critical region. Then one of the following two cases holds:
\begin{enumerate}
    \myitem[(A)] \label{caseA}
    {The set of parameters} $(n; g, h)$ satisfies either of
    $
    \begin{cases}
        \eqref{non-generic line} \text{ and } n\in (0,2);\\
        \eqref{n=2 line} \text{, } g = \frac{h}{2} \text{ and } n=2; 
    \end{cases}
    $
    \myitem[(B)] \label{caseB}
    {The set of parameters}
    $(n;g,h)$ satisfies \eqref{n=2 line}, $g<\frac{h}{2}$ and $n=2$.
\end{enumerate}
Recall the relation \eqref{eq: relation alpha and n} between $\alpha$ and $n$ when $n\in(0, 2)$ and $\alpha = \frac32$ when $n=2$, and let $\bf{\hat{g}}$ and $\mu_{\textnormal{JS}}$ as in (\ref{eq: fixed point equation}) and (\ref{def:muJS}). Then the offspring distribution $\mu_{\textnormal{JS}}$ has mean one and satisfies
\begin{equation} \label{def: non generic critical mu_JS}
\mu_{\textnormal{JS}}(k) \sim C_{\textnormal{JS}} 
\begin{cases}
    k^{-\alpha-1}, \quad & \text{ in {Case}~\ref{caseA}, }\\
    k^{-5/2}\ln (k), \quad & \text{ in {Case}~\ref{caseB}, }
\end{cases} \quad \text{as } k\to \infty,
\end{equation}
for some constant $C_{\text{JS}}>0$.
{In particular, for Boltzmann maps corresponding to the gaskets of loop--$O(n)$ quadrangulations in the non-generic critical regime, the Galton--Watson trees appearing in $\mathscr F^{(p)}_{\textnormal{JS}}$ are critical. }

{
Since the estimates are different in {Case}~\ref{caseA} and {Case}~\ref{caseB}, it will often be important to split the proofs according to these two cases. We therefore emphasise that these cases are globally defined throughout the paper.
}

\medskip
\noindent \textbf{Random walk encoding.} Going one step further in the above chain of transformations, we may now encode the Galton--Watson trees {using Łukasiewicz paths \cite[Section 1.1]{le2005randomtrees}}. 
Under $\bb P$, let $(X_i)_{i\ge 1}$ be an i.i.d. sequence of random variables in $\{-1,0,1,\ldots\}$, with common distribution
\begin{equation} \label{eq: RW Janson--Stefansson tree}
\bb P (X_1 = k-1) = \mu_{\textnormal{JS}}(k), \quad k\ge 0.
\end{equation}
Let $S_n := X_1 + \cdots + X_n$, $n\ge 1$, the associated random walk. Define also, for $p>0$,
\begin{equation} \label{eq: def Tp and Lp}
T_p := \inf\{n\ge 1, \; S_n=-p\},
\quad \text{and} \quad
L_p := \sum_{i=1}^{T_p} \mathds{1}_{\{X_i=-1\}}.
\end{equation}
Notice by \cref{def: non generic critical mu_JS} that, in the non-generic critical case, the random walk lies in the domain of attraction of a spectrally positive $\alpha$--stable process.

We summarise the previous chain of transformations, from pointed Boltzmann maps to random walks, by the following key formula, which is the starting point of the analysis for the multiplicative cascades of Chen, Curien and Maillard (see \cite[Equations (8)--(9)]{chen2020perimeter}). We use the same notation as \cite{chen2020perimeter} for continuity of literature. Let $\Deg^{\downarrow}_f (\frak m)$ be the eventually-zero sequence of degrees of faces in a bipartite planar map $\frak m$, ranked in descending order (with some deterministic way to break ties). Observe by the gasket decomposition (see \eqref{eq: Boltzmann weights} and \eqref{eq: fixed point equation}) that if $\frak g$ denotes the gasket of a loop--$O(n)$ quadrangulation $(\frak q,\bf\ell)$, $\Deg^{\downarrow}_f (\frak g)$ records the perimeters of the outermost loops in $(\frak q, \bf \ell)$, except for some entries (equal to $4$) corresponding to regular quadrangles not crossed by a loop. Likewise, define $\Deg^{\downarrow}_{\bullet} (\mathscr{F})$ to be the (eventually-zero and descending) sequence of degrees of black vertices in a forest of bipartite trees, and $\Deg^{\downarrow}_{\text{out}} (\mathscr{F})$ the sequence of outdegrees (\textit{i.e.} number of children) in a forest of trees. Let $\circ\big(\mathscr{F}^{(p)}_{\text{BDG}}\big)$ be the set of white vertices in $\mathscr{F}^{(p)}_{\text{BDG}}$. Finally, let $\bf X^{(p)}$ be the (eventually-zero, descending) sequence made out of $(X_i+1)_{1\le i\le T_p}$. Under $\bb P$, let $\frak B_{\bf{\hat{g}}}^{(p)}$ and $\frak B_{\bf{\hat{g}}}^{(p), \bullet}$ denote respectively a Boltzmann planar map and a pointed Boltzmann planar map with weights $\bf{\hat{g}}$ and fixed perimeter $2p$.
For all non-negative measurable $\varphi : \bb N^{\bb N^*} \to \bb R$, the previous chain of transformations translates to
\begin{align}
\bb E \Big[\varphi\Big(\frac12 \Deg^{\downarrow}_f \Big(\frak B_{\bf{\hat{g}}}^{(p)}\Big)\Big) \Big]
&=
\frac{1}{\bb E \Big[1/ \# \textsf{Vertex}\Big(\frak B_{\bf{\hat{g}}}^{(p), \bullet}\Big)\Big]} \cdot \bb E \left[\frac{\varphi\Big(\frac12 \Deg^{\downarrow}_f \Big(\frak B_{\bf{\hat{g}}}^{(p), \bullet}\Big)\Big)}{\# \textsf{Vertex}\Big(\frak B_{\bf{\hat{g}}}^{(p), \bullet}\Big)} \right] \notag  \\
&=
\frac{1}{\bb E \Big[1/ \big(1+\# \circ\big(\mathscr{F}^{(p)}_{\text{BDG}}\big)\big)\Big]} \cdot  \bb E \left[\frac{\varphi\Big( \Deg^{\downarrow}_{\bullet} \Big(\mathscr{F}^{(p)}_{\text{BDG}}\Big)\Big)}{ 1+ \# \circ\big(\mathscr{F}^{(p)}_{\text{BDG}}\big)} \right] \quad \text{(BDG bijection)} \notag \\
&=
\frac{1}{\bb E \Big[1/ \big(1+\# \textsf{Leaf}\big(\mathscr{F}^{(p)}_{\text{JS}}\big)\big)\Big]} \cdot  \bb E \left[\frac{\varphi\Big( \Deg^{\downarrow}_{\text{out}} \Big(\mathscr{F}^{(p)}_{\text{JS}}\Big)\Big)}{1+ \# \textsf{Leaf}\big(\mathscr{F}^{(p)}_{\text{JS}}\big)} \right] \qquad \text{(JS trick)} \notag \\
&=
\frac{1}{\bb E [1/(1+L_p)]} \cdot \bb E\bigg[ \frac{\varphi(\bf X^{(p)})}{1+L_p} \bigg] \quad \hspace{2cm} \text{(Łukasiewicz path encoding)}. \label{eq: key formula maps to RW}
\end{align}
We shall use \eqref{eq: key formula maps to RW} several times in our estimates of \cref{sec: discrete Biggins transform} to grasp information on the planar maps from random walk arguments.

\subsection{The multiplicative cascade of Chen, Curien and Maillard}
\label{sec: CCM results}
We describe the multiplicative cascade setup of Chen, Curien and Maillard \cite{chen2020perimeter} that is relevant to our work, and present some of the results that we will use for the derivation of \cref{thm: main}. Although they restrict to $n\in(0,2)$, the construction makes sense for all $n\in(0,2]$ with our definitions. We comment on the extension to the case $n=2$ as we go along.

\medskip
\noindent\textbf{Nested loops and perimeter cascade.} We start by encoding the nesting structure of the $O(n)$ loops using the Ulam tree
\[
\cal U := \bigcup_{i\ge 0} (\bb N^*)^i,
\]
where $\bb N =\{0,1,2,\ldots\}$, $\bb N^*= \bb N\setminus \{ 0\}$, and by definition, $(\bb N^*)^0$ contains a single label denoted $\varnothing$. As usual, for $u,v\in \cal U$, we denote by $uv$ the concatenation of $u$ and $v$ (with $\varnothing u = u\varnothing =u$), and by $|u|$ the length or \emph{generation} of $u$ (with $|\varnothing|=0$). For $0\leq i\leq |u|$, we denote by $u_i$ the ancestor of $u$ at generation $i$.

Under \added{$\bb P^{(p)}$}, let $(\frak q, \bf \ell)$ be a loop--$O(n)$ quadrangulation of perimeter $2p$. Then we give labels in $\cal U$ to the loops in $\bf \ell$ in the following way. Add an imaginary loop around the boundary of $\frak q$, and label it with $\varnothing$. Then the outermost loops are considered to be the children of $\varnothing$. More precisely, we rank them in descending order of their perimeters, labelling them with $1, 2, \ldots$ (breaking any potential tie in a deterministic way). Finally, we carry on in the same way for later generations, \textit{i.e.} children of the loop labelled $u\in\cal U$ are the $ui$, $i\in \bb N^*$, ranked by decreasing perimeter. For any label $u\in\cal U$, we define $\chi^{(p)}(u)\in \bb N$ to be the half-perimeter of the loop with label $u$, with the convention that $\chi^{(p)}(u)=0$ if there is no such loop. We call $(\chi^{(p)}(u), u\in \cal U)$ the \textbf{discrete perimeter cascade}. The gasket decomposition (\cref{prop: spatial markov gasket}) ensures that the collection $(\chi^{(p)}(u), u\in \cal U)$ is a \emph{branching Markov chain} under $\bb P^{(p)}$. We shall denote by $(\frak q(u), \bf \ell(u))$ the loop--decorated quadrangulation inside the loop labelled by $u$ (if any), that is the connected component obtained inside the loop $u$ after deleting it together with the edges it crosses. If no such loop exists, we take $(\frak q(u), \bf \ell(u))$ to be the empty map. Note that the half-perimeter of $(\frak q(u), \bf \ell(u))$ is $\chi^{(p)}(u)$. We also write $V(u)$ for the volume of $(\frak q(u), \bf \ell(u))$. As a consequence of the spatial Markov property of the gasket decomposition (\cref{prop: spatial markov gasket}), we have in particular the identity, for all $\ell \ge 1$:
\begin{equation} \label{eq: gasket decomposition summing outermost loops}
\bb E^{(p)} \bigg[ \sum_{|u|=\ell} V(u) \bigg]
=
\bb E^{(p)} \bigg[ \sum_{|u|=\ell} \overline{V}(\chi^{(p)}(u)) \bigg].
\end{equation}

\medskip
\noindent \textbf{Convergence towards the multiplicative cascade.}
For $n\in(0,2)$, Chen, Curien and Maillard \cite{chen2020perimeter} proved the convergence of the discrete perimeter cascade $(\chi^{(p)}(u), u\in \cal U)$ towards the following continuous \textbf{multiplicative cascade}. Let $\zeta$ be a spectrally positive $\alpha$--stable Lévy process, with $\alpha\in(1,2)$ as in \eqref{eq: relation alpha and n}, so that for all $s\ge 0$, {$\bb E[\mathrm{e}^{-q\zeta_s}] = \exp(c_\alpha s q^{\alpha})$} for some constant $c_\alpha>0$ which is irrelevant to the construction. Since $\zeta$ does not drift to infinity, we can define
\begin{equation} \label{eq: def tau}
\tau := \inf\{s>0, \; \zeta_s=-1\} <\infty, \quad \text{a.s.}
\end{equation}
Finally, let $\nu_{\alpha}$ be the probability measure on $(\bb R_+)^{\bb N^*}$ defined by
\begin{equation} \label{eq: def nu_alpha}
\int_{(\bb R_+)^{\bb N^*}} \nu_{\alpha} (\mathrm{d} \mathbf{x}) F(\mathbf{x}) = \frac{\bb E[\frac{1}{\tau} F((\bf{\Delta \zeta})_{\tau}^{\downarrow})]}{\bb E[\frac{1}{\tau}]},
\end{equation}
where $(\bf{\Delta \zeta})_{\tau}^{\downarrow}$ denotes the collection of jumps made by $\zeta$ up to time $\tau$, ranked in descending order. In order to define the multiplicative cascade, we take an i.i.d. collection $((\xi^{(u)}_i)_{i\ge 1}, u\in \cal U)$ with common law $\nu_{\alpha}$. The multiplicative cascade with offspring distribution $\nu_{\alpha}$ is then the collection $(Z_{\alpha}(u), u\in \cal U)$ defined recursively by $Z_{\alpha}(\varnothing)=1$ and for all $u\in\cal U$ and $i\in \bb N^*$, $Z_{\alpha}(ui)= Z_{\alpha}(u)\cdot  \xi_i^{(u)}$. The main result of Chen, Curien and Maillard in \cite{chen2020perimeter} concerns the convergence of the discrete cascade towards the limiting multiplicative cascade as the perimeter goes to infinity. 

\begin{thm}\emph{(\cite[Theorem 1]{chen2020perimeter}.)}\label{thm: CCM convergence cascade}

\noindent Suppose $n\in(0,2)$. The following convergence in distribution holds in $\ell^\infty(\cal U)$:
\[
\frac{1}{p} (\chi^{(p)}(u), u\in \cal U) \xrightarrow[p\to\infty]{(\mathrm{d})} (Z_{\alpha}(u), u\in \cal U).
\]
\end{thm}

\noindent One should not be surprised by the appearance of the $\alpha$--stable process in the limiting cascade. Indeed, the loops correspond to (large) faces of the gasket, whose perimeters are encoded by the increments of the random walk $S$ after performing the Janson--Stefánsson trick (see equation \eqref{eq: key formula maps to RW}). Such a random walk is in the domain of attraction of an $\alpha$--stable process (see  \eqref{eq: RW Janson--Stefansson tree}) -- from now on, we can assume that $c_\alpha$ has been chosen so that $S_n$ scales to $\zeta$. Furthermore, again in light of \eqref{eq: key formula maps to RW}, the $\tau^{-1}$ bias in \eqref{eq: def nu_alpha} can be understood as a scaling limit of the $L_p$ term.

\bigskip
\noindent \textbf{Additive martingales and the Biggins transform.}
We gather here for future purposes some additional results that were obtained in \cite{chen2020perimeter}. Notice that $(-\ln(Z_\alpha(u)), u\in \cal U)$ is a branching random walk.
One key feature of the discrete and continuous cascades is therefore the so-called \textbf{Biggins transform}, which captures asymptotic information about the cascades. These are defined respectively \textit{via}
\begin{equation} \label{eq: biggins transform}
\phi^{(p)}(\theta) = \mathbb{E}^{(p)}\bigg[\sum_{\abs{u}=1}\bigg(\frac{\chi^{(p)}(u)}{p}\bigg)^{\theta}\bigg],
\quad
\text{and}
\quad
\phi_{\alpha}(\theta) = \mathbb{E}\bigg[\sum_{\abs{u}=1}(Z_{\alpha}(u))^{\theta}\bigg].
\end{equation}
As it turns out, the limiting Biggins transform $\phi_{\alpha}$ can be calculated explicitly.
\begin{prop} \label{prop: CCM expression of biggins}
 \emph{(\cite[Equation (17)]{chen2020perimeter}.)}

\noindent 
For all $\alpha\in(1, 2)$ and $\theta\in \mathbb{R}$, we have
\[
\phi_{\alpha}(\theta)
=
\begin{cases}
\frac{\sin(\pi(2-\alpha))}{\sin(\pi(\theta-\alpha))}, & \text{if } \theta\in(\alpha,\alpha+1), \\
+\infty, & \text{otherwise}.
\end{cases}
\]
\end{prop}
\noindent It is standard that, by the branching property, $\phi_\alpha$ paves the way for \emph{additive martingales} (see \cite[Section 4.1]{chen2020perimeter}). Of special importance is the so-called \textbf{Malthusian martingale}, which corresponds to the minimal solution $\theta_{\alpha}\in(\alpha,\alpha+1)$ of $\phi_{\alpha}(\theta)=1$. From \cref{prop: CCM expression of biggins}, it is easily seen that $\theta_{\alpha} = \min(2, 2\alpha-1)$, leading to the martingale
\begin{equation} \label{eq: def martingale W_ell}
W_{\ell} := \sum_{|u|=\ell} (Z_{\alpha}(u))^{\theta_{\alpha}}, \quad \ell\ge 0.
\end{equation}
Observe that the expression of $\theta_\alpha$ matches that of the growth exponent for the mean volume in Budd's asymptotics \eqref{eq: mean volume}.
The value of $\theta_{\alpha}$ is such that $\theta_{\alpha}=2$ in the dilute case ($\alpha>3/2$), while $\theta_{\alpha}=2\alpha-1$ in the dense case ($\alpha<3/2$). In the $O(2)$ model, where $\alpha = 3/2$, the two exponents collapse to a single one, $\theta_\alpha = 2$.
Chen, Curien and Maillard proved the convergence of the martingale $(W_{\ell})_{\ell\ge 0}$ and determined the law of the limit $W_{\infty}$, \textit{cf.} \cref{thm: CCM convergence martingales}.
In addition, they proved the following convergence result for $n\in (0,2)$. {One} can check that the proof transfers without change to the case $n=2$ once \cref{prop: CCM convergence gen 1} is established. Let $\cal U_{\ell} := \{u\in\cal U, \; |u|=\ell\}$, $\ell\in \bb N^*$. 

\begin{prop} \label{prop: CCM convergence fixed gen}
\emph{(\cite[Proposition 15 and Lemma 16]{chen2020perimeter}.)}

\noindent Let $\ell\in \bb N^*$ and $n\in(0,2]$. The following convergence in distribution holds in $\ell^{\theta}(\cal U_{\ell})$ for all $\theta>\alpha$:
\[
\frac{1}{p} (\chi^{(p)}(u), u\in \cal U_{\ell}) \xrightarrow[p\to\infty]{\mathrm{(d)}} (Z_{\alpha}(u), u\in \cal U_{\ell}).
\]
\end{prop}

\subsection{The derivative martingale of the multiplicative cascade}
\label{sec:bdry case derivative}

In this section, we work under the assumption that $n=2$, hence $\alpha = 3/2$ and $\theta_\alpha=2$. This case corresponds to the so-called {\it boundary case} in the setting of branching random walks \cite{BigginsJ.D.2005Fpot}, in which the Biggins transform \eqref{eq: biggins transform} satisfies  $\phi'_\alpha(\theta_\alpha)=0$.

{
\bigskip
\noindent \textbf{Boundary case of the multiplicative cascade.}
We start by recalling a few properties of branching random walks in the boundary/non-boundary cases. We only work in a restrictive setup and refer to \cite{BigginsJ.D.2005Fpot,shi2016branching} for details and more general results. Moreover, to stick to the framework of our paper, we will describe these in terms of a generic multiplicative cascade $X$ rather than branching random walks. Let $\phi_X$ its Biggins transform, defined as in \eqref{eq: biggins transform} with $X$ in place of $Z_\alpha$. We assume that $\phi_X$ is finite on an interval $(a,b) \subset \bb R_+$ and that there exists a value $\theta_X \in (a,b)$ such that $\phi_X(\theta_X)=1$. Since $\phi_X$ is convex, one may further assume that $\phi_X'(\theta_X) \leq 0$ by choosing the smallest root of $\phi_X-1$ (provided $\phi_X$ is non-constant). By the branching property, one can see that the process
\[
W_{\ell}^X := \sum_{|u|=\ell} X(u)^{\theta_X}, \quad \ell\ge 0,
\]
is a martingale. Since it is non-negative, it must converge almost surely to some limit $W_\infty^X$ as $\ell\to\infty$. It is then important to know whether $W_\infty^X$ is degenerate or not.} 

{The \emph{non-boundary case} corresponds to $\phi_X'(\theta_X)<0$. In this case, the Biggins martingale convergence theorem \cite[Theorem 3.2]{shi2016branching} implies that $W_\infty^X$ is non-degenerate and in fact $W_\ell^X \to W_\infty^X$ in $L^1$. For future reference, we also recall from \cite[Theorem 1.3]{shi2016branching} that in this case, the extremal particles in the cascade decay exponentially: there exists $\upsilon>0$ such that on the event of non-extinction,
\begin{equation} \label{eq: extremal asympt}
 \frac{1}{\ell} \ln \Big(\sup_{|u|=\ell}  X(u)\Big ) \rightarrow - \upsilon,
 \qquad \text{almost surely as } \ell\to\infty,
\end{equation}
which implies that $\sup_{|u|=\ell}  X(u) \to 0$ almost surely as $\ell\to\infty$.}

{
In the \emph{boundary case}, however, $\phi_X'(\theta_X)=0$ and the Biggins martingale convergence theorem yields that $W_\infty^X=0$ almost surely. Therefore, the convergence of $W_\ell^X$ towards $W_\infty^X$ does not provide precise enough information, and one needs to look at the next order term. In fact, because of the assumption $\phi_X'(\theta_X)=0$, the process
\[
D_{\ell}^X := - \sum_{|u|=\ell} X(u)^{\theta_X}\ln (X(u)), \quad \ell\ge 0,
\]
obtained by ``differentiating'' $W^X_\ell$ with respect to the exponent $\theta_X$, is now a martingale. This martingale is no longer non-negative, but under our assumptions it is possible to show that it converges almost surely to some non-negative random variable $D_\infty^X$. Under mild conditions, this limit is non-degenerate. We provide a more detailed account on these conditions in the case of $Z_\alpha$ in the next paragraph.
}

{
\bigskip
\noindent \textbf{The law of the derivative martingale.}
}
The derivative martingale as introduced in \eqref{eq: intro def martingale D_ell} has the expression
\[
D_{\ell} := - 2 \sum_{|u|=\ell} (Z_{\alpha}(u))^{2}\ln (Z_{\alpha}(u)), \quad \ell\ge 0.
\]

\noindent Under the assumption 
\begin{equation} \label{eq: log^2 condition}
\mathbb{E}\bigg[\sum_{\abs{u}=1}(Z_{\alpha}(u))^{2}(\ln Z_{\alpha}(u))^{2}\bigg] < \infty,
\end{equation}

\noindent (which holds since $\phi_\alpha(\theta)<\infty$ in a neighborhood of $\theta_\alpha=2$), $D_\infty:=\lim_{\ell\to\infty} D_\ell$ exists a.s. and is non-negative. The necessary and sufficient condition for $D_\infty>0$ on the event of non-extinction is provided respectively by \cite{ChenXinxin2015Anas} and \cite{aidekonminimum}, and reads
\begin{equation}\label{derivative martingale condition}
\mathbb{E}[W_1(\ln_+ W_1)^2] < \infty \quad \text{and} \quad \mathbb{E}[Y\ln_+Y] < \infty,
\end{equation}

\noindent for $Y = \sum_{|u|=1}(Z_{\alpha}(u))^2\ln_+ \Big(\frac{1}{Z_{\alpha}(u)}\Big)$. Here and throughout $\ln_+(x) = \max(0, \ln x)$. 

\begin{prop}\label{p:derivative martingale}
    We have $D_\infty>0$ a.s. and $\frac{1}{D_\infty}$ is exponentially distributed with parameter $1$.
\end{prop}

\begin{proof}
    A sufficient condition for {condition} \eqref{derivative martingale condition} is the existence of $\eta>0$ such that
    \[
    \bb E\bigg[ \bigg( \sum_{|u|=1}(Z_{\alpha}(u))^{\theta}\bigg)^{1+\eta}\bigg]<\infty,
    \]

    \noindent for $\theta$ in a neighborhood of $\theta_\alpha=2$. It is a consequence of \cref{thm: tail biggins} and \cref{prop: CCM convergence fixed gen}. 
    Now let
    \[\psi(x) = \psi_{\frac32, 2}(x) = \int^{\infty}_0 \mathrm{e}^{-xy-\frac{1}{y}} \frac{\mathrm{d} y}{y^2} = \int^{\infty}_0 \mathrm{e}^{-\frac{x}{t} - t}\mathrm{d}t, \]

    \noindent as in \eqref{eq: intro psi}. Then $\psi(x)$ is the Laplace transform of the inverse-exponential distribution with parameter $1$. By \cite[Lemma 10]{chen2020perimeter}, $\psi$ satisfies the following functional equation:
    \[\psi(x) = \mathbb{E}\bigg[\prod^{\infty}_{i=1} \psi\big(x(Z_{\frac32}(i))^2\big)\bigg], \quad x\ge 0. \] 
    \noindent This equation is a particular case of the so-called \emph{smoothing transform}. By \cite[Theorem 1.5]{AlsmeyerGerold2022ASMt} (with $T_j = (Z_{\alpha}(j))^2$, $\alpha = r = 1$ in their setup), such a solution $\psi$ {must be the Laplace transform of $D_\infty$}, up to a scaling factor. Hence $1/D_{\infty}$ has exponential distribution. By \cite{Buraczewskitailderivative}, $\bb P(D_\infty \ge x)\sim \frac{1}{x}$ as $x\to\infty$, which shows that the parameter of the exponential distribution is $1$.   
\end{proof}

\subsection{Estimates for the left-continuous random walk $S$ and Kemperman's formula}\label{sec: estimtate r.w.}

We first gather some estimates for the times $T_p$ and $L_p$ of \eqref{eq: def Tp and Lp} that will be helpful in light of formula \eqref{eq: key formula maps to RW}. They can be seen as consequences or analogues of \cite[Section 2.3.1]{chen2020perimeter}.  By \cite[Lemma 2.1]{AfanasyevV.I.2005CfBP} and its proof,
\begin{equation}\label{eq:doney}
\bb P(T_1 > k)=\bb P\Big(\min_{0\le j\le k} S_j\ge 0\Big) \sim   k^{-1/\alpha} \ell(k), \qquad \bb P\Big(\max_{0\le j\le k} S_j\le 0 \Big) \sim   k^{-(1-1/\alpha)} \overline{\ell}(k),
\end{equation}

\noindent where $\ell, \overline{\ell}$ are slowly varying functions such that $\ell(k)\overline{\ell}(k)$ converges to a positive constant as $k$ goes to infinity. 
{Recall the two Cases~\ref{caseA} and~\ref{caseB} for the offspring distribution $\mu_{\textnormal{JS}}$ in \eqref{def: non generic critical mu_JS}.}
Using \cite[Theorem 1]{doney1982exact} for the asymptotics of $\bb P(\max_{0\le j\le k} S_j\le 0)$, we deduce the {more precise asymptotics}
\begin{equation}\label{eq: tail T1}
\bb P(T_1 > k) \sim 
\begin{cases}
    C k^{-1/\alpha}, & \text{ \textnormal{in} {Case}~\ref{caseA},}\\
    C (k\ln (k))^{-2/3}, & \text{ \textnormal{in} {Case}~\ref{caseB}, }
\end{cases}
\quad \text{as } k\to\infty.
\end{equation}

\noindent It will be convenient to set 
\begin{equation}\label{def: fp}
f(p) = 
\begin{cases}
    p^{\alpha}, & \text{ \textnormal{in} {Case}~\ref{caseA},}\\
    p^{3/2}(\ln p + 1)^{-1}, & \text{ \textnormal{in} {Case}~\ref{caseB}.}
\end{cases}
\end{equation}

\noindent Observe that $f(p)^{-1/\alpha} \ell(f(p)) \sim C/p$ as $p$ goes to infinity.  By the invariance principle applied to the random walk $(T_p,\,p\ge 1)$, $\frac{T_p}{f(p)}$ converges in distribution towards $\tau$, the hitting time of $-1$ by the stable process $\zeta$, a fact that  could also be deduced from the invariance principle of the random walk $S$. On the other hand, using that for all $k\ge p$, $\bb P(T_p\le k) \le \bb P(T_1\le k)^p$, \cref{eq: tail T1} shows that $\frac{f(p)}{T_p}$ is bounded in any $L^r$, $r\ge 1$ and in particular 
\[
f(p)\mathbb{E}\Big[\frac{1}{T_p}\Big] \underset{p\to \infty}\longrightarrow  \bb E[1/\tau].
\]

\noindent Since $L_p=\sum_{i=1}^{T_p} \mathds{1}_{\{X_i=-1\}}$, the law of large numbers entails that $\frac{L_p}{T_p} \to \bb P(X_1=-1)=\mu_{\text{JS}}(0)$ in probability. The following inequality is Equation (14) of \cite[Lemma 5]{chen2020perimeter} whose proof transfers to the case $n=2$ without change: there exists a constant $c>0$, for all $K\ge 2/\mu_{\textnormal{JS}}(0)$, and all $p\ge 1$,
\begin{equation} \label{eq: exponential tail Tp/Lp}
\bb P\bigg(\frac{T_p}{L_p}\ge K \bigg) \le c^{-1} \mathrm{e}^{-cKp}.
\end{equation}

\noindent As a consequence, we have the following convergence: 
\begin{equation} \label{eq: L1 convergence Lp}
f(p)\mathbb{E}\Big[\frac{1}{1+L_p}\Big] \underset{p\to \infty}\longrightarrow \mu_{\textnormal{JS}}(0)^{-1} \bb E[1/\tau],
\end{equation}

\begin{equation} \label{eq: L1 convergence Tp/Lp}
 \mathbb{E}\Big[\frac{T_{p}}{1+L_{p}} \Big] \underset{p\to\infty}{\longrightarrow} \mu_{\textnormal{JS}}(0)^{-1}.
\end{equation}

It will also be convenient to state here for future reference the following technical proposition, which expresses the exchangeability of the increments of the random walk $(S_k)_{k \ge 1}$ in \eqref{eq: RW Janson--Stefansson tree}. Such identities are valid for left-continuous random walks, in the sense that the increments satisfy a.s. $X_i \ge -1$ for all $i\ge 1$. They already appear in the work of Chen, Curien and Maillard \cite{chen2020perimeter}; for more general context, see also \cite[Section 6.1]{pitman2006combinatorial} and references therein. 

\begin{prop}\label{prop: kemperman} \emph{(\cite[Theorem 2 \& Proposition 6]{chen2020perimeter}.)}
Assume that $T_1<\infty$ a.s. For any positive measurable function $f:\mathbb{Z}\to \bb R$ and any $p\ge 2$,
\[
 \bb E \bigg[\frac{1}{T_p-1}\sum_{i=1}^{T_p} f(X_i)\bigg]
 =
 \bb E\bigg[\frac{p}{p+X_1} f(X_1)\bigg]. 
\]

\noindent More generally, if moreover $g:\bigcup_{j=1}^\infty \mathbb{Z}^j\to \bb R_+$ is a symmetric measurable function,
\[
 \bb E \bigg[\frac{1}{T_p-1}\sum_{i=1}^{T_p} f(X_i)g((X_j)_{j\neq i,j\le T_p})\bigg]
 =
 \bb E\bigg[\frac{p}{p+X_1} f(X_1) \bb E[g((X_j)_{j\le T_q})]_{q=p+X_1}\bigg].
\]
\end{prop}


\section{Discrete Biggins transform and discrete martingales}\label{sec: discrete Biggins transform}

\subsection{Uniform tail estimates for the discrete Biggins transform}
\label{sec: tail estimates Biggins}
Recall the definition of the discrete Biggins transform
\[\phi^{(p)}(\theta) = \mathbb{E}^{(p)}\bigg[\sum_{\abs{u}=1}\bigg(\frac{\chi^{(p)}(u)}{p}\bigg)^{\theta}\bigg]. \]
Chen, Curien and Maillard proved the convergence of $\phi^{(p)}(\theta)$ for $\theta\in(\alpha,\alpha+1)$ towards the Biggins transform $\phi_\alpha(\theta)$ of the multiplicative cascade, namely
\[\phi_{\alpha}(\theta) = \mathbb{E}\bigg[\sum_{\abs{u}=1}(Z_{\alpha}(u))^{\theta}\bigg]. \]
In this section, we establish uniform tail estimates for the discrete Biggins transform $\phi^{(p)}$. These estimates will be helpful to control the size of the offspring of typical loops in the discrete multiplicative cascade $(\chi^{(p)}(u),u\in\cal U)$. As we shall see, this will enable us to rule out the contribution of a \emph{bad region} of the map to the volume in the scaling limit as $p\to \infty$ (namely those which do not satisfy the \emph{moderate increments} property of \cref{def: bad vertices}). We first consider the tail of sums related to the Biggins transform in the following lemma.

\begin{lem} \label{lem: tail power RW}
Let $\theta\in (\alpha,\alpha+1)$ and $r\in (0, 1/(\alpha+1))$. There exists a constant $C>0$ such that, for all $p\ge 2$ and $A>1$, 
\[
\mathbb{P}\bigg(\sum^{T_{p}}_{i=1}\left(\frac{X_i+1}{p}\right)^\theta>A\bigg)
\le
C  A^{- r }.
\]
\end{lem}

\begin{proof}
Let $\eps \in(0,1)$ (to be determined later) and $f(p)$ as defined in \cref{def: fp}. We now split the above probability as
\begin{align}
&\mathbb{P}\bigg(\sum^{T_p}_{i=1}\left(\frac{X_i+1}{p}\right)^{\theta}>A \bigg) \notag \\
&=
\mathbb{P}\bigg(\sum^{T_p}_{i=1}\left(\frac{X_i+1}{p}\right)^{\theta}>A, \, T_p \leq\eps^{-1} f(p)\bigg)
+ \mathbb{P}\bigg(\sum^{T_p}_{i=1}\left(\frac{X_i+1}{p}\right)^{\theta}>A, \, T_p > \eps^{-1} f(p)\bigg) \notag \\
&\leq \mathbb{P}\bigg(\frac{1}{T_p}\sum^{T_p}_{i=1}\left(\frac{X_i+1}{p}\right)^{\theta}>\eps f(p)^{-1} A \bigg) + \mathbb{P}(T_p> \eps^{-1} f(p)). \label{eq: Delta_2 split}
\end{align}

\noindent By Equation (2.2) in \cite[Lemma 2.1]{AfanasyevV.I.2005CfBP}, there exists a constant $C>0$ such that for all $p\geq 1$ and $n\geq 1$, in the notation of \cref{eq:doney}, 
\[\mathbb{P}(T_p > n) = \mathbb{P}(\min_{1\leq i\leq n} S_i \geq - {p}) \leq C pn^{-1/\alpha}\ell(n). \]  

\noindent {Since $\ell$ is slowly varying,} for all $\delta>0$, one can find  $C>0$ such that $\ell(ax)\le C a^\delta \ell(x)$ for all $a,x\ge 1$. Fix $\delta>0$. By our choice of $f(p)$,  the second term of \eqref{eq: Delta_2 split} is {thus} bounded by
\begin{equation}\label{eq: Delta_2 bound second}
\mathbb{P}(T_p> \eps^{-1} f(p))
\le
C  \eps^{-\delta}\eps^{1/\alpha}.
\end{equation}

\noindent On the other hand, we can handle the first term of \eqref{eq: Delta_2 split} by Markov's inequality and Proposition \ref{prop: kemperman}:
\begin{multline*}
\mathbb{P}\bigg(\frac{1}{T_p}\sum^{T_p}_{i=1}\left(\frac{X_i+1}{p}\right)^{\theta}>\eps f(p)^{-1} A \bigg) \\
\le
\frac{f(p)p^{-\theta}}{\eps A} \bb E \bigg[ \frac{1}{T_p}\sum^{T_p}_{i=1}(X_i+1)^{\theta}\bigg]
\le
\frac{f(p)p^{-\theta}}{\eps A} \bb E \left[ (X_1+1)^{\theta}\frac{p}{p+X_1}\right].
\end{multline*}
Recall from \eqref{def: non generic critical mu_JS} that there are two cases for the tail of the random walk. 
In {Case}~\ref{caseA}, the step distribution of the walk satisfies $\mu_{\textnormal{JS}}(k) \le C k^{-\alpha-1}$, so that the expectation
\begin{equation} \label{eq: sum uniformly bounded}
f(p)p^{-\theta} \bb E \bigg[ (X_1+1)^{\theta}\frac{p}{p+X_1}\bigg]
=
p^\alpha \sum_{k=1}^{\infty} \mu_{\textnormal{JS}}(k)\left(\frac{k}{p}\right)^{\theta}\frac{p}{p+k-1}
\added{\le 
\frac{1}{p} \sum_{k=1}^{\infty} \left(\frac{k}{p}\right)^{\theta-\alpha-1}\frac{1}{1+\frac{k-1}{p}},}
\end{equation}
is uniformly bounded over $p\ge 1$, \added{as a Riemann sum}. In {Case}~\ref{caseB}, the step distribution satisfies $\mu_{\textnormal{JS}}(k) \le C k^{-5/2}(\ln k + 1)$, the expectation
\begin{align*}
f(p) p^{-\theta} \bb E \bigg[ (X_1+1)^{\theta}\frac{p}{p+X_1}\bigg]
&=
\frac{p^{3/2}}{(\ln p+1)}\sum_{k=1}^{\infty} \mu_{\textnormal{JS}}(k)\left(\frac{k}{p}\right)^{\theta}\frac{p}{p+k-1} \\
&\le 
\frac{p^{3/2}}{(\ln p+1)}\sum_{k=1}^{\infty}{k^{-5/2}} \left(\frac{k}{p}\right)^{\theta}\frac{p}{p+k-1} \Big(\ln \frac{k}{p} + \ln p + 1\Big)
\\
&\added{\le 
\frac{C}{p} \sum_{k=1}^{\infty} \left(\frac{k}{p}\right)^{\theta-5/2}\frac{1}{1+\frac{k-1}{p}} + \frac{C}{p(\ln p + 1)} \sum_{k=1}^{\infty} \left(\frac{k}{p}\right)^{\theta-5/2}\frac{\ln(k/p)}{1+\frac{k-1}{p}},}
\end{align*}
is again uniformly bounded over $p\ge 1$, since the first term is a Riemann sum and the second term is $1/\ln p$ multiplied by a Riemann sum. In any case, we end up with the simple bound
\begin{equation} \label{eq: Delta_2 bound first}
\mathbb{P}\bigg(\frac{1}{T_p}\sum^{T_p}_{i=1}\left(\frac{X_i+1}{p}\right)^{\theta}>\eps f(p)^{-1}A \bigg)
\le
\frac{C}{\eps A}.
\end{equation}
Combining \eqref{eq: Delta_2 split}, \eqref{eq: Delta_2 bound second} and \eqref{eq: Delta_2 bound first}, we have obtained the following bound:
\begin{equation} \label{eq: tail sum X_i final eps}
\mathbb{P}\bigg(\sum^{T_p}_{i=1}\left(\frac{X_i+1}{p}\right)^{\theta}>A\bigg)
\le
C\eps^{-\delta}\varepsilon^{1/\alpha} + \frac{C}{\varepsilon A}.
\end{equation}
Take $\eps = A^{-\alpha/(\alpha+1-\alpha \delta)}$. Plugging this relation into \eqref{eq: tail sum X_i final eps} implies 
\[
\mathbb{P}\bigg(\sum^{T_p}_{i=1}\left(\frac{X_i+1}{p}\right)^{\theta}>A\bigg)
\le
C  A^{-\frac{1-\alpha\delta }{\alpha+1-\delta\alpha}},
\]
for some constant $C>0$. It yields the desired upper bound since $\delta>0$ can be taken arbitrarily small.
\end{proof}

\begin{thm}\label{thm: tail biggins}
Let $\theta\in (\alpha,\alpha+1)$, $\gamma_0 = \min(\frac{\alpha+1-\theta}{\theta},\frac{1}{\alpha+1})$ and $\eta\in (0, \gamma_0)$. There exists a  constant $C>0$ such that for all $p\ge 2$, 
\[\mathbb{E}^{(p)}\bigg[\bigg(\sum_{\abs{u}=1} \bigg(\frac{\chi^{(p)}(u)}{p}\bigg)^\theta \bigg)^{1+\eta}\bigg] \leq C.\]
\end{thm}
\begin{proof} 
We allow the constant $C>0$ to vary from line to line in the proof. Recall that the faces of the gasket $\frak g$ of $(\frak q,\bf\ell)$ correspond to loops of $(\frak q,\bf\ell)$, except possibly for some faces of degree $4$.
By \eqref{eq: key formula maps to RW}, we therefore get
\begin{equation}\label{eq: main tail biggins}
     \mathbb{E}^{(p)}\bigg[\bigg(\sum_{\abs{u}=1} \bigg(\frac{\chi^{(p)}(u)}{p}\bigg)^{\theta}\bigg)^{1+\eta}\bigg]
 \le \frac{1}{\mathbb{E}\big[\frac{1}{1+L_p}\big]}\mathbb{E}\bigg[\frac{1}{1+L_p}\bigg(\sum^{T_p}_{i=1}\left(\frac{X_i+1}{p}\right)^{\theta} \bigg)^{1+\eta}\bigg],
\end{equation}

\noindent so that we only need to bound the right-hand side. It is easier to first bound the expectation
\[
\bb E\bigg[\frac{1}{T_p-1}\bigg(\sum^{T_p}_{i=1}\left(\frac{X_i+1}{p}\right)^{\theta} \bigg)^{1+\eta}\bigg].
\]

\noindent Since $(a+b)^\eta\le a^\eta +b^\eta$ for $\eta\in(0,1)$ and $a,b>0$,  the latter expectation is smaller than
\[
\bb E\bigg[\frac{1}{T_p-1}\sum^{T_p}_{i=1}\left(\frac{X_i+1}{p}\right)^{\theta(1+\eta)} \bigg]+ \bb E \bigg[\frac{1}{T_p-1}\sum^{T_p}_{i=1}\left(\frac{X_i+1}{p}\right)^{\theta}\bigg( \sum_{j=1}^{T_p} \left(\frac{X_j+1}{p}\right)^\theta\mathds{1}_{\{j\neq i\}}\bigg)^\eta \bigg].
\]

\noindent  By Proposition \ref{prop: kemperman}, for all $p\ge 2$, these terms are respectively equal to
    \begin{equation} \label{eq: tail biggins term1}
\bb E\bigg[ \left(\frac{X_1+1}{p}\right)^{\theta(1+\eta)} \frac{p}{p+X_1} \bigg] 
\le
\bb E\bigg[ \left(\frac{X_1+1}{p}\right)^{\theta} \left(\frac{p+X_1}{p}\right)^{\theta \eta -1 } \bigg], 
    \end{equation}

\noindent and
\begin{equation}\label{eq:tail biggins term2}
\mathbb{E}\bigg[\left(\frac{X_1+1}{p}\right)^{\theta}  \frac{p}{p+X_1}\mathbb{E}\bigg[\left( \sum_{j=1}^{T_{q}} \left(\frac{X_j+1}{p}\right)^\theta\right)^\eta\bigg]_{q=p+X_1} \bigg].
\end{equation}
   
\noindent By Lemma \ref{lem: tail power RW}, using that $\eta(\alpha+1)<1$, there exists a constant $C>0$ such that for all $q\ge 2$,
\[
\mathbb{E}\bigg[\bigg( \sum_{j=1}^{T_{q}} (X_j+1)^\theta \bigg)^\eta\bigg] \le C q^{^{\theta \eta }}. 
\]

\noindent We obtain that \eqref{eq:tail biggins term2} is also smaller than a constant times
\[
\bb E\bigg[ \left(\frac{X_1+1}{p}\right)^{\theta} \left(\frac{p+X_1}{p}\right)^{\theta \eta - 1} \bigg].
\]

\noindent Combining this with \eqref{eq: tail biggins term1}, we deduce
\[
\bb E\bigg[\frac{1}{T_p-1}\bigg(\sum^{T_p}_{i=1}\left(\frac{X_i+1}{p}\right)^{\theta} \bigg)^{1+\eta}\bigg]
\le
C \bb E\bigg[ \left(\frac{X_1+1}{p}\right)^{\theta} \left(\frac{p+X_1}{p}\right)^{\theta \eta - 1} \bigg].
\]

\noindent Now
\begin{equation}\label{eq:tail biggins term3}
\bb E\bigg[ \left(\frac{X_1+1}{p}\right)^{\theta} \left(\frac{p+X_1}{p}\right)^{\theta \eta - 1} \bigg] \leq  C\sum^{\infty}_{k=1} \left(\frac{k}{p}\right)^{\theta}\left(\frac{k}{p} + 1\right)^{\theta \eta - 1}\mu_{\textnormal{JS}}(k). 
\end{equation}

\noindent In {Case}~\ref{caseA}, $\mu_{\textnormal{JS}}(k)\leq Ck^{-\alpha-1}$. Thus \eqref{eq:tail biggins term3} is smaller than 
\[Cp^{-\alpha{-1}} \cdot \sum^{\infty}_{k=1} \left(\frac{k}{p}\right)^{\theta-\alpha-1} \left(\frac{k}{p} + 1\right)^{\theta\eta-1}\leq C p^{-\alpha}, \]

\noindent since the summation is a Riemann sum, by our choice of $0<\eta<\frac{\alpha+1-\theta}{\theta}$. In {Case}~\ref{caseB}, $\mu_{\textnormal{JS}}(k)\leq Ck^{-5/2}(\ln k + 1)$. In that case \eqref{eq:tail biggins term3} is smaller than a constant times
\[
p^{{-5/2}}\sum^{\infty}_{k=1} \left(\frac{k}{p}\right)^{\theta-5/2}\left(\frac{k}{p} + 1\right)^{\theta\eta-1} \ln\bigg(\frac{k}{p}\bigg)\\
+ p^{{-5/2}}\ln p \sum^{\infty}_{k=1} \left(\frac{k}{p}\right)^{\theta-5/2} \left(\frac{k}{p} + 1\right)^{\theta\eta-1} \leq Cp^{-3/2}\ln p,
\]

\noindent since the two summations are Riemann sums {again by our choice of $\eta$}. We conclude that
\begin{equation}\label{eq:tail biggins proof eta}
    \bb E\bigg[\frac{1}{T_p-1}\bigg(\sum^{T_p}_{i=1}\left(\frac{X_i+1}{p}\right)^{\theta} \bigg)^{1+\eta}\bigg] \le C \bb E\bigg[ \left(\frac{X_1+1}{p}\right)^{\theta} \left(\frac{p+X_1}{p}\right)^{\theta \eta - 1} \bigg] \leq Cf(p)^{-1}. 
\end{equation}

We now need to replace $T_p-1$ by $1+L_p$. Let $r, r'> 1$ be Hölder conjugates, and take $r$ big enough so that $\eta_r:=\frac{r}{r-1}(1+\eta) -1< \gamma_0 = \min(\frac{\alpha+1-\theta}{\theta},\frac{1}{\alpha+1})$. In particular \eqref{eq:tail biggins proof eta} is valid for $\eta_r$ in place of $\eta$. H\"older's inequality implies that
\begin{multline*}
\bb E\bigg[\frac{1}{1+L_p}\bigg(\sum^{T_p}_{i=1}\left(\frac{X_i+1}{p}\right)^{\theta} \bigg)^{1+\eta}\bigg] 
=
\bb E\bigg[\frac{(T_p-1)^{1/r'}}{1+L_p}\cdot \frac{1}{(T_p -1)^{1/r'}}\bigg(\sum^{T_p}_{i=1}\left(\frac{X_i+1}{p}\right)^{\theta} \bigg)^{1+\eta}\bigg] 
\\
\le 
\mathbb{E}\bigg[ \frac{1}{T_p -1} \bigg(\sum^{T_p}_{i=1}\left(\frac{X_i+1}{p}\right)^{\theta} \bigg)^{(1+\eta)r'} \bigg]^{1/r'} \mathbb{E}\bigg[\frac{(T_p-1)^{r/r'}}{(1+L_p)^r}\bigg]^{1/r}
\stackrel{\eqref{eq:tail biggins proof eta}}{\le}
C f(p)^{\frac{1}{r}-1} \, \mathbb{E}\bigg[\frac{(T_p-1)^{r-1}}{(1+L_p)^r}\bigg]^{1/r}.
\end{multline*}

\noindent In view of \eqref{eq: main tail biggins} and \eqref{eq: L1 convergence Lp}, it remains to show that $f(p)\, \mathbb{E}\Big[\frac{(T_p-1)^{r-1}}{(1+L_p)^r}\Big]$ is bounded in $p\ge 2$. Let $K\ge 2/\mu_{\textnormal{JS}}(0)$. We have
\[
\mathbb{E}\bigg[\frac{(T_p-1)^{r-1}}{(1+L_p)^r}\bigg]
\le
\mathbb{E}\bigg[\frac{1}{1+L_p} \bigg(\frac{T_p}{L_p}\bigg)^{r-1}\bigg]
\le
K^{r-1} \mathbb{E}\bigg[\frac{1}{1+L_p}\bigg] + \mathbb{E}\bigg[  \bigg(\frac{T_p}{L_p}\bigg)^{r-1} \mathds{1}_{\big\{\frac{T_p}{L_p} \ge K\big\}}\bigg].
\]

\noindent It is indeed smaller than $Cf(p)^{-1}$ by \eqref{eq: L1 convergence Lp} and \eqref{eq: exponential tail Tp/Lp}.
\end{proof}

\subsection{Convergence towards the additive martingale}

Let $(h_p:\mathbb{R}_+\to \mathbb{R}_+)_{p\ge 1}$ be a family of non-negative measurable functions. {Recall {from \eqref{eq: def theta_alpha} that} $\theta_{\alpha} = \min(2, 2\alpha-1)$.} We suppose that there exists $c>0$ and $\theta\in(\alpha,\alpha + 1)$ such that 
\begin{equation}\label{eq: bound gp}
h_p(x)x^{\theta_\alpha}\le c (x^{\theta}+x^{\theta_\alpha}),
\end{equation}

\noindent for all $p\ge 1$ and $x\ge 0$. We also suppose that there exists a measurable function $h:\mathbb{R}_+\mapsto \mathbb{R}$ such that for any sequence $(x_p)_{p\ge 1}$ which converges to some $x>0$, $\lim_{p\to\infty} h_p(x_p)=h(x)$. For future use, fix $\delta > 0$ such that $\theta - \delta > \alpha$ and $\theta_{\alpha} - \delta > \alpha$. Then there exists a constant $c>0$ such that for all $p,q>0$, 
\[\frac{1 + \ln p}{1 + \ln q}\leq c\bigg(\frac{p}{q}\bigg)^{\delta} + 1. \]

\begin{lem}\label{l: passing to infinity}
Let $n\in (0,2]$ and $(h_p)_{p\ge 1}$ be as above.  As $p\to\infty$, 
\begin{equation}\label{passing to infinity} 
\sum_{|u| = \ell} h_p \bigg(\frac{\chi^{(p)}(u)}{p} \bigg)\frac{\overline{V}(\chi^{(p)}(u))}{\overline{V}(p)}  \overset{(\mathrm{d})}{\longrightarrow}  \sum_{|u| = \ell} h \big(  Z_\alpha(u) \big) (Z_\alpha(u))^{\theta_\alpha}.
\end{equation}
\end{lem}

{
\begin{rem}
    We will chiefly apply \cref{l: passing to infinity} when $h_p$ does not depend on $p$ (actually even when $h_p=1$). The general form of the statement will be used at the very end of the proof of \cref{thm: main} (\cref{sec: proof main result}). 
\end{rem}
}

\begin{proof} 
By \cref{prop: CCM convergence fixed gen}, in $\ell^{\theta}(\cal U_{\ell})$, 
\[
\frac{1}{p} (\chi^{(p)}(u), |u|=\ell) \xrightarrow[p\to\infty]{(\mathrm{d})} (Z_{\alpha}(u),  |u|=\ell).
\]

\noindent We can suppose by Skorokhod's representation theorem that the convergence is almost sure. Note that it implies the convergence in any $\ell^r(\cal U_{\ell})$ with $r\in[\theta,\infty]$. Let $\varepsilon>0$. Let $U_\eps\subset \cal U_\ell$ be a (random) finite set such that $\sum_{u\in {\cal U}_\ell\setminus U_\eps} (Z_\alpha(u))^{\beta} < \varepsilon$ for $\beta = \theta, \theta - \delta, \theta_{\alpha}$ and $\theta_{\alpha} - \delta$.
{Recall our two {Cases}~\ref{caseA} and~\ref{caseB}.}
By the mean asymptotics in \eqref{eq: mean volume} and \eqref{eq: mean volume 2}, in {Case}~\ref{caseA}, $Cq^{\theta_\alpha}\le \overline{V}(q)\le C'q^{\theta_\alpha}$. Then, by \eqref{eq: bound gp}, for $p$ large enough,
\[
\sum_{u\in {\cal U}_\ell\setminus U_\eps} h_p \bigg(\frac{\chi^{(p)}(u)}{p} \bigg)\frac{\overline{V}(\chi^{(p)}(u))}{\overline{V}(p)}  \le c \sum_{u\in {\cal U}_\ell\setminus U_\eps}\bigg( \frac{\chi^{(p)}(u)}{p}\bigg)^{\theta}+\bigg( \frac{\chi^{(p)}(u)}{p}\bigg)^{\theta_\alpha}\le 2\varepsilon.
\]
\noindent In {Case}~\ref{caseB}, by \eqref{eq: mean volume 2}, $Cq^{\theta_\alpha}(1 + \ln q)^{-1}\le \overline{V}(q)\le C'q^{\theta_\alpha}(1 + \ln q)^{-1}$. Then 
\[\frac{\overline{V}(\chi^{(p)}(u))}{\overline{V}(p)} \leq c\bigg(\frac{\chi^{(p)}(u)}{p}\bigg)^{\theta_{\alpha}}\frac{1 + \ln p}{1 + \ln \chi^{(p)}(u)} \leq c\bigg(\frac{\chi^{(p)}(u)}{p}\bigg)^{\theta_{\alpha}}\left(\bigg(\frac{\chi^{(p)}(u)}{p}\bigg)^{-\delta} + 1\right).\]

\noindent Again by \eqref{eq: bound gp}, for $p$ large enough, 
\[
\sum_{u\in {\cal U}_\ell\setminus U_\eps} h_p \bigg(\frac{\chi^{(p)}(u)}{p} \bigg)\frac{\overline{V}(\chi^{(p)}(u))}{\overline{V}(p)}  \le c \sum_{u\in {\cal U}_\ell\setminus U_\eps}\sum_{\beta = \theta, \theta - \delta, \atop \theta_{\alpha}, \theta_{\alpha} - \delta}\bigg( \frac{\chi^{(p)}(u)}{p}\bigg)^{\beta}\le 4\varepsilon.
\]

\noindent Moreover, almost surely, $ h_p \Big(\frac{\chi^{(p)}(u)}{p} \Big)\to h(Z_\alpha(u))$ by assumption on $h_p$, and  $\frac{\overline{V}(\chi^{(p)}(u))}{\overline{V}(p)}\to (Z_\alpha(u))^{\theta_\alpha}$ by \eqref{eq: mean volume}. We deduce that for $p$ large 
\begin{align*}
&\bigg| \sum_{|u| = \ell} h_p \bigg(\frac{\chi^{(p)}(u)}{p} \bigg)\frac{\overline{V}(\chi^{(p)}(u))}{\overline{V}(p)}- 
\sum_{|u| = \ell}  h \big(  Z_\alpha(u) \big) (Z_\alpha(u))^{\theta_\alpha} \bigg|\\
&\le  \bigg|\sum_{u\in U_\eps} \bigg( h_p \bigg(\frac{\chi^{(p)}(u)}{p} \bigg)\frac{\overline{V}(\chi^{(p)}(u))}{\overline{V}(p)}-h \big(  Z_\alpha(u) \big)(Z_\alpha(u))^{\theta_\alpha} \bigg) \bigg| +8\varepsilon \underset{p\to\infty}{\longrightarrow} 8\varepsilon.
\end{align*}

\noindent Letting $\eps\to 0$  completes the proof. 
\end{proof}

\section{Estimates on the Markov chain $\tt S$}
\label{sec: markov chain estimates}

In this section we introduce a discrete-time Markov chain $(\tt S_n)_{n\geq 0}$ which describes the behaviour of a typical particle in the branching Markov chain $(\chi^{(p)}(u), u\in\cal U)$, and will be the key observable for the derivation of \cref{thm: main}.
{
For an overview of how these estimates combine into the final proof, we suggest to have a look at the proof diagrams of Figures~\ref{fig: proof diagram 1}, \ref{fig: proof diagram 2} and \ref{fig: proof diagram 3} in \cref{sec: proof main result}.
}

\subsection{The Markov chain $\tt S$ and the many-to-one formula}
\label{sec: markov chain}

\medskip
\noindent \textbf{The Markov chain $\tt S$.} The Markov chain $\tt S$ is informally defined by recording the half-perimeter of the nested loops containing a uniform target point in the decorated quadrangulation $(\frak q, \bf\ell)$, starting from the outermost loop. More precisely, under $\tt P_p$ we introduce the Markov chain $(\tt S_n)_{n\geq 0}$ on $\{0, 1, 2, ...\}$ starting at $\tt S_0=p$, and with transition probabilities
\begin{eqnarray*}
\tt{P}_{p}(\tt S_1 = q) &=& \frac{1}{\overline{V}(p)}\mathbb{E}^{(p)}\left[\sum^{\infty}_{i=1} \mathds{1}_{\{\chi^{(p)}(i) = q\}} \overline{V}(q)\right], \quad p, q >0, \\
\tt{P}_{p}(\tt S_1 = 0) &=& 1 - \sum^{\infty}_{q = 1}
\tt{P}_{p}(\tt S_1 = q), \quad p > 0,\\
\tt{P}_{0}(\tt S_1 = 0) &=& 1.
\end{eqnarray*}

This can be rephrased in a more geometric way using the gasket decomposition. Denote by $\bb P_{\bullet}^{(p)}$ the law of a \emph{pointed} loop--$O(n)$ quadrangulation of perimeter $2p$, \added{so that for any event $A$, $\mathbb{P}^{(p)}_{\bullet}(A) = \mathbb{E}^{(p)}[\mathds{1}_A V(\varnothing)]/\overline{V}(p)$} {with $V(\varnothing)$ denoting the volume of the whole map}. Recall from \cref{sec: CCM results} the notation preceding \eqref{eq: gasket decomposition summing outermost loops}. By the Markov property of the gasket decomposition, we immediately have for $p,q \ge 1$,
\begin{align*}
\tt{P}_{p}(\tt S_1 = q) &= \frac{1}{\overline{V}(p)}\mathbb{E}^{(p)}\left[\sum^{\infty}_{i=1} \mathds{1}_{\{\chi^{(p)}(i) = q\}} \overline{V}(q)\right]\\
&= \frac{1}{\overline{V}(p)}\mathbb{E}^{(p)}\left[\sum^{\infty}_{i=1} \mathds{1}_{\{\chi^{(p)}(i) = q\}} V(i)\right]\\
&= \mathbb{E}^{(p)}_{\bullet}\left[\sum^{\infty}_{i=1}\mathds{1}_{\{\chi^{(p)}(i) = q\}} \frac{V(i)}{V(\varnothing)}\right],
\end{align*}
\added{by definition of $\bb P^{(p)}_\bullet$.}

In other words, the Markov chain $\tt S$ can be obtained as follows. First, pick a pointed loop--$O(n)$ quadrangulation $(\frak q_{\bullet}, \bf \ell)$ under $\mathbb{P}^{(p)}_{\bullet}$. Then each time $n\ge 1$ records the half-perimeter of the loop at generation $n$ that contains the distinguished vertex of $(\frak q_{\bullet}, \bf \ell)$; if no such loop exists, that is, the distinguished vertex lies in the gasket of the sub-map, we send the Markov chain to the absorbing cemetery point $0$.

For $M>0$, we denote by $\tt T_M$ the hitting time of $[0,M)$ by the Markov chain $\tt S$:
\begin{equation} \label{eq: tau_M}
	\tt T_M
	:=
	\inf\{n\ge 0, \; \tt S_n < M\}.
\end{equation}

\medskip
\noindent \textbf{A many-to-one formula.}
We have the following many-to-one formula, which is natural in light of the branching property of the gasket decomposition. \added{This key formula relates the behaviour of all cells in the particle system $(\chi^{{(p)}}(u),\; u\in\cal U)$ to that of the distinguished typical particle $\tt S$ (see \cite{kahane1976certaines,lyons1997simple,shi2016branching} for similar results in different contexts).}
\added{Recall that for $u\in \cal U$ and $i\le |u|$, $u_i$ denotes the ancestor of $u$ at generation $i$.}
\begin{prop}\label{prop: many to one}
For all $p\ge 1$, all $n\ge 1$ and all non-negative measurable function $g$ with $g(x_1, x_2, \ldots, x_{n-1}, 0) = 0$, we have
\begin{equation} \label{eq: many-to-one}
\frac{1}{\overline{V}(p)}\mathbb{E}^{(p)}\bigg[\sum_{\abs{u} = n} g(\chi^{(p)}(u_1), \chi^{(p)}(u_2), \ldots, \chi^{(p)}(u_n)) \overline{V}(\chi^{(p)}(u))\bigg] = \tt{E}_p\left[g(\tt S_1, \tt S_2, \ldots, \tt S_n)\right].
\end{equation}
\end{prop}

\begin{proof}
We proceed by induction. For $n = 1$, by definition of $\tt S$, we have with the notation of \cref{prop: many to one},
\begin{align*}
\frac{1}{\overline{V}(p)}\mathbb{E}^{(p)}\left[\sum^{\infty}_{i = 1} g(\chi^{(p)}(i)) \overline{V}(\chi^{(p)}(i))\right]
&=\sum^{\infty}_{q=1}\frac{g(q)}{\overline{V}(p)}\mathbb{E}^{(p)}\left[\sum^{\infty}_{i = 1}\mathds{1}_{\{\chi^{(p)}(i)=q\}}\cdot \overline{V}(q)\right] \\
&=\sum^{\infty}_{q=1}g(q)\tt{P}_{p}(\tt S_1 = q)
= \tt{E}_p[g(\tt S_1)\cdot \mathds{1}_{\{\tt S_1\ne 0\}}]
= \tt{E}_p[g(\tt S_1)].
\end{align*}
Now suppose that the claim holds for some $n \ge 1$. We use the Markov property of the gasket decomposition to find that
\begin{multline*}
\frac{1}{\overline{V}(p)}\mathbb{E}^{(p)}\bigg[\sum_{\abs{u} = n+1} g(\chi^{(p)}(u_1), \chi^{(p)}(u_2), \ldots, \chi^{(p)}(u_{n+1})) \overline{V}(\chi^{(p)}(u))\bigg]\\
=\frac{1}{\overline{V}(p)}\mathbb{E}^{(p)}\bigg[\sum^{\infty}_{i=1}\mathbb{E}^{(q)}\bigg[\sum_{\abs{u} = n}g(q, \chi^{(q)}(u_1), \ldots, \chi^{(q)}(u_n)) \overline{V}(\chi^{(q)}(u))\bigg]_{q= \chi^{(p)}(i)}\bigg].
\end{multline*}
Then, we use \eqref{eq: many-to-one} by the induction assumption, along with the Markov property of $\tt S$:
\begin{align*}
\frac{1}{\overline{V}(p)}
\mathbb{E}^{(p)}\bigg[\sum_{\abs{u} = n+1} & g(\chi^{(p)}(u_1), \chi^{(p)}(u_2), \ldots, \chi^{(p)}(u_{n+1})) \overline{V}(\chi^{(p)}(u))\bigg] \\
&=\frac{1}{\overline{V}(p)}\mathbb{E}^{(p)}\bigg[\sum^{\infty}_{i=1}\overline{V}(\chi^{(p)}(i))\tt{E}_{\chi^{(p)}(i)}[g(\chi^{(p)}(i), \tt S_1, \tt S_2, \ldots, \tt S_n)]\bigg]\\
&=\tt{E}_{p}[g(\tt S_1, \tt S_2, \ldots, \tt S_{n+1})].
\end{align*}
\end{proof}

\medskip
\noindent \textbf{Convergence in distribution of $\tt S$.}
Using that {$\phi_{\alpha}(\theta_\alpha)=1$}, we can define a random variable $\xi$ such that for any $g:\mathbb{R}\to \mathbb{R}$ bounded and measurable, 
\begin{equation}\label{def: random variable xi}
\mathbb{E}[g(\xi)] = \mathbb{E}\bigg[\sum_{|u| = 1} g(Z_{\alpha}(u))(Z_{\alpha}(u))^{\theta_{\alpha}}\bigg]. 
\end{equation}

\noindent Observe that $\bb E[\ln \xi]<0$ if $n\in (0,2)$ and $\bb E[\ln \xi]=0$ if $n=2$. Let $(\xi_i,\,i\ge 1)$ be i.i.d. copies of $\xi$, and let $Y$ be the multiplicative random walk defined by $Y_0=1$ and $Y_n=\prod_{i=1}^n \xi_i$ for all $n\ge 1$.   For the limiting multiplicative cascade $(Z_{\alpha}(u), u\in\mathcal{U})$, the many-to-one formula reads: for $g: \mathbb{R}^n\to \mathbb{R}$ non-negative and measurable, 
\begin{equation}\label{eq: many-to-one 2}
\mathbb{E}\bigg[\sum_{|u|=n} g(Z_{\alpha}(u_1), Z_{\alpha}(u_2), \ldots, Z_{\alpha}(u_n)) (Z_{\alpha}(u))^{\theta_{\alpha}}\bigg] = \mathbb{E}[g(Y_1, Y_2, \ldots, Y_n)]. 
\end{equation}

\noindent {The proof is analogous to \cref{prop: many to one}. } Note the similarity between equations \eqref{eq: many-to-one} and \eqref{eq: many-to-one 2}. 

\begin{prop}\label{prop: scaling limits S}
As $p\to \infty$, we have
\begin{equation}\label{eq: scaling limits S}
\mathrm{Law}\bigg(\frac{\tt S_1}{p}\bigg|{\tt S}_0 = p\bigg) \longrightarrow \mathrm{Law}(\xi).
\end{equation}
\noindent More generally, for $n\geq 1$, we have 
\begin{equation}\label{eq: scaling limits S joint}
\mathrm{Law}\bigg(\bigg(\frac{\tt S_1}{p}, \frac{\tt S_2}{p}, \ldots, \frac{\tt S_n}{p}\bigg) \bigg|{\tt S}_0 = p\bigg) \longrightarrow \mathrm{Law}(Y_1, Y_2, \ldots, Y_n). 
\end{equation}
\end{prop}

\begin{proof}
 By the definition of ${\tt S}_1$, for any non-negative measurable function $h_p$ such that $h_p(0)=0$,
\begin{equation}\label{proof many-to-one hp}
{\tt E}_p\bigg[h_p\bigg( \frac{{\tt S}_1}{p} \bigg)\bigg] = \bb E^{(p)}\bigg[\sum_{|u|=1} h_p\bigg( \frac{{{{\chi^{(p)}}(u)}}}{p} \bigg) \frac{\overline{V}(\chi^{(p)}(u))}{\overline{V}(p)}\bigg].
\end{equation}

\noindent We first prove \eqref{eq: scaling limits S}. Let $g:\mathbb{R}_+\mapsto \mathbb{R}_+$ be a bounded non-negative continuous function with $g(0)=0$. By \cref{l: passing to infinity} with $h_p=g$ and {the uniform integrability of $\sum_{|u|=1} \big(\frac{\chi^{(p)}(u)}{p}\big)^{\theta_\alpha}$ provided by} \cref{thm: tail biggins}, the right-hand side of \eqref{proof many-to-one hp} converges to $\bb E[\sum_{|u|=1} g\big(Z_\alpha(u)\big)(Z_{\alpha}(u))^{\theta_\alpha}]$ which is $\bb E[g(\xi)]$ indeed. Suppose that we proved \eqref{eq: scaling limits S joint} for $n$, and let us prove it for $n+1$. Let $g:\mathbb{R}_+^{n+1}\to \mathbb{R}_+$ be a bounded continuous function such that $g(0,\ldots)=0$. Define 
\[
h_p(x):= {\tt E}_{px}\bigg[g\bigg(x,\frac{{\tt S}_1}{p}, \frac{{\tt S}_2}{p}, \dots, \frac{{\tt S}_n}{p}\bigg)\bigg].
\]

\noindent Let $(x_p)_{p\ge 1}$ be some sequence such that $x:=\lim_{p\to\infty} x_p$ exists and is positive. Under ${\tt P}_{px_p}$, 
\[
g\bigg(x_p,\frac{{\tt S}_1}{p}, \frac{{\tt S}_2}{p}, \dots, \frac{{\tt S}_n}{p}\bigg), 
\]

\noindent converges in distribution towards $g(x,xY_1,\ldots,xY_n)$ by the induction hypothesis. The dominated convergence theorem implies that $h_p(x_p)\to h(x):=\bb E[g({x}, xY_1,\ldots,xY_n)]$. By the Markov property, 
\[
{\tt E}_p\bigg[g\bigg(\frac{{\tt S}_1}{p}, \frac{{\tt S}_2}{p}, \dots, \frac{{\tt S}_{n+1}}{p}\bigg)\bigg] = {\tt E}_p\bigg[ h_p\bigg( \frac{{\tt S}_1}{p}\bigg)\bigg],
\]

\noindent hence we can apply \cref{l: passing to infinity} again and use \eqref{proof many-to-one hp} {and \cref{thm: tail biggins}} to finish the induction. 
\end{proof}

\medskip
\noindent \textbf{Further comments in the case $n=2$.}
In the case $n=2$, $\ln Y$ is a centred random walk with finite variance. More explicitly, {as seen from \cref{eq: distribution fn n=2} in the appendix, we have}
\begin{equation} \label{eq: density xi}
F(x):= \bb P(\xi\le x)= \frac{2}{\pi} \arctan \sqrt{x},\quad x\ge 0.
\end{equation}

 \noindent Let {$\sigma_1 := \inf\{k\geq 0: Y_k < 1\}$}. Let $R(x)$ be the renewal function of $\ln Y$: 
\begin{equation}\label{def: renewal function R}
R(x) = \mathbb{E}\bigg[\sum^{\sigma_1-1}_{n = 0} \mathds{1}_{\{Y_{n}\leq \mathrm{e}^x\}}\bigg], \quad x\geq 0. 
\end{equation}

\noindent By the duality relation \cite[Chapter XII]{Feller1971}, $R(x)=\sum_{i\ge 0} \bb P(Y_{\nu_i}\le \mathrm{e}^x)$ where $\nu_0:=0$, $\nu_{i+1}:= \inf\{n>\nu_i\,:\, Y_{n}>Y_{\nu_i}\}$ are the strict ascending ladder epochs of $\ln Y$.  By the renewal theorem \cite[Chapter XI]{Feller1971}, there exists a constant $c_0>0$ such that
\begin{equation}\label{def:c0}
\lim_{x\to\infty} \frac{R(x)}{x}=c_0.
\end{equation}

\noindent The renewal function relates to hitting probabilities as follows. Let 
\begin{align}\label{def:sigma}
\sigma_r &:= \inf\{k\ge 0\,:\, Y_k{<} r\}, \\
\sigma^+_r &:=\inf\{k\ge 0\,:\, Y_k {>} r\}. \label{def:sigma+}
\end{align}
 
 \noindent For fixed $b>1$, as $a\to 0^+$, 
\begin{equation}\label{asymptotics hitting Y}
     {\bb P}(\sigma_a <\sigma^+_b) \sim \frac{R(\ln b)}{c_0{|\ln a|}}.
\end{equation}
 
 \noindent We can justify the asymptotics by observing that  $\Big(R(\ln \frac{b}{Y_n})\mathds{1}_{\{ n<\sigma_b^+\}}\Big)$  is a martingale by \cite[Lemma 1]{tanaka} and applying the optional stopping theorem at time $\sigma_a$, using that undershoots are bounded in $L^1$, see \cite[Lemma 5.1.9]{lawler-limic}.  In general, by \cite[Theorem 5.1.7]{lawler-limic}, there exists a constant $c>1$ such that for all $0<a< q < b$,
\begin{equation}\label{hitting Y}
   \frac{1}{c} \frac{1+\ln(b) - \ln(q)}{1+\ln(b)-\ln(a)} \le  {\bb P}_{q}(\sigma_a <\sigma^+_b) \le c \frac{1+\ln(b) - \ln(q)}{1+\ln(b)-\ln(a)}. 
\end{equation}


\subsection{Hitting probabilities in the case $n\in (0,2)$}
\label{sec: hitting n<2}

We suppose in this section that $n\in (0,2)$. 
{
As mentioned in the proof outline in \cref{sec: main result + outline}, the overall strategy to derive our main theorem (\cref{thm: main}) is to provide a tractable classification of the map into bad regions, which are unlikely, and a good region, whose volume is square integrable. Partly because of this square integrability requirement, we will need to place an upper barrier $B$ on the perimeter cascade $(\chi^{(p)}(u), u\in\cal U)$. By the many-to-one formula (\cref{prop: many to one}), related estimates will translate into understanding the hitting probabilities of the Markov chain $\tt S$.
}

{This is the purpose of the present section. Our main objective will be to derive the hitting time estimate in \cref{prop: hitting S infinite}. We start with a preliminary technical result that will be used to derive \cref{prop: hitting S infinite}.} 
For any $M\ge 1$, set $\tt L_M(k):=\sum_{n=0}^k \mathds{1}_{\{\tt S_n \in [1,M]\}}$.
\begin{prop} \label{prop: positive moment estimate S} 
Suppose $n\in (0,2)$. There exist  $\gamma,M,C>0$ and $D\in (0,1)$ such that  for all $p\ge 1$, $k\ge 0$ and $L\ge 1$,
\begin{equation} \label{eq: positive moments bound}
\tt E_p[\tt S_k^{\gamma} \cdot \mathds{1}_{\{\tt L_M(k)\le L \}}]
\le
D^{k}C^L p^{\gamma}.
\end{equation}
\end{prop}

\noindent {The reason for the indicator function in \eqref{eq: positive moments bound} is that we will derive the bound by comparing $\tt S_1/p$ to its limit as $p\to\infty$ (using \cref{prop: scaling limits S}) and using the Markov property. Hence we will need $\tt S$ not to have too many small terms (less than $M$) up to time $k$ to end up with a sensible bound, that we will combine with a crude bound for the small values.}

\begin{proof}

\noindent  For any $\gamma> 0$, by \eqref{eq: many-to-one} and \eqref{eq: mean volume},
\[
{\tt E}_p\bigg[\bigg(\frac{{\tt S}_1}{p}\bigg)^{\gamma}\bigg]
=
\frac{1}{\overline{V}(p)}\mathbb{E}^{(p)}\bigg[\sum_{\abs{u}=1} \bigg(\frac{\chi^{(p)}(u)}{p}\bigg)^{\gamma}\overline{V}(\chi^{(p)}(u)) \bigg]
\le 
C \mathbb{E}^{(p)}\bigg[\sum_{\abs{u}=1} \bigg(\frac{\chi^{(p)}(u)}{p}\bigg)^{\theta_\alpha+\gamma} \bigg].
\]

\noindent Since $(\sum x_i^{\theta_\alpha})^{1+\eta}\ge \sum x_i^{\theta_\alpha(1+\eta)}$, Theorem \ref{thm: tail biggins} implies that $\frac{{\tt S}_1}{p}$ is bounded in $L^\gamma$ for $\gamma \in (0, \theta_{\alpha}\gamma_0)$. {Since $n\in (0, 2)$, we can} take such a $\gamma>0$ small enough so that 
$\phi_\alpha(\theta_\alpha+\gamma)<1$. By \cref{prop: scaling limits S}, $\frac{{\tt S}_1}{p}$ converges in distribution to ${\xi}$. Since $\Big(\frac{{\tt S}_1}{p}\Big)^\gamma$ is uniformly integrable,  its mean converges to  ${\tt E}[{\xi^{\gamma}}]= \phi_\alpha(\theta_\alpha+\gamma)<1$.  Then there exist  two constants $M:=M(\gamma)>0$ and $D:=D(\gamma)\in(0,1)$ such that, for all $p> M$,
\[
\tt E_p[\tt S_1^{\gamma}] \le D p^{\gamma}.
\]
Let $c:=\sup_{p\in [1,M]} \tt E_p[\tt S_1^{\gamma}]p^{-{\gamma}}$ and $C:=\max\Big(\frac{c}{D},1\Big)$. Consider the process $k\mapsto \tt S_k^{\gamma} D^{-k} C^{-\tt L_M(k-1)}$ using the convention that ${\tt L_M(-1)}=0$. It is a non-negative supermartingale. This can be deduced from the Markov property of $\tt S$ since for all $p\ge 0$, 
\[
\tt E_p[\tt S_1^{\gamma} ] \le D C^{\tt L_M(0)}p^{\gamma}.
\]
\noindent  For any $k\ge 0$ and $L\ge 1$,
\[
\tt E_p[\tt S_k^{\gamma} D^{-k} C^{-\tt L_M(k-1)} \cdot \mathds{1}_{\{\tt L_M(k)\le  L \}}]
\le \tt E_p[\tt S_k^{\gamma} D^{-k} C^{-\tt L_M(k-1)}] \le p^{\gamma}.
\]

\noindent Notice that $C^{-\tt L_M(k-1)} \ge C^{-L}$ on the event $\{\tt L_M(k)\le  L \}$ to complete the proof. 
\end{proof}
 

\noindent For $r\ge 1$, let ${\tt T}_r^+:=\inf\{n\ge 0\,:\, {\tt S}_n >r\}$. 
{The hitting time estimate we are after is the following. Note that since the branching Markov chain $(\chi^{(p)}(u), u\in\cal U)$ is of order $p$ under $\bb P^{(p)}$ (by \cref{thm: CCM convergence cascade}), we will later need to take $B=bp$ so that the limiting multiplicative cascade sees a barrier at height $b$. This will leave a good enough upper bound in \eqref{eq: decay probability} to get rid of the barrier as $b\to\infty$.}
\begin{prop} \label{prop: hitting S infinite} 
Suppose $n\in (0,2)$. There exist $\gamma>0$ and $C>0$ such that for all $B\ge p\ge 1$,
\begin{equation} \label{eq: decay probability}
\tt P_p( {\tt T}_{B}^+<\infty)
\le
C \bigg(\frac{B}{p}\bigg)^{-\gamma}.
\end{equation}
\end{prop}
\begin{proof}
Take $\gamma$, $M$, $C$ and $D$ as in \cref{prop: positive moment estimate S}.  Let $L\ge 1$ (its value will be fixed at the end of the proof). Let $\tt L_M:=\lim_{k\to\infty} \tt L_M(k)$. We use the upper bound
\begin{equation}\label{eq: proof decay probability}
\tt P_p( {\tt T}_{B}^+<\infty)
\le
\tt P_p(\tt L_M > L) + 
B^{-\gamma} \tt E_p\bigg[\sum_{k=0}^\infty  {\tt S}_k^{\gamma} \mathds{1}_{\{ {\tt L}_M\le L\}}\bigg].
\end{equation}

\noindent For the first term, remark that on the event that $\tt L_M>L$, the first $L$ times when $\tt S_n \in[1,M]$, one must have in particular $\tt S_{n+1}\ne 0$ (otherwise $\tt S$ would get absorbed at $0$). Writing $c = \min_{p\in [1,M]} \tt P_p(\tt S_1=0)$, we thus get
\begin{equation}\label{eq: proof bound LM}
\tt P_p(\tt L_M > L)\le (1-c)^L.
\end{equation}

\noindent  On the other hand, \cref{prop: positive moment estimate S}  implies that
\[
\tt E_p\bigg[\sum_{k=0}^\infty  {\tt S}_k^{\gamma} \mathds{1}_{\{ {\tt L}_M\le L\}}\bigg]
\leq
p^{\gamma} C^L \sum_{k=0}^\infty D^k = p^{\gamma} C^L \frac{1}{1-D}  .
\]

\noindent 
\noindent Take $L=\delta \ln \frac{B}{p}$ with $\delta=\delta( \gamma)>0$ small enough and use \eqref{eq: proof bound LM} in \eqref{eq: proof decay probability} to complete the proof.
\end{proof}

\subsection{Hitting probabilities in the case $n=2$}
\label{sec:hitting n=2}

{
For the same reasons as in \cref{sec: hitting n<2}, we will also be after hitting time estimates in the case $n=2$. In fact, introducing an upper barrier in this case will also be crucial for another reason: since this case corresponds to the boundary case of branching random walk (\cref{sec:bdry case derivative}), we will need to apply a truncation argument. We will see that the perimeter cascade $(\chi^{(p)}(u), u\in\cal U)$ \emph{feels} the effect of the barrier, causing the presence of extra logarithmic corrections.
}

{
The main result of this section is the following.
Recall that ${\tt T}_M$ is the hitting time of $[0,M)$ by ${\tt S}$ and ${\tt T}_r^+$ the hitting time of $(r,\infty)$ by ${\tt S}$. 
{In view of the asymptotics \eqref{def:c0} of the renewal function, one should interpret \eqref{asymptotics hitting S} as a glimpse of the fact that the cascade will feel the barrier as $b\to\infty$.}
\begin{prop}\label{prop: hitting time S}
Suppose $n=2$. There exists a constant $c>1$ such that for all $M\ge 1$, $B\ge  M$ and $p\in {[1,B]}$, 
\begin{equation}\label{eq hitting time S upper bound}
     {\tt P}_p({\tt T}_M < {\tt T}^+_{B}) \leq c\frac{1+\ln B - \ln p}{1+ \ln B -\ln M}.
\end{equation}

\noindent Let $b>1$ and $M\ge 1$ and recall the constant $c_0$ defined in \eqref{def:c0}. For any sequence $(x_p)_p$ with limit $x<b$, 
\begin{equation}\label{asymptotics hitting S}
  \lim_{p\to\infty} \ln(p)\,  {\tt P}_{x_p p}({\tt T}_M < {\tt T}^+_{bp}) = \frac{R(\ln(b/x))}{c_0}.
\end{equation}
\end{prop}
}

{The result will be proved at the end of the section.}
{Unsurprisingly, hitting time} estimates in the critical $O(2)$ case are more delicate. 
{Intuitively, the reason is that $\xi$ is centred in this case (recall \eqref{def: random variable xi} and the discussion there), so that $(\chi^{(p)}(u), |u|=k)$ does not \textit{a priori} decay with $k$.}
{To understand more precisely if and how it decays, our key idea will be to} use a coupling with the limiting multiplicative random walk $Y$. {We will first present nice properties of this coupling (namely \cref{lem: E_M} and \cref{coupling 2}), in order to complete the proof of \cref{prop: hitting time S}.}

Recall that $(\xi_i)_{i\ge 1}$ are i.i.d. random variables with distribution given by \eqref{eq: density xi}. Under ${\tt P}_p$, we let  $Y_0=p$ and $Y_n= p \prod_{i=1}^n\xi_i$ for $n\ge 1$. By Lemma \ref{l:convergence speed S1}, there exist  constants $C>0$ and $p_0\ge 1$ such that for all $p\ge p_0$,  and $q\ge 0$, 
\begin{equation}\label{convergence speed S1}
    \abs{{\tt P}_p({\tt S_1} \leq q) - F\Big(\frac{q}{p}\Big)} \leq \frac{C}{\sqrt{p}}.
\end{equation}

\noindent We define our coupling of $\tt S$ and $Y$ under ${\tt P}_p$ by induction. First let ${\tt S}_0 = Y_0 = p$, and denote by $F_r$ the distribution function of ${\tt S}_1/r$ conditioned on ${\tt S}_0 = r$, and $F_r^{-1}$ its generalised inverse. If ${\tt S}_{n-1} = r \ne 0$, we define ${\tt S}_n$ by
\begin{equation}
{\tt S}_n = rF_r^{-1}(F(\xi_n)).
\end{equation}
Then by definition {of ${\tt S}_n$ and monotonicity of $F_r$}, {$F_{r}(({\tt S}_n-1)_+/r) \leq F(\xi_n) \leq F_r({\tt S}_n/r)$}. From the exact formula of $F(x)$ for \eqref{eq: density xi}, there exists a constant $C'>0$ such that for $q\in\mathbb{N}$ and $r\in\mathbb{N}^*$, $|F(q/r) - {F((q-1)_+/r)}| \leq C'/\sqrt{r}$. {Combining this with the bound} \eqref{convergence speed S1} with $p=r$ and $q={\tt S}_n$ or {${\tt S_n}-1$}, we have when $r\ge p_0$, on $\{{\tt S}_{n-1} = r\}$, 
{\[F\bigg(\frac{{\tt S}_n}{r}\bigg) - \frac{C+C'}{\sqrt{r}} \leq  F\bigg(\frac{({\tt S}_n-1)_+}{r}\bigg) - \frac{C}{\sqrt{r}} \leq F(\xi_n) \leq F\bigg(\frac{{\tt S}_n}{r}\bigg) + \frac{C}{\sqrt{r}}. \]
}
\noindent Substituting ${\tt S}_{n-1}$ for $r$, we conclude that there exists a constant $C>0$ such that
 
\begin{equation}\label{main lemma 2}
\left|F\bigg(\frac{Y_n}{Y_{n-1}}\bigg) - F\bigg(\frac{{\tt S}_n}{{\tt S}_{n-1}}\bigg)\right|\leq \frac{C}{\sqrt{{\tt S}_{n-1}}} \quad \text{a.s.},
\end{equation}

\noindent on the event $\{{\tt S}_{n-1} \ge p_0\}$. 

{The following lemma introduces a good event $E_M$ on which $\tt S$ and $Y$ are within constants under the above coupling.}
Recall the definition of $\sigma_r$ in \eqref{def:sigma}. 

\begin{lem} \label{lem: E_M}
    Suppose that $n=2$, and fix $a>1$. For $M>0$, let
    \[
    E_M := \left\{\forall\, 0 \leq k<\sigma_M, \; Y_k^{\frac{7}{8}}\leq Y_{k+1}\leq Y_k^{\frac{9}{8}} \right\} \bigcap \left\{\prod_{k=0}^{\sigma_M-1} (1+  Y_k^{-\frac{3}{16}})\le \sqrt{a} \right\}.
    \]
    There exists $p_1\ge 1$ such that for all $M\ge p_1$, on the event $E_M$, we have $\frac{1}{a}\leq \frac{{\tt S}_k}{Y_k} \leq a$  for all $k\le \sigma_M$.
\end{lem}

\begin{proof}
We first choose $p_1$. {Let $p_0$ and $C>0$ be as in \eqref{convergence speed S1}.} For the sake of the below arguments, we will need to take $p_1 \ge 1$ large enough so that: (i) $p_1\ge a p_0$, (ii) $4\sqrt{a}C p_1^{-7/16} \le 1$, (iii) $8 \sqrt{a}C p_1^{-1/8} \le 1$, and (iv) $\tan\big(\frac{\pi}{2}x) \le 2x$ for all $x<\sqrt{a}Cp_1^{-1/2}$. The constant $p_1$ being set, we now take $M\ge p_1$ and prove the claim by induction on $k$ (up to $\sigma_M$). Suppose that $k<\sigma_M$ and ${\tt S}_{\ell}/Y_{\ell} \in [1/a, a]$ for all $\ell\le k$. For $\ell\le k$, let 
\[
f_\ell:= \abs{F\bigg(\frac{{\tt S}_{\ell+1}}{{\tt S}_\ell}\bigg) - F\bigg(\frac{Y_{\ell+1}}{Y_\ell}\bigg) }.
\]

\noindent Setting $s_\ell=\sqrt{\frac{{\tt S}_{\ell+1}}{{\tt S}_{\ell}}}$ and $y_\ell=\sqrt{\frac{Y_{\ell+1}}{Y_\ell}}$, $f_\ell = f_\ell(s_\ell,y_\ell)$ where
\begin{equation}\label{eq:Fsk}
    f_\ell(s_\ell,y_\ell) =\left|\frac{2}{\pi}\arctan s_\ell - \frac{2}{\pi}\arctan y_\ell\right| = \frac{2}{\pi}\arctan \left|\frac{s_\ell- y_\ell}{1+s_\ell y_\ell} \right|. 
\end{equation}

\noindent  Notice that ${\tt S}_\ell  \ge \frac{Y_\ell}{a} \stackrel{\text{(i)}}{\ge} p_0$ since $\ell<\sigma_M$. In view of \eqref{main lemma 2}, $f_\ell \le \frac{C}{\sqrt{{\tt S}_\ell}}\leq \frac{\sqrt{a}C}{\sqrt{Y_\ell}}$. Besides, on the event $E_M$, $y_\ell=\sqrt{\frac{Y_{\ell+1}}{Y_\ell}} \le Y_\ell^{\frac{1}{16}}$ and $\frac{1}{y_\ell} =\sqrt{\frac{Y_{\ell}}{Y_{\ell+1}}} \le Y_\ell^{\frac{1}{16}}$. We deduce that $\tan \big(\frac{\pi}{2} f_\ell\big) \stackrel{\text{(iv)}}{\le} 2f_\ell$, $ f_\ell \widetilde y_\ell \stackrel{\text{(ii)}}{\le} \frac{1}{4}$,  $f_\ell \widetilde y_\ell^2 \stackrel{\text{(iii)}}{\le}  \frac{1}{8}Y_{\ell}^{-1/4}$ and $f_\ell \stackrel{\text{(iii)}}{\le} \frac{1}{8}Y_\ell^{-1/4}$, where $\widetilde y_\ell$ is either $y_\ell$ or $\frac{1}{y_\ell}$. Equation \eqref{eq:Fsk} then yields
\[
|s_\ell-y_\ell|= (1+s_\ell y_\ell) \tan \Big(\frac{\pi}{2} f_\ell\Big)   \le 2 f_\ell (1+s_\ell y_\ell) \le 2 f_\ell (1+ y_\ell^2+ |s_\ell-y_\ell|y_\ell), 
\]
hence
\[
 |s_\ell-y_\ell| \le \frac{2 f_\ell (1+ y_\ell^2)}{1-2f_\ell y_\ell}\le 4f_\ell (1+ y_\ell^2) \le Y_\ell^{-1/4}.
\]

 \noindent We proved that
 \[
      \left|\sqrt{\frac{{\tt S}_{\ell+1}}{{\tt S}_\ell}} - \sqrt{\frac{Y_{\ell+1}}{Y_\ell}} \right| \le Y_\ell^{-1/4}.
 \]

 \noindent Multiplying both sides by $\sqrt{\frac{{\tt S}_\ell}{Y_{\ell+1}}}= \sqrt{\frac{{\tt S}_\ell}{Y_{\ell}}} \sqrt{\frac{Y_\ell}{Y_{\ell+1}}}\le \sqrt{\frac{{\tt S}_\ell}{Y_{\ell}}} Y_{\ell}^{1/16} $, we get
 \[
 \left|\sqrt{\frac{{\tt S}_{\ell+1}}{Y_{\ell+1}}} - \sqrt{\frac{{\tt S}_\ell}{Y_\ell}} \right| \le Y_\ell^{-3/16} \sqrt{\frac{{\tt S}_\ell}{Y_\ell}}.
 \]

\noindent The above chain of arguments is valid for all $\ell\le k$. In particular, by induction,
\[ \sqrt{\frac{{\tt S}_{k+1}}{Y_{k+1}}}\leq \sqrt{\frac{{\tt S}_k}{Y_k}}(1 + Y_k^{-3/16})\leq \ldots \leq \prod^{k}_{n=0}(1 + Y_{n}^{-3/16})\leq \sqrt{a},\]

\noindent on $E_M$. This proves $\tt S_{k+1}/Y_{k+1} \le a$. For the other inequality, we first note that $f_\ell(s_\ell,y_\ell) = f_\ell(1/s_\ell,1/y_\ell)$. Substituting $\frac{1}{s_k}$, \textit{resp.}\ $\frac{1}{y_k}$, for $s_k$, \textit{resp.}\ $y_k$, in \eqref{eq:Fsk} and running the above argument therefore yields
 \[
 \left|\sqrt{\frac{{\tt S}_{\ell}}{{\tt S}_{\ell+1}}} - \sqrt{\frac{Y_{\ell}}{Y_{\ell+1}}} \right| \le Y_\ell^{-1/4}.
 \]

\noindent Multiplying it by $\sqrt{\frac{Y_{\ell+1}}{{\tt S}_{\ell}}}\le \sqrt{\frac{Y_{\ell}}{{\tt S}_\ell}} Y_{\ell}^{1/16}$ leads to the inequality $\sqrt{\frac{Y_{k+1}}{{\tt S}_{k+1}}}\leq \sqrt{a}$. Taking the squares completes the induction and the proof of the theorem.
\end{proof}

{We can now state our main coupling result, showing that $\tt S$ and $Y$ are within constants with high probability.}
\begin{thm}\label{coupling 2}
Suppose that $n=2$, and fix $a>1$.   There {exist} $c_1>0$ and $p_1>0$ such that for all $p \ge M\ge p_1$,
\[
{\tt P}_p\left(\forall\, 0\le k\le \sigma_M,\, \frac{1}{a}\leq \frac{{\tt S}_k}{Y_k} \leq a\right)\ge 1 - c_1 M^{-\frac{1}{16}}.
\]
\end{thm}

\begin{proof}
In the proof, the letter $c$ denotes a positive constant whose value can change from line to line. By \cref{lem: E_M}, we only need to see that there exists $c_1>0$ such that $E_M$ occurs with probability at least $1-c_1 M^{-1/16}$ whenever $p\ge M\ge p_1$. Let $E_1$ be the event that for all $k<\sigma_M$, we have $Y_k^{\frac{7}{8}}\leq Y_{k+1}\leq Y_k^{\frac{9}{8}}$. By a union bound, discussing on the time $k$ when the condition is violated, we find
\[
{\tt P}_p(E_1^\mathrm{c})\leq  {\tt E}_p\bigg[\sum^{\sigma_M - 1}_{k=0} \mathds{1}_{\big\{ \frac{Y_{k+1}}{Y_k} \notin [Y_k^{-1/8}, Y_k^{1/8}] \big\}}\bigg].
\]

\noindent By conditioning on $Y_k$ and using the Markov property,
\[
  {\tt E}_p\bigg[\sum^{\sigma_M - 1}_{k=0} \mathds{1}_{\big\{\frac{Y_{k+1}}{Y_k} \notin [Y_k^{-1/8}, Y_k^{1/8}] \big\}}\bigg]
 =  {\tt E}_p\bigg[\sum^{\sigma_M - 1}_{k=0}{\tt P}(\xi \notin [r^{-1/8},r^{1/8}])_{\big| r=Y_k} \bigg]. 
\]

\noindent From the explicit distribution \eqref{eq: density xi} of $\xi$, we see that ${\tt P}_p( \xi \notin [r^{-1/8},r^{1/8}] )\leq c\,  r^{-1/16}$, hence 
\[
{\tt P}_p(E_1^\mathrm{c})\leq c\, {\tt E}_p \bigg[\sum^{\sigma_M - 1}_{k=0}Y_k^{-\frac{1}{16}}\bigg].
\] 

\noindent The last expectation is smaller than $c\, M ^{-\frac{1}{16}}$, see for example Lemma A.1 in \cite{AHS24}.  On the other hand, 
\[{\tt E}_p\bigg[\sum^{\sigma_M-1}_{k=0} \ln(1 +  Y_k^{-\frac{3}{16}})\bigg] \leq  {\tt E}_p\bigg[\sum^{\sigma_M-1}_{k=0} Y_k^{-\frac{3}{16}}\bigg]
\leq c\,  M^{-\frac{3}{16}},\]

\noindent by the same result. By Markov's inequality, we deduce that for all $p\ge M\ge 1$, the event $E_M$ of \cref{lem: E_M} has probability greater  than $1- c\, M ^{-\frac{1}{16}}$. \cref{lem: E_M} concludes the proof.
\end{proof}

{Using this coupling, we may finally derive the hitting time estimates in \cref{prop: hitting time S}.}
\begin{proof}[Proof of \cref{prop: hitting time S}] 
We first prove \eqref{eq hitting time S upper bound}. 
Let $M\ge 1$. We begin by fixing some constants: our precise choice comes from the arguments below. We fix ${\rho}>2$ (large enough) and ${a_0}>0$ (small enough) so that: (i) $1+\ln({\rho})> p_1^{1/16}$ with $p_1$ as in \cref{coupling 2} {with $a=2$ there}, (ii) ${a_0} \le \inf_{x\ge \ln 2} \frac{1-\ln 2 +x}{1+x}$, (iii) $1+\ln(x)<x^{1/16}$ for all $x\ge {\rho}$, and (iv) for all $x\ge {\rho}$,
\[
1+\ln(x)\ge {a_0} (1+\ln(x)) + 16 \ln(1+\ln(x)).
\]
We will first prove \eqref{eq hitting time S upper bound} when $B\ge {\rho}M$. Consider such a $B$, and define $r_0\ge 1$ by
\begin{equation}\label{proof hitting S choice r}
    r_0^{-\frac{1}{16}}= \frac{1}{1+\ln(B/M)},
\end{equation}
\textit{i.e.} \ $\ln(r_0) = 16 \ln(1+\ln(B/M))$.  Since $B\ge {\rho}M$ and by (i) above, we have $r_0\ge p_1$ and we may apply \cref{coupling 2} to $r_0$. Let 
{\[\tau:=\inf\bigg\{k\ge 0\,:\, \frac{{\tt S}_k}{Y_k}\notin \Big[\frac{1}{2},2\Big]\bigg\}.\]} 
Then {there exists $c_1 > 0$ such that} for all $ p \ge r_0$,
\begin{equation}\label{proof hitting S coupling}
    {\tt P}_p(\tau \le \sigma_{r_0}) \le c_1 r_0^{-\frac{1}{16}},
\end{equation}

\noindent hence
\begin{equation}\label{proof hitting S first bound}
{\tt P}_p({\tt T}_M< {\tt T}^+_{B})
 \le 
 c_1 r_0^{-\frac{1}{16}}
+
{\tt P}_p( {\tt T}_M < {\tt T}^+_{B},\, \tau >  \sigma_{r_0})
=
 \frac{c_1}{1+\ln(B/M)}
+
{\tt P}_p( {\tt T}_M < {\tt T}^+_{B},\, \tau >  \sigma_{r_0}). 
\end{equation}

\noindent Now let $r_1:=\max(2M,r_0)$. Notice that $r_1 \le B$ since ${\rho}>2$ and 
\[
r_0^{-\frac{1}{16}}=\frac{1}{1+\ln(B/M)} \ge \frac{1}{1+\ln(B)} \stackrel{\text{(iii)}}{>} B^{-\frac{1}{16}}.
\]
Besides, observe that 
\[
1+\ln(B)-\ln(r_0) \ge  1+\ln(B/M)-\ln(r_0) \overset{\text{(iv)}}{\ge} {a_0} (1+\ln (B/M)),
\]
and also
\[
   1+\ln(B)-\ln(2M) 
\stackrel{\text{(ii)}}{\ge} 
{a_0} (1+\ln (B/M)).
\]
Therefore 
\begin{equation}\label{proof hitting S bound r}
1+\ln(B)-\ln(r_1) \ge {a_0} (1+\ln (B/M)),
\end{equation}

\noindent  We first take an upper bound in \eqref{proof hitting S first bound}: for all $p\in [r_1,B]$,
\begin{equation}
{\tt P}_p({\tt T}_M< {\tt T}^+_{B})
\le 
\frac{c_1}{1+\ln(B/M)} 
+
{\tt P}_p( {\tt T}_M < {\tt T}^+_{B},\, \tau >  \sigma_{r_1}). \label{proof hitting S upper bound}
\end{equation}

\noindent  Since $r_1\ge 2M$, on the event $\{\tau>\sigma_{r_1}\}$, ${\tt T}_M \ge \sigma_{r_1}$ and $\{ \sigma_{r_1} <{\tt T}_B^+\} \subset \{ \sigma_{r_1}< \sigma^+_{2B} \}$. We get for all $p\in [r_1,B]$
\[
{\tt P}_p( {\tt T}_M < {\tt T}^+_{B},\, \tau >  \sigma_{r_1}) 
      \le  {\tt P}_p( \sigma_{r_1} < {\tt T}^+_{B},\, \tau>\sigma_{r_1}) 
     \le  {\tt P}_p( \sigma_{r_1} < \sigma_{2B}^+), 
\]

\noindent which by \eqref{hitting Y} is smaller than a constant times
\[
 \frac{1+\ln(2B)-\ln(p)}{1+\ln(2B)-\ln(r_1)} 
 \le c \frac{1+\ln(B)-\ln(p)}{1+\ln(B)-\ln(r_1)}
     \overset{\eqref{proof hitting S bound r}}{\le} \frac{c}{{a_0}}  \frac{1+\ln(B)-\ln(p)}{1+\ln(B)-\ln(M)}.
\]

 \noindent  Plugging this inequality into \eqref{proof hitting S upper bound}, we proved that for any $M\ge 1$, $B\ge {\rho} M$ and  $p\in [r_1,B]$, 
\[
{\tt P}_p({\tt T}_M< {\tt T}^+_{B})
 \le 
c \frac{1+\ln(B)-\ln(p)}{1+\ln(B)-\ln(M)} .
\]

\noindent This bound stays true (for some bigger constant $c$ if necessary) when $p\le r_1$ by \eqref{proof hitting S bound r}, since in that case
\[
1+\ln(B)-\ln(p) \ge 1+\ln(B)-\ln(r_1) \ge {a_0} (1+\ln(B/M)).
\]
Finally, it also stays true if $B/M \in[1,{\rho}]$ by taking $c>1+\ln({\rho})$. This concludes the proof of \eqref{eq hitting time S upper bound}.

We turn to the proof of \eqref{asymptotics hitting S}. Fix $b>1$ and $M\ge 1$. Let $r_0$ be defined as in \eqref{proof hitting S choice r} with $B$ replaced with $bp$. Notice that $r_0\to \infty$ as $p\to\infty$. Let $K\ge 1$ and $a>1$. Define $\tau_a:=\inf\big\{k\ge 0\,:\, \frac{{\tt S}_k}{Y_k}\notin [\frac{1}{a},a]\big\}$. By \cref{coupling 2}, we have for $p$ large enough,
\[
    {\tt P_{x_p p}}({\tt T}_M< {\tt T}_{bp}^+)  \le c_1 (Kr_0)^{-\frac{1}{16}} + {\tt P_{x_pp}}({\tt T}_M< {\tt T}_{bp}^+,\, \tau_a> \sigma_{Kr_0}).
\]
Now, on the event $\{\tau_a>\sigma_{Kr_0}\}$, we have $\tt T_M > \sigma_{Kr_0}$ (for large enough $p$), and $\{\sigma_{Kr_0} < \tt T^+_{bp} \} \subset \{\sigma_{Kr_0} < \tt \sigma^+_{abp} \}$. Hence
\[
    {\tt P_{x_p p}}({\tt T}_M< {\tt T}_{bp}^+)
    \le 
     c_1 (Kr_0)^{-\frac{1}{16}} + {\tt P_{x_p p}}(\sigma_{Kr_0}< \sigma^{+}_{abp}).
\]

\noindent By \eqref{asymptotics hitting Y} and the continuity of $R$, as $p\to \infty$, ${\tt P}_{x_p p}(\sigma_{Kr_0}< \sigma^{+}_{a b p}) = {\tt P}_1(\sigma_{Kr_0/(x_p p)}< \sigma^{+}_{a b/x_p}) \sim \frac{R(\ln(ab/x))}{c_0\ln(p)}$. Moreover, \eqref{proof hitting S choice r} gives $r_0^{-\frac{1}{16}} \sim \frac{1}{\ln(p)}$ as $p\to\infty$. We deduce that
    \[
    \limsup_{p\to\infty} \; \ln(p) \, {\tt P_{x_p p}}({\tt T}_M< {\tt T}_{bp}^+)
    \le c_1K^{-\frac{1}{16}} + \frac{R(\ln(ab/x))}{c_0}.
    \]

\noindent {Sending} $K\to\infty$ then $a\to 1$ (recalling that $R$ is continuous) yields the upper bound. The lower bound is proved similarly, since for large enough $p$,
\begin{align*}
    {\tt P_{x_p p}}({\tt T}_M< {\tt T}_{bp}^+) & \ge 
    {\tt P_{x_p p}}({\tt T}_M< {\tt T}_{bp}^+,\, \tau_a > \sigma_{Kr_0})\\
    &\ge
    {\tt P_{x_p p}}({\sigma_{Ma}} < \sigma^+_{b p/a} ,\, \tau_a > \sigma_{Kr_0}) \\
    &\ge 
      {\tt P_{x_p p}}({\sigma_{Ma}} < \sigma^+_{b p/a}) -c_1 (Kr_0)^{-\frac{1}{16}},
\end{align*}
by \cref{coupling 2}. We conclude as for the upper bound.
\end{proof}

\subsection{Green function}
{Part of our arguments for controlling the volume of \emph{good} or \emph{bad} regions in our classification of \cref{sec: classification} will rely on moment estimates on the volume inside various types of loops. Through the many-to-one formula (\cref{prop: many to one}), this will translate into moment estimates along the spine, confined between two barriers (\cref{c:sum moments}). Again, we stress that the constants $M$ and $B$ in the statements should be thought of as lower and upper barriers for the discrete perimeter cascade. The need for these barriers will become clearer as we move towards the classification, see \textit{e.g.}\ the discussion preceding \cref{def: bad vertices}.} 

{The moment estimates will be presented in \cref{c:sum moments}. They are based on the following Green function estimates.} 

\begin{prop}\label{p:green}
(i) Suppose $n\in(0,2)$. There exists a constant $c>0$ such that for all $t\ge 0$ and $p\ge 1$, 
\begin{equation} \label{eq:green_O_n}
{\tt E}_p\bigg[\sum_{n=0}^\infty \mathds{1}_{\{ {\tt S}_n\in [\mathrm{e}^t,\mathrm{e}^{t+1}]\}}\bigg]\le c.
\end{equation}
(ii) Suppose $n=2$.  There exists a constant $c>0$ such that for all $M\ge 1$, $B\ge M$, $p\in [M,B]$ and $t\in [\ln(M),\ln(B)]$,  
\begin{equation} \label{eq:green_O_2}
{\tt E}_p\bigg[\sum_{n=0}^{{\tt T}_M\land {\tt T}_B^+} \mathds{1}_{\{ {\tt S}_n\in [\mathrm{e}^t,\mathrm{e}^{t+1}]\}}\bigg]\le c (1+\min(\ln(B) - t ,t -\ln(M))) \min\bigg(\frac{2+\ln(B)-\ln(p)}{1+\ln(B)- t}, 1\bigg).
\end{equation}
\end{prop}

\begin{proof}
(i) 
We first prove that there exists a constant $c\in(0,1)$ such that for all $x\ge 1$, ${\tt P}_x({\tt S}_k <x/\mathrm{e},\, \forall\, k\ge 1) \ge c$. 
We only need to show it for $x$ large enough. By Proposition \ref{prop: hitting S infinite}, there exists $b>1$ such that for all $p\ge 1$, ${\tt P}_p({\tt T}_{bp}^+=\infty) \ge \frac{1}{2}$. By the Markov property at time $1$, for $x\ge b \mathrm{e}$, 
\[
{\tt P}_x({\tt S}_k <x/\mathrm{e},\, \forall\, k\ge 1) 
\ge 
\tt E_x \big[\mathds{1}_{\{1\le \tt S_1 \le \frac{x}{b \mathrm{e}}\}} {\tt P}_{\tt S_1}({\tt S}_k <x/\mathrm{e},\, \forall\, k\ge 1)\big]
=
\tt E_x \big[\mathds{1}_{\{1\le \tt S_1 \le \frac{x}{b\mathrm{e}}\}} {\tt P}_{\tt S_1}(\tt T^+_{x/\mathrm{e}}=\infty)\big].
\]
Now on the event $\{1\le \tt S_1 \le \frac{x}{b\mathrm{e}}\}$, we have ${\tt P}_{\tt S_1}(\tt T^+_{x/\mathrm{e}}=\infty) \ge {\tt P}_{\tt S_1}(\tt T^+_{b \tt S_1}=\infty)\ge \frac12$ by our choice of $b$.
Thus ${\tt P}_x({\tt S}_k <x/\mathrm{e},\, \forall\, k\ge 1) \ge \frac{1}{2}{\tt P}_x({\tt S}_1\in [1,x/b\mathrm{e}])$ while ${\tt P}_x({\tt S}_1\in [1,x/b\mathrm{e}])$ has a positive limit as $x\to\infty$ since $\frac{\tt S_1}{x}$ converges in distribution under $\tt P_x$ by \cref{prop: scaling limits S}.

Let $N_t := \sum_{n=0}^\infty \mathds{1}_{\{ {\tt S}_n\in [\mathrm{e}^t,\mathrm{e}^{t+1}]\}}$, and $T_t^{k}$ defined recursively by $T_t^0 := \inf\{n\geq 0 : \; {\tt S}_n \in [\mathrm{e}^t, \mathrm{e}^{t+1}]\}$ and $T_t^{k+1}:=\inf\{n>T_t^k : \; {\tt S}_n \in [\mathrm{e}^t,\mathrm{e}^{t+1}] \}$. By the Markov property, the above discussion entails that for $k\ge 1$,
\begin{multline*}
{\tt P}_p(N_t > k) 
= {\tt P}_p(T_t^{k} <\infty) 
= {\tt E}_p\big[\mathds{1}_{\{T_t^{k-1} <\infty\}} {\tt P}_{\mathtt{S}(T_t^{k-1})}(T_t^1<\infty) \big] \\
\le (1-c) {\tt P}_p(T_t^{k-1} <\infty)
=(1-c) {\tt P}_p(N_t > k-1) , 
\end{multline*}
since $T_t^{1}=\infty$ on the event $\{{\tt S}_k <x/\mathrm{e},\, \forall\, k\ge 1 \}$, when $\tt S_0 = x\in [\mathrm{e}^t,\mathrm{e}^{t+1}]$. This proves that ${\tt P}_p(N_t > k) \le (1-c)^k$ for all $k\ge 0$, and hence \eqref{eq:green_O_n} since $\tt E_p [N_t] = \sum_{k\ge 0} {\tt P}_p(N_t > k)$.

(ii) We claim that it is enough to prove that there exists a constant $c>0$ such that for $B\ge M\ge 1$ and $t\in[\ln M,\ln B]$, 
\begin{equation}\label{eq: proof green function}
\sup_{p\in [\mathrm{e}^t,\mathrm{e}^{t+1}]}{\tt E}_p\bigg[\sum_{n=0}^{{\tt T_M}\land {\tt T}_B^+} \mathds{1}_{\{ {\tt S}_n\in  [\mathrm{e}^t,\mathrm{e}^{t+1}]   \}}\bigg]\le c (1+\min(\ln(B) - t ,t -\ln(M))).
\end{equation}

\noindent Indeed, by the Markov property at $T^0_t:= \inf\{n\geq 0 : \; {\tt S}_n \in [\mathrm{e}^t, \mathrm{e}^{t+1}]\}$ we would then have
\begin{multline*}
{\tt E}_p\bigg[\sum_{n=0}^{{\tt T_M}\land {\tt T}_B^+} \mathds{1}_{\{ {\tt S}_n\in [\mathrm{e}^t,\mathrm{e}^{t+1}]\}}\bigg] = {\tt E_p}\Bigg[\mathds{1}_{\{T^0_t \leq {\tt T}_M\land {\tt T}_B^+\}} {\tt E}_{\tt S(T_t^0)} \bigg[\sum_{n=0}^{{\tt T}_M\land {\tt T}_B^+} \mathds{1}_{\{ {\tt S}_n\in [\mathrm{e}^t,\mathrm{e}^{t+1}]\}}\bigg]\Bigg]\\
\stackrel{\eqref{eq: proof green function}}{\leq} c(1+ \min(\ln(B) - t ,t -\ln(M))){\tt P}_p(T^0_t \leq {\tt T}_M\land {\tt T}_B^+).
\end{multline*} 

\noindent When $\mathrm{e}^{t+1} \leq p \leq B$, \cref{prop: hitting time S} gives that 
\[{\tt P}_p(T^0_t \leq {\tt T}_M\land {\tt T}_B^+) \leq {\tt P}_p({\tt T}_{\mathrm{e}^{t+1}}\leq {\tt T}_B^+)\leq c\frac{1 + \ln B - \ln p}{\ln B - t}\leq c\frac{2 + \ln B - \ln p}{1 + \ln B - t}. \]

\noindent When $p\leq \mathrm{e}^{t+1}$ (in particular when $\mathrm{e}^{t+1} > B$), we note that the above ratio is larger than $1$, and so we can bound by the minimum with $1$ in any case. Thus we have proved \eqref{eq:green_O_2} modulo \eqref{eq: proof green function}. 

We turn to proving \eqref{eq: proof green function}. Fix $B,M$ and $t$ accordingly. Let $e(t)$ denote the left-hand side of \eqref{eq: proof green function} and $\tau:=\inf\{k\ge 0\,:\, \frac{{\tt S}_k}{Y_k}\notin [\frac{1}{2}, 2]\}$. By the Markov property at the first return time to $ [\mathrm{e}^t,\mathrm{e}^{t+1}]$ after $\tau$, 
\begin{equation}\label{eq: proof green function tau1}
{\tt E}_p\bigg[\sum_{n=0}^{{\tt T_M}\land {\tt T}_B^+} \mathds{1}_{\{ {\tt S}_n\in  [\mathrm{e}^t,\mathrm{e}^{t+1}]   \}}\bigg]
\le 
{\tt E}_p\bigg[\sum_{n=0}^{{\tt T_M}\land {\tt T}_B^+\land \tau} \mathds{1}_{\{ {\tt S}_n\in  [\mathrm{e}^t,\mathrm{e}^{t+1}]   \}}\bigg]
+ {\tt P}_p(\tau< {\tt T_M}\land {\tt T}_B^+) e(t).
\end{equation}

\noindent We recall from \cref{coupling 2} that for $M\ge p_1$, 
\begin{equation} \label{eq: e(t) tau_1}
    {\tt P}_p(\tau< {\tt T_M}\land {\tt T}_B^+) \le {\tt P}_p(\tau< {\tt T_M})\le c M^{-\frac{1}{16}}.
\end{equation}
On the other hand, 
\[
{\tt E}_p\bigg[\sum_{n=0}^{{\tt T_M}\land {\tt T}_B^+\land \tau} \mathds{1}_{\{ {\tt S}_n\in  [\mathrm{e}^t,\mathrm{e}^{t+1}]   \}}\bigg]
\le 
\bb{E}_p\bigg[\sum_{n=0}^{{\sigma_{M/2}}\land \sigma_{2B}^+} \mathds{1}_{\{ Y_n\in  [\mathrm{e}^t/2, 2\mathrm{e}^{t+1}]   \}}\bigg]. 
\]

\noindent To derive an upper bound on this sum, we now mimic the proof of (i), with $Y$ in place of $\tt S$. Let $\widetilde{N}_t$ be the sum on the right-hand side above. Let $\widetilde{T}_t^0 := \inf\{0 \le n < \sigma_{M/2}\wedge \sigma^+_{2B} : \; Y_n \in [\mathrm{e}^t/2, 2\mathrm{e}^{t+1}]\}$ and $\widetilde{T}_t^{k+1}:=\inf\{\widetilde{T}_t^k < n < \sigma_{M/2}\wedge \sigma^+_{2B} : \; Y_n \in [\mathrm{e}^t/2, 2\mathrm{e}^{t+1}] \}$, with the convention that $\inf \emptyset = \infty$. 
Let $\tilde{c} = \inf_{q\in[\mathrm{e}^t/2, 2\mathrm{e}^{t+1}]} \mathbb{P}_q(\widetilde{T}^{1}_t = \infty)>0$. Using again the Markov property at $\widetilde{T}^{k-1}_t$, we have for all $k\ge 1$,
\begin{equation} \label{eq: expectation Ntilde}
\mathbb{P}_p(\widetilde{N}_t > k)
=
\mathbb{E}_p\Big[ \mathds{1}_{\{\widetilde{T}^{k-1}_t < \infty\}} \mathbb{P}_{Y(\widetilde{T}^{k-1}_t)}(\widetilde{T}^{1}_t<\infty)\Big]
\leq (1-\tilde{c})\mathbb{P}_p(\widetilde{N}_t>k-1),    
\end{equation}
hence $\mathbb{E}_p[\widetilde{N}_t]\leq 1/\tilde{c}$.

We now bound $\tilde{c}$. Conditioning on the first step,  for all $q\in[\mathrm{e}^t/2, 2\mathrm{e}^{t+1}]$,

\begin{multline}\label{eq: Ttilde=infty}
    \mathbb{P}_q(\widetilde{T}^1_t = \infty) 
    = \mathbb{E}_q[\mathds{1}_{\{Y_1\notin[\mathrm{e}^t/2, 2\mathrm{e}^{t+1}]\}}\mathbb{P}_{Y_1}(\widetilde{T}^0_t = \infty)]\\
    \geq \mathbb{E}_q\bigg[\mathds{1}_{\{Y_1\ge 2\mathrm{e}^{t+2}\}}\mathbb{P}_{Y_1}(\sigma_{2B}^+ < \sigma_{2\mathrm{e}^{t+1}})\bigg] + \mathbb{E}_q\bigg[\mathds{1}_{\{Y_1\le\mathrm{e}^{t-1}/2\}} \mathbb{P}_{Y_1}(\sigma_{M/2}<\sigma_{\mathrm{e}^t/2}^+)\bigg].
\end{multline}

\noindent We deal with the first term. 
Note that the law of $\ln \xi$ is symmetric. Thus applying \cref{hitting Y} to $\tilde{Y} = 1/Y$ we derive that for $B > y \geq 2\mathrm{e}^{t+{1}}$,
\[
\mathbb{P}_{y}(\sigma_{2B}^+ < \sigma_{2\mathrm{e}^{t+1}})
\stackrel{\eqref{hitting Y}}{\ge}
\frac{1}{c}\frac{1+\ln(y)-\ln(2\mathrm{e}^{t+1})}{\ln(B)-t} 
\ge
\frac{1}{c}\frac{1}{\ln(B)-t}
\ge
\frac{1}{c(1+\ln(B) - t)},
\]
where we used $y\ge 2\mathrm{e}^{t+{1}}$ in the second inequality. This bound still holds for $y\geq B$ since in that case the probability on the left-hand side is $1$.
Thus the first term of \eqref{eq: Ttilde=infty} is bounded from below as
\[
\mathbb{E}_q\bigg[\mathds{1}_{\{Y_1\ge 2\mathrm{e}^{t+2}\}}\mathbb{P}_{Y_1}(\sigma_{2B}^+ < \sigma_{2\mathrm{e}^{t+1}})\bigg]
\ge 
\frac{1}{c}\frac{\mathbb{P}_q(Y_1\ge 2\mathrm{e}^{t+2})}{1+\ln(B)-t}.
\]
Likewise, a direct application of \eqref{hitting Y} shows that the second term of \eqref{eq: Ttilde=infty} is bounded as
\[
\mathbb{E}_q\bigg[\mathds{1}_{\{Y_1\le\mathrm{e}^{t-1}/2\}} \mathbb{P}_{Y_1}(\sigma_{M/2}<\sigma_{\mathrm{e}^t/2}^+)\bigg]
\ge 
\frac{1}{c} \frac{\mathbb{P}_q(Y_1\le \mathrm{e}^{t-1}/2)}{1+t-\ln(M)}.
\]
Further, notice that since $q\in[\mathrm{e}^t/2, 2\mathrm{e}^{t+1}]$, 
\[
\mathbb{P}_q(Y_1\ge 2\mathrm{e}^{t+2})
\ge 
\mathbb{P}(\xi \ge 4e^2)
\quad \text{and} \quad 
\mathbb{P}_q(Y_1\le \mathrm{e}^{t-1}/2)
\ge
\mathbb{P}\Big(\xi \le \frac{1}{4e^2}\Big).
\]
Combining these two bounds, we get from \eqref{eq: Ttilde=infty} to
\[
\mathbb{P}_q(\widetilde{T}^1_t = \infty)
\ge
C \max\bigg( \frac{1}{1+\ln(B)-t}, \frac{1}{1+t-\ln(M)}\bigg),
\]
for some constant $C>0$. This bound being uniform in $q\in[\mathrm{e}^t/2, 2\mathrm{e}^{t+1}]$, we conclude from \eqref{eq: expectation Ntilde} that
\begin{equation} \label{eq: e(t) Ntilde}
    \mathbb{E}_p[\widetilde{N}_t]\leq c(1 + \min(\ln B-t, t-\ln M)).
\end{equation}
\noindent Plugging \eqref{eq: e(t) tau_1} and \eqref{eq: e(t) Ntilde} back into \eqref{eq: proof green function tau1} and taking the supremum on the left hand side, we obtain
\[
e(t) \le c (1+ \min(\ln(B) - t ,t -\ln(M))) + cM^{-\frac{1}{16}} e(t),
\]

\noindent so $e(t) \le c (1+ \min(\ln(B) - t ,t -\ln(M)))$ if $M$ is greater than some $M_0$. We now deal with the case $M\le M_0$. By the Markov property at the return time of $\tt S$ to $[\mathrm{e}^t,\mathrm{e}^{t+1}]$ (strictly) after  ${\tt T}_{M_0}$ (call it ${\tt T}$), 
\[
{\tt E}_p\bigg[\sum_{n=0}^{{\tt T_M}\land {\tt T}_B^+} \mathds{1}_{\{ {\tt S}_n\in  [\mathrm{e}^t,\mathrm{e}^{t+1}]   \}}\bigg]
\le
{\tt E}_p\bigg[\sum_{n=0}^{{\tt T_{M_0}}\land {\tt T}_B^+} \mathds{1}_{\{ {\tt S}_n\in  [\mathrm{e}^t,\mathrm{e}^{t+1}]   \}}\bigg]
+
{\tt P}_p({\tt T}<\infty) e(t).
\]

\noindent  The probability to jump from $k\in [0,M_0]$ to $0$ is greater than a constant $c_0 \in (0,1)$ hence ${\tt P}_p({\tt T}<\infty) \le 1-c_0$. From what we have already proved, we deduce that
\[
e(t) \le c (1+ \min(\ln(B) - t ,t -\ln(M_0))) + (1-c_0)e(t),
\]

\noindent so that $e(t) \le c (1+ \min(\ln(B) - t ,t -\ln(M)))$ for all $M\ge 1$.
\end{proof}

\begin{Cor}\label{c:sum moments}
Let $\beta>0$.

    (i) Suppose $n\in (0,2)$ and let $\gamma \in (0,\beta)$ be as in Proposition \ref{prop: hitting S infinite}. There exists a constant $c>0$ such that for all $B\ge M \ge 1$  and $p\in [1,B]$,
    \[
    {\tt E}_p\bigg[\sum_{n=0}^{\infty} {\tt S}_n^{-\beta} \mathds{1}_{\{{\tt S}_n\ge M\}}\bigg] \le c M^{-\beta}, \quad \text{and} \quad  {\tt E}_p\bigg[\sum_{n=0}^{\infty} {\tt S}_n^{\beta} \mathds{1}_{\{{\tt S}_n\le B\}}\bigg] \le c B^{\beta}\left(\frac{B}{p}\right)^{-\gamma}.
    \]
    
    (ii) Suppose $n=2$. There exists a constant $c>0$ such that for all $B\ge M\ge 1$ and $p\in [M,B]$,
    \[
    {\tt E}_p\bigg[\sum_{n=0}^{{\tt T}_M\land {\tt T}_B^+ -1 } {\tt S}_n^{-\beta}\bigg] \le c M^{-\beta}\frac{1+\ln(B)-\ln(p)}{1+\ln(B)-\ln(M)}, \quad \text{and} \quad  {\tt E}_p\bigg[\sum_{n=0}^{{\tt T}_M\land{\tt T}_B^+-1} {\tt S}_n^{\beta}\bigg] \le c B^{\beta}.
    \]

\end{Cor}

\begin{proof}
Throughout the proof we allow the constant $c>0$ to vary from line to line. 

\smallskip
\noindent (i) Let $t_M:= \lfloor \ln M \rfloor $ and write $\sum_{n=0}^{\infty} {\tt S}_n^{-\beta} \mathds{1}_{\{{\tt S}_n\ge M\}} \le \sum_{t=t_M}^\infty \mathrm{e}^{-\beta t}\sum_{n=0}^\infty \mathds{1}_{\{{\tt S}_n\in [\mathrm{e}^t,\mathrm{e}^{t+1}]\}}$. By Proposition \ref{p:green} (i), ${\tt E}_p\big[\sum_{n=0}^{\infty} {\tt S}_n^{-\beta} \mathds{1}_{\{{\tt S}_n\ge M\}}\big] \le c \sum_{t=t_M}^\infty \mathrm{e}^{-\beta t} \le c M^{-\beta}$. Similarly, let $t_B:= \lfloor \ln B \rfloor $. We have  $\sum_{n=0}^\infty {\tt S}_n^{\beta}\mathds{1}_{\{{\tt S}_n\le B\}} \le  \sum_{t=0}^{t_B} \mathrm{e}^{\beta (t+1)} \sum_{n=0}^\infty\mathds{1}_{\{{\tt S}_n\in [\mathrm{e}^t,\mathrm{e}^{t+1}]\}}$. Notice that by the Markov property at time ${\tt T}_{\mathrm{e}^t}^+$ and Proposition \ref{p:green} (i),
    ${\tt E}_p\big[\sum_{n=0}^\infty \mathds{1}_{\{ {\tt S}_n\in [\mathrm{e}^t,\mathrm{e}^{t+1}]\}}\big]\le c {\tt P}_p({\tt T}_{\mathrm{e}^t}^+<\infty)$ for $\mathrm{e}^t>p$. Proposition \ref{prop: hitting S infinite} implies that ${\tt E}_p\big[\sum_{n=0}^\infty \mathds{1}_{\{ {\tt S}_n\in [\mathrm{e}^t,\mathrm{e}^{t+1}]\}}\big] \le c \big(\frac{\mathrm{e}^t}{p}\big)^{-\gamma} $ whenever $\mathrm{e}^t>p$. This {bound} still holds when $\mathrm{e}^t \le p$, by Proposition \ref{p:green} (i). We get
    \[
{\tt E}_p\bigg[\sum_{n=0}^\infty {\tt S}_n^{\beta}\mathds{1}_{\{{\tt S}_n\le B\}}\bigg] \le c  \sum_{t=0}^{t_B} \mathrm{e}^{\beta t}\bigg(\frac{\mathrm{e}^t}{p}\bigg)^{-\gamma} \le c B^{\beta} \left(\frac{B}{p}\right)^{-\gamma}.
    \]

    \smallskip
    \noindent (ii) Keeping the notation $t_M$ and $t_B$ from (i) and setting $t_p = \lfloor\ln p\rfloor$, we have 
    \[\sum_{n=0}^{{\tt T}_M\land {\tt T}_B^+ -1 } {\tt S}_n^{-\beta} \le \sum_{t=t_M}^{t_B} \mathrm{e}^{-\beta t}\sum_{n=0}^{\tt T_M \wedge \tt T_B^+} \mathds{1}_{\{{\tt S}_n\in [\mathrm{e}^t,\mathrm{e}^{t+1}]\}}.\]
    By Proposition \ref{p:green} (ii), we obtain
    \begin{align*}
    {\tt E}_p\bigg[\sum_{n=0}^{{\tt T}_M\land {\tt T}_B^+ -1 } {\tt S}_n^{-\beta}\bigg]
    &\le 
    c \sum_{t=t_M}^{t_B} \mathrm{e}^{-\beta t}  (1+ \min(\ln(B) - t ,t -\ln(M))) \frac{2+\ln(B)-\ln(p)}{1+\ln(B)- t}\\
    &\le 
    c \sum_{t=t_M}^{t_B} \mathrm{e}^{-\beta t}(1+t-\ln (M))\frac{2+\ln(B)-\ln(p)}{1+\ln(B)-t} \\
    &
    \stackrel{t=s+t_M}{\le} c M^{-\beta} \frac{2+\ln(B)-\ln(p)}{1+\ln(B)- \ln(M)} \sum_{s=0}^{t_B-t_M} (1+s)\mathrm{e}^{-\beta s} \frac{1+\ln(B)-\ln(M)}{1+\ln(B)-\ln(M)-s}\\
    &
    \le c M^{-\beta} \frac{2+\ln(B)-\ln(p)}{1+\ln(B)- \ln(M)} \sum_{s=0}^{t_B-t_M} \mathrm{e}^{-\beta s} (1+s)^2
    \\
    &\le 
    c M^{-\beta}\frac{2+\ln(B)-\ln(p)}{1+\ln(B)-\ln(M)}.
    \end{align*}

    Finally we have $\sum_{n=0}^{{\tt T}_M\land{\tt T}_B^+-1} {\tt S}_n^{\beta} \le \sum_{t=t_M}^{t_B} \mathrm{e}^{\beta (t+1)}\sum_{n=0}^{\tt T_M \wedge \tt T_B^+} \mathds{1}_{\{{\tt S}_n\in [\mathrm{e}^t,\mathrm{e}^{t+1}]\}}$ hence by Proposition \ref{p:green} (ii), 
    \begin{multline*}
        {\tt E}_p\bigg[\sum_{n=0}^{{\tt T}_M\land{\tt T}_B^+-1} {\tt S}_n^{\beta}\bigg] \le 
        c \sum_{t=t_M}^{t_B} \mathrm{e}^{\beta t} (1+ \ln(B) - t) \min\bigg(\frac{2+\ln(B)-\ln(p)}{1+\ln(B)- t}, 1\bigg)\\
        \le c \sum_{t=t_M}^{t_p} \mathrm{e}^{\beta t} (2+\ln (B)-\ln (p)) + c\sum_{t=t_p}^{t_B} \mathrm{e}^{\beta t}(1+\ln B-t)
        \leq cp^{\beta}(2+\ln(B/p)) + cB^{\beta}
        \leq 
        c B^{\beta}. 
    \end{multline*}
\end{proof}

%
%
\section{Classification into good or bad regions of the map}
\label{sec: classification}
We translate the estimates obtained in the previous section into volume estimates for various portions of the loop-decorated quadrangulation $(\frak q, \bf\ell)$.  This results in a classification into good or bad regions of the map (\cref{sec: good or bad}), where we rule out the contribution of the bad regions to the volume (\cref{sec: estimates bad regions}). Our \cref{def: bad vertices} below gives in particular the profile of \emph{good} or \emph{bad} vertices, \textit{i.e.} vertices which will {respectively} contribute or not to the volume in the scaling limit as the perimeter goes to infinity. 
{
For an overview of how this classification is used in the final proof, we suggest to have a look at Figures~\ref{fig: proof diagram 1}, \ref{fig: proof diagram 2} and \ref{fig: proof diagram 3} in \cref{sec: proof main result}.
}

For $B\ge 1$, let ${\cal E}(B)$ be the event 
\begin{equation}\label{def:EB}
    {\cal E}(B):=\{\forall\, u\in \mathcal{U}, \, \chi^{(p)}(u)\le B\}.
\end{equation}

\noindent For  $u\in \mathcal{U}$, let $T_{B}^+(u):=\inf\{k\in [0,|u|]\,:\, \chi^{(p)}(u_k)>B\}$ with the convention that $\inf\emptyset=\infty$. 
\begin{prop}\label{prop: EB}
    For all $B\ge 1$ and $p\in [1,B]$, 
    \begin{equation}\label{eq:expectation volume greater then B}
    \bb E^{(p)}\bigg[\sum_{u\in \mathcal U} \overline{V}(\chi^{(p)}(u))\mathds{1}_{\{T_{B}^+(u)=|u|\}} \bigg] = \overline{V}(p){\tt P}_p({\tt T}_{B}^+<\infty).
    \end{equation}
    In particular, there exists a constant $C>0$ such that for all $B\ge 1$ and $p\in [1,B]$, $\mathbb{P}^{(p)}({\cal E}(B)^\mathrm{c})\le C\left(\frac{p}{B}\right)^{\theta_\alpha}\frac{1 + \ln B}{1 + \ln p}$.
\end{prop}
\begin{proof}
By the many-to-one formula \eqref{eq: many-to-one}, the expectation in \eqref{eq:expectation volume greater then B}  is
    \[
    \sum_{n=0}^\infty \bb E^{(p)}\bigg[\sum_{|u|=n} \overline{V}(\chi^{(p)}(u))\mathds{1}_{\{T_{B}^+(u)=n\}} \bigg] = \overline{V}(p){\tt P}_p({\tt T}_{B}^+<\infty),
    \]
    indeed. By \eqref{eq: mean volume} and \eqref{eq: mean volume 2}, it is smaller than $ \overline{V}(p)\le Cp^{\theta_\alpha}$ in {Case}~\ref{caseA} and smaller than $\overline{V}(p)\le Cp^{\theta_\alpha}(1 + \ln p)^{-1}$ in {Case}~\ref{caseB}. On the other hand, on the event ${\cal E}(B)^\mathrm{c}$,  
    \[
    \sum_{u\in \mathcal{U}} \overline{V}(\chi^{(p)}(u))\mathds{1}_{\{|u|=T_{B}^+(u)\}}\ge 
    \begin{cases}
        CB^{\theta_{\alpha}}, & \text{ in {Case}~\ref{caseA},}\\
        CB^{\theta_{\alpha}}(1 + \ln B)^{-1}, & \text{ in {Case}~\ref{caseB}.}
    \end{cases}
    \]

    \noindent We conclude by Markov's inequality.
    \end{proof}

\subsection{Estimates on the gasket of the map}
\label{sec: gasket estimates}
As a first application of the estimates in \cref{sec: markov chain estimates}, we start by ruling out the volume of the gasket $\frak g_M$ corresponding to the portion of the planar map which lies outside all loops of perimeter smaller than $2M$, where $M$ will be a large constant.
More precisely, recall from \cref{sec: gasket decomposition} that, by the gasket decomposition, we may view the loop-decorated quadrangulation $(\frak q, \bf\ell)$ as the combination of a gasket, rings inside the holes, and planar maps with loop configurations inside the rings. For $u\in\mathcal{U}$ such that $\chi^{(p)}(u) \ne 0$, the loop labelled by $u$ contains a loop-decorated planar map. Define $V_g(u)$ to be the volume of the gasket for this map, with the convention that $V_g(u)=0$ if $\chi^{(p)}(u) = 0$. Notice that by definition of ${\tt S}_1$, for any $p\ge 1$,
\begin{equation}\label{eq P(S_1=0)}
{\tt P}_p({\tt S}_1=0)=1-\frac{1}{\overline{V}(p)}\bb E^{(p)}\bigg[\sum_{i=1}^\infty \overline{V}(\chi^{(p)}(i))\bigg] = \frac{\bb E^{(p)}[V_g(\varnothing)]}{\overline{V}(p)}.
\end{equation}

Recall from \cref{sec: CCM results} the notation $u_k$ for the ancestor of $u$ at generation $k\le |u|$. For $u\in \mathcal{U}$, let $T_M(u):=\inf\{k\in [0,\abs{u}]\,:\, \chi^{(p)}(u_{k})<M\}$ with the convention that $\inf\emptyset=\infty$. Then the volume of $\frak g_{M}$ is 
\begin{equation}
\label{eq: volume g_M def}
|\frak g_M| = \sum^{\infty}_{n=0} \sum_{\abs{u} = n} V_g(u) \mathds{1}_{\{\chi^{(p)}(u_0), \chi^{(p)}(u_1), \ldots, \chi^{(p)}(u_n)\geq M\}} =
\sum^{\infty}_{n=0} \sum_{\abs{u} = n} V_g(u) \mathds{1}_{\{ T_M(u)>|u|\}
}.
\end{equation}

\noindent For $B\ge M$, we let $\frak g_{B,M}$ be the gasket outside loops of half-perimeter smaller than {$M$} or greater than $B$. Recall that $T_B^+(u):=\inf\{k\in [0,\abs{u}]\,:\, \chi^{(p)}(u_{k})> B\}$. The volume of $\frak g_{B,M}$ is 
\begin{equation}
\label{eq: volume g_BM def}
|\frak g_{B,M}| = 
\sum^{\infty}_{n=0} \sum_{\abs{u} = n} V_g(u) \mathds{1}_{\{ T_M(u)>|u|,\, T_{B}^+(u)>|u|\}
}.
\end{equation}

\noindent In the case $n=2$, we define  for $B\ge M\ge 1$ and $p\in {[1,B]}$ 
\begin{equation}\label{def:overline small v}
    \overline{v}_{B,M}(p):= \overline{V}(p)\frac{1+\ln(B)-\ln(p)}{1+\ln(B) -\ln(M)}.
\end{equation}

Our main result in this subsection is the following estimate on the volume of $\frak g_M$ and $\frak g_{B,M}$. We let $\beta_\alpha = \theta_\alpha-\alpha$. 
{Note that $\beta_\alpha$ is the difference of the two exponents corresponding to the mean volumes of the map and its gasket respectively, see \eqref{eq: mean volume} and \eqref{eq: mean volume boltzmann}. As already hinted at in the discussion following \cref{thm: CCM convergence martingales}, the fact that $\beta_\alpha$ is positive should allow to neglect the gasket relative to the whole map. More precisely, the exponent $\beta_\alpha$ should measure how the two volumes compare, which is in essence the content of the following result.}

\begin{thm}\label{thm: gasket estimate}
(i) Suppose $n\in (0,2]$. There exists a constant $C>0$ such that for all $p\ge M\ge 1$
\[
\mathbb{E}^{(p)}[\abs{\frak g_{M}}] \leq C M^{-\beta_\alpha}\overline{V}(p).
\]

(ii) Suppose $n=2$. There exists a constant $C>0$ such that for all $B\ge M\ge 1$ and $p\in [M,B]$,
\[\mathbb{E}^{(p)}[\abs{\frak g_{B,M}}] \leq 
    C M^{-\beta_\alpha}\overline{v}_{B,M}(p).
 \]

\end{thm}
\begin{proof}
For $q\ge 0$, we denote by $\overline{V_g}(q)=\bb E^{(q)}[V_g(\varnothing)]$ the \emph{expected} volume of the gasket of a loop-decorated quadrangulation with half-perimeter $q$ -- with the convention that $\overline{V_g}(q) = 0$ for $q=0$.  
By equation \eqref{eq: volume g_BM def} and the Markov property of the gasket decomposition, we have
\begin{align*}
\mathbb{E}^{(p)}[\abs{\frak g_{B,M}}] &= \sum^{\infty}_{n=0}\mathbb{E}^{(p)}\bigg[\sum_{\abs{u} = n} V_g(u)\mathds{1}_{\{T_M(u)>|u|,\,T_{B}^+(u)>|u|\}}\bigg]\\
&=\sum^{\infty}_{n=0}\mathbb{E}^{(p)}\bigg[\sum_{\abs{u} = n}\overline{V_g}(\chi^{(p)}(u)) \mathds{1}_{\{T_M(u)>|u|,\,T_{B}^+(u)>|u|\}}\bigg].
\end{align*}
The many-to-one formula (\cref{prop: many to one}) reduces the above expression to
\begin{equation}
\label{eq: volume gasket many-to-one}
\mathbb{E}^{(p)}[\abs{\frak g_{B,M}}]
=\sum^{\infty}_{n=0}\overline{V}(p)\tt{E}_{p}\bigg[\mathds{1}_{\{{\tt T}_M>n,\, {\tt T}_{B}^+>n\}}\frac{\overline{V_g}(\tt S_n)}{\overline{V}(\tt S_n)}\bigg]
=\overline{V}(p)\tt{E}_{p}\bigg[\sum^{{\tt T}_M \wedge {\tt T}_{B}^+ -1}_{n=0}\frac{\overline{V_g}(\tt S_n)}{\overline{V}(\tt S_n)}\bigg],
\end{equation}
where recall from \eqref{eq: tau_M} that $\tt T_M = \inf\{n\ge 0, \; \tt S_n < M\}$ denotes the hitting time of $[0, M)$.
Making use of the asymptotic behaviour of $\overline{V_g}$ and $\overline{V}$ \eqref{eq: mean volume boltzmann} and \eqref{eq: mean volume} when $n\in (0, 2)$ and in \eqref{eq: mean volume boltzmann n=2} and \eqref{eq: mean volume 2} when $n=2$, we can pick a constant $C>0$ such that, for all $q\ge 1$,
\begin{equation} \label{eq: volume ratio}
\frac{\overline{V_g}(q)}{\overline{V}(q)}
\le
C q^{-\beta_{\alpha}},
\end{equation}
with $\beta_{\alpha} = \theta_{{\alpha}} -\alpha>0$. Therefore we conclude from \eqref{eq: volume gasket many-to-one} and \eqref{eq: volume ratio} that
\[
\mathbb{E}^{(p)}[\abs{\frak g_{B,M}}]
\le
C \overline{V}(p) \tt{E}_{p}\bigg[\sum^{{\tt T}_M \wedge {\tt T}_{B}^+-1}_{n=0}\tt S_n^{-\beta_{\alpha}}\bigg].
\]

\noindent An application of \cref{c:sum moments}  (ii) 
for $\beta=\beta_{\alpha}$ yields the volume estimate (ii). {Sending} $B\to\infty$ gives (i) in the case $n=2$. In the case $n\in (0,2)$, we write $\mathbb{E}^{(p)}[\abs{\frak g_{M}}]
\le
C \overline{V}(p) \tt{E}_{p}\Big[\sum^{{\tt T}_M-1}_{n=0}\tt S_n^{-\beta_{\alpha}}\Big]$ (using monotone convergence, say) and apply \cref{c:sum moments} (i).
\end{proof}

For future reference, note that \eqref{eq P(S_1=0)} and \eqref{eq: volume gasket many-to-one} imply for any $p\ge 1$,
\begin{equation}\label{eq: expectation gasket P(S_1=0)}
\bb E^{(p)}[|{\frak g}_{B,M}|] = \overline{V}(p) {\tt P}_p({\tt T}_M<{\tt T}_B^+, {\tt S}_{{\tt T}_M}=0).
\end{equation}

\subsection{Good or bad regions of the map}
\label{sec: good or bad}

\cref{thm: gasket estimate} implies that, by choice of $M$, $\frak g_{M}$ can be made as small as we want with respect to the bulk of the map. Here we give a classification of the bulk into regions of $(\frak q, \bf\ell)$ which are \emph{good} or \emph{bad}. Bad regions will be ruled out using first moment methods in the next section (\cref{sec: estimates bad regions}), whereas we shall need second moment estimates on the good region (\cref{sec: estimates good region}).

To start with, note that the complement of $\frak g_M$ consists of vertices which are eventually in some loop of \added{perimeter} smaller than $2M$. This is the union of the various (loop-decorated) maps corresponding to loops labelled by the nodes $u\in \cal U$ such that $\chi^{(p)}(u_k) \geq M$, for all $0\leq k\leq \abs{u}-1$, but $0<\chi^{(p)}(u)<M$. This motivates the following definition.

\begin{defn} \label{def: branch points}
A node $u\in\mathcal{U}$ is called a \textbf{branch point} if 
\[\forall 0\leq k\leq \abs{u}-1, \; \chi^{(p)}(u_k) \in [M,B], \quad \text{and} \quad 0<\chi^{(p)}(u)<M.\]
The set of branch points is denoted by $\mathcal{N}$.
\end{defn}

\noindent 
We can view ${\mathcal N}$ as a stopping line for the discrete perimeter cascade $(\chi^{(p)}(u), u\in \cal U)$. While exploring the loops of the map by generation, one freezes a loop when it has perimeter smaller than $2M$ and does not explore the submap inside. A branch point in the above definition corresponds to a frozen loop in the exploration. It is easy to see that the spatial Markov property still holds in this case, \textit{i.e.} conditionally on $\frak g_{M}$, all these maps are independent and distributed according to $\mathbb{P}^{(\chi^{(p)}(u))}$. 

Given our estimates on the gasket (\cref{thm: gasket estimate}), we now turn our attention to the volume of the branch points:
\begin{equation} \label{eq: volume branch points}
V_{\cal N} := \sum_{u\in \mathcal{N}} V(u).
\end{equation}

The expectation of $V_{\cal N}$ can be calculated as a straightforward application of the many-to-one formula (\cref{prop: many to one}). Indeed 
\begin{align*}
\mathbb{E}^{(p)}[V_{\cal N}]
&= \sum^{\infty}_{n=0} \mathbb{E}^{(p)}\bigg[\sum_{\abs{u}=n} V(u)\mathds{1}_{\{\chi^{(p)}(u_0) \in [M,B], \ldots,\chi^{(p)}(u_{n-1}) \in [M,B], \; 0<\chi^{(p)}(u_n)<M\}} \bigg]\\
&= \sum^{\infty}_{n=0} \overline{V}(p)\tt{P}_p\big(\tt S_0 \in [M,B], \ldots, \tt S_{n-1} \in [M,B], \; 0<\tt S_n < M\big),
\end{align*}
\textit{i.e.} \ 
\begin{equation} \label{eq: expected volume branch points}
\mathbb{E}^{(p)}[V_{\cal N}] = \overline{V}(p)\tt P_p(\tt T_M<{\tt T}_B^+, \; \tt S_{\tt T_M}\ne 0),
\end{equation}
where $\tt T_M$ is defined by \eqref{eq: tau_M}.

\bigskip

\noindent \textbf{Classification of branch points.}
We are now ready to present our classification of the branch points in $\cal N$. {It is important to notice that the classification depends on a quadruplet $(B,M,A,L)$ of positive numbers, with $B>M$. These numbers will appear in all our later estimates (\cref{sec: estimates bad regions} and \cref{sec: estimates good region}) and will be tuned in the final proof in \cref{sec: proof main result}. We think of $M$ and $B$ as lower and upper \emph{barriers} for the branching Markov chain $(\chi^{(p)}(u),u\in\mathcal{U})$ under $\mathbb{P}^{(p)}$. Since the latter cascade is of order $p$ (recall \cref{thm: CCM convergence cascade}), they will both be sent to infinity in some precise way. In particular, we will send $M\to\infty$ so as to ensure that we may rule out the contribution of the gasket (by \cref{thm: gasket estimate}), while we will take $B=bp$ (with $b$ large) so that the limiting multiplicative cascade $(Z_\alpha(u), u\in\mathcal{U})$ sees a barrier of height $b$. As expected from the comparison with the boundary case of branching random walk, we will see that the discrete perimeter cascade \emph{feels} the effect of the upper barrier if, and only if, $n=2$ (this explains the logarithmic correction in the second and third items of \cref{thm: second moment good volume}). See \cref{sec: hitting n<2} and \cref{sec:hitting n=2} for related discussions.}

\medskip

\begin{defn} \label{def: bad vertices}
{Fix $\mu=\frac{\theta_\alpha}{4}$ (any constant $\mu \in (0,\frac{\theta_\alpha}{2})$ would do). Let $(B,M,A,L)$ be positive numbers with $B > M$.}
\begin{itemize}
\item For $u\in \mathcal{U}$, we say that $u$ has \textbf{moderate increments} if \footnote{{Roughly speaking, this technical condition should be thought of as a control on the \emph{size} of the offspring along the branch from the root $\varnothing$ to the node $u$, where the size is measured with respect to the \emph{expected volume carried by the offspring}. Indeed, recalling the asymptotics \eqref{eq: mean volume} and \eqref{eq: mean volume 2}, we see that \eqref{eq: size increments} can be interpreted as some decay condition on the expected volume contained at each generation. In applications, $B$ will act as an upper barrier so that the term $\frac{B}{\chi^{(p)}(u_k)}$ will be typically large.}} 
\begin{equation}\label{eq: size increments}
\sum^{\infty}_{i=1} (\chi^{(p)}(u_ki))^{\theta_{\alpha}} \bigg(1+\ln_+ \frac{\chi^{(p)}(u_k)}{\chi^{(p)}(u_ki)}\bigg)^{2}\leq A\left(\frac{B}{\chi^{(p)}(u_k)}\right)^\mu (\chi^{(p)}(u_k))^{\theta_{\alpha}}, \quad \forall 0\leq k<\abs{u}.
\end{equation}
\item For $u\in \mathcal{U}$, let $V(u)$ be the volume of map inside the loop labelled by $u$. We say that $u$ \textbf{contains a small map} if
\begin{equation} \label{eq: small map}
V(u) \leq L.
\end{equation}
\end{itemize}
Consider the set $\mathcal N$ of branch points associated to  the level $M$ as in \cref{def: branch points}. The set of $u\in \cal N$ that satisfy conditions \eqref{eq: size increments} and \eqref{eq: small map} above is denoted by $\cal G$. The set of other branch points $u\in \cal N \setminus \cal G$ is denoted by $\cal B$.
\end{defn}

The above definition gives the profile of \textbf{good loops} (with labels in $\cal G$), as opposed to \textbf{bad loops} (with labels in $\cal B$). We stress that all the nodes in $\cal G$ or $\cal B$ are branch points in the sense of \cref{def: branch points}, and that all the above definitions, in particular that of $\cal G$ and $\cal B$, depend on $(B,M,A,L)$.  

{
At several points we will need to discuss on the most recent common ancestor of two distinct good loops. It will be convenient to have the following definition.\footnote{Note that any ancestor of two distinct good loops with labels $u$ and $w$ satisfy both properties in \cref{def:calA}, the first one by definition of $\cal G$, and the second one since $u$ and $w$ are branch points in the sense of \cref{def: branch points}.}
\begin{defn}\label{def:calA}
We denote by $\cal A$ the set of nodes $v$ that exhibit moderate increments \eqref{eq: size increments}, and such that $\chi^{(p)}(v_k) \in [M,B]$ for all $1\le k\le |v|$. 
\end{defn}
}

The estimate in \cref{thm: tail biggins} will enable us to control the volume carried by loops for which \eqref{eq: size increments} fails, \textit{i.e.} a typical loop will satisfy  \eqref{eq: size increments} (and we shall see that \eqref{eq: small map} also typically holds).
To this end, we divide the remaining volume $V_{\cal N}$ defined in \eqref{eq: volume branch points} into $V_{\cal N} = V_{\cal G} + V_{\cal B}$, where
\begin{equation} \label{eq: good and bad volumes}
V_{\cal G} := \sum_{u\in \mathcal{G}} V(u),
\quad \text{and} \quad
V_{\cal B} := \sum_{u\in \mathcal{B}} V(u).
\end{equation}

\subsection{First moment estimates on the bad regions of the map}
\label{sec: estimates bad regions}
In this paragraph, we deal with the contribution of the bad vertices by ruling out the corresponding volume $V_{\cal B}$ in the scaling limit as $p\to \infty$. Recall that the definition of ${\cal B}$ involves a quadruplet $(B,M,A,L)$. Recall the notation $\overline{v}_{B,M}(p)$ in \eqref{def:overline small v}. Our main result is the following. 

\begin{thm} \label{thm: volume bad}
There exist some constants $c>0$ and $\eta>0$ such that for all $B\ge M\ge 2$, $L\ge 1$, $A>0$  and $p\in [M,B]$,
\[
\bb E^{(p)} [V_{\cal B}] \le c(A^{-\eta}+\sup_{q\in [1,M]}\mathbb{E}^{(q)}[V\cdot \mathds{1}_{\{V\geq L\}}]) \begin{cases}
    \overline{V}(p) & \textrm{if }  n\in (0,2), \\
    \overline{v}_{B,M}(p)& \textrm{if }  n=2.
\end{cases} 
\]
\end{thm}

\cref{thm: volume bad} will be derived from a bound on the volume of each bad region following \cref{def: bad vertices}. More precisely, let  $V_1=V_1(B,M,A)$ stand for the volume associated to labels $u\in {\mathcal N}$ that do not satisfy \eqref{eq: size increments}. Similarly, let $V_2=V_2(B,M,L)$ be the volume associated to labels $u\in {\mathcal N}$ for which \eqref{eq: small map} fails. Breaking $V_{\cal B}$ as $V_{\cal B} \le V_{1} + V_{2} $, we see that \cref{thm: volume bad} will be a consequence of the following two lemmas. 

\begin{lem} \label{lem: volume V2}
There exists a constant $C>0$ and some $\eta>0$ such that, for all $B\ge M\ge 2$, $A>0$ and $p\in [M,B]$, 
\[\mathbb{E}^{(p)}[V_1]\leq C A^{-\eta} \begin{cases}
    \overline{V}(p) & \textrm{if }  n\in(0,2), \\
    \overline{v}_{B,M}(p)& \textrm{if }  n=2.
\end{cases} \]
\end{lem}

\begin{lem} \label{lem: volume V3}
There exists a constant $c>0$ such that for all  $B\ge p\ge M\ge 1$ and $L\ge 1$, 
\[\mathbb{E}^{(p)}[V_2] \leq \sup_{q\in [1,M]}\mathbb{E}^{(q)}[V\cdot \mathds{1}_{\{V\geq L\}}] \begin{cases}
    \overline{V}(p) & \textrm{if }  n\in(0,2), \\
    c \overline{v}_{B,M}(p)& \textrm{if }  n=2.
\end{cases} 
 \]
\end{lem}

\noindent The remainder of this section presents the proofs of the above lemmas.

\begin{proof}[Proof of \cref{lem: volume V2}]
In this proof, we {extend the notation $\overline{v}_{B,M}(q)$ by setting} $ \overline{v}_{B,M}(q):=\overline{V}(q)$ if $n\in(0,2)$ for simplicity. Therefore we need to show that $\bb E^{(p)}[V_1]\le C A^{-\eta} \overline{v}_{B,M}(p)$. 

{
\bigskip
\noindent \emph{\underline{Step 1}: Fixing the parameters $\eta$ and $\delta$.}
}
{For future purposes, we prefer to fix right away a couple of parameters that will be used in the later estimates of Step 4.} By \cref{thm: tail biggins}, there exists $\eta>0$ such that 
\[\sup_{p\ge 2}\mathbb{E}^{(p)}\bigg[\bigg(\sum_{\abs{u}=1} \bigg(\frac{\chi^{(p)}(u)}{p}\bigg)^\theta \bigg)^{1+\eta}\bigg] <\infty,\]

\noindent for all $\theta$ close enough to $\theta_\alpha$. Fix such $\eta$ and let $\delta>0$ so that the display above holds for $\theta\in [\theta_\alpha-{2}\delta,\theta_\alpha]$ and ${3}\delta< \mu \eta$. Let $g(x):= x^{\theta_\alpha}(1+\ln_+ \frac1x)^2$ for $x\ge 0$.  Using that $g(x) \le c(x^{\theta_\alpha}+x^{\theta_\alpha - 2\delta})$ for some constant $c>0$, we have that
$(\sum_i g(x_i))^\eta \sum_i x_i^{\theta_{\alpha}-{2}\delta} \le c (\sum_i x_i^{\theta_{\alpha}-{2}\delta} + \sum_i x_i^{\theta_{\alpha}})^{1+\eta}$. Thus Minkowski's inequality implies that 
\begin{equation}\label{proof moderate increments}
    \sup_{p\ge 2}\mathbb{E}^{(p)}\bigg[\bigg(\sum_{\abs{u}=1} g\bigg(\frac{\chi^{(p)}(u)}{p}\bigg) \bigg)^{\eta} \sum_{\abs{u}=1} \bigg(\frac{\chi^{(p)}(u)}{p}\bigg)^{\theta_\alpha-{2}\delta} \bigg] <\infty.
\end{equation}

{
\bigskip
\noindent \emph{\underline{Step 2}: A many-to-one estimate for the expectation of $V_1$.}
}
For a node $u$, let $T(u):= T_B^+(u)\land T_M(u)$. {We note that any} node counting in $V_{1}$ must have {a strict} ancestor $u$ such that $T(u)>|u|$ and $\sum^{\infty}_{i=1} g\Big(\frac{\chi^{(p)}(ui)}{\chi^{(p)}(u)}\Big) > A \Big(\frac{B}{\chi^{(p)}(u)}\Big)^\mu$. We get 
\begin{multline*}
\mathbb{E}^{(p)}[V_{1}] 
\le \mathbb{E}^{(p)}\bigg[ \sum_{u \in \mathcal{U}} \mathds{1}_{\{ T(u)>|u| \}}\cdot \mathds{1}_{\big\{\sum^{\infty}_{i=1} g\big(\frac{\chi^{(p)}(ui)}{\chi^{(p)}(u)}\big) > A \big(\frac{B}{\chi^{(p)}(u)}\big)^\mu\big\}} 
\cdot \sum_{k=1}^{\infty}\sum_{w\in \cal U} V(ukw){\mathds{1}_{\{ukw\in\mathcal{N}\}}}\bigg]. 
\end{multline*}
By the gasket decomposition, the right-hand side is
\begin{multline*} 
\mathbb{E}^{(p)}\bigg[ \sum_{u \in \mathcal{U}} \mathds{1}_{\{ T(u)>|u| \}} \cdot 
\mathbb{E}^{(q)}\bigg[\mathds{1}_{\big\{\sum^{\infty}_{i=1} g\big(\frac{\chi^{(q)}(i)}{q}\big) > A \left(\frac{B}{q}\right)^\mu\big\}}
\cdot \sum_{k=1}^{\infty}\bb 
E^{(\chi^{(q)}(k))}\Big[
{\sum_{w\in \cal N} V(w)}\Big]{\bigg]}_{q=\chi^{(p)}(u)}\bigg].
\end{multline*}
{Using \eqref{eq: expected volume branch points} and \cref{prop: hitting time S}, there exists a constant $c>0$ such that the innermost expectation rewrites\footnote{{Note that the last bound is obvious with $c=1$ when $n\in (0,2)$.}}
\[
    \bb E^{(\chi^{(q)}(k))}\Big[\sum_{w\in \cal N} V(w)
    \Big]
    =
    \overline{V}(\chi^{(q)}(k)) \tt P_{\chi^{(q)}(k)}(\tt T_M<\tt T_B^+, \tt S_{\tt T_M} \ne 0)
    \leq
    c \overline{v}_{B,M}(\chi^{(q)}(k)) {\mathds{1}_{\{\chi^{(q)}(k)\le B\}}}.
\]
}
\noindent Therefore, we end up with the upper-bound
\begin{equation} \label{eq:_bound_E[V1]}
\mathbb{E}^{(p)}[V_{1}] 
\leq  c
\mathbb{E}^{(p)}\bigg[ \sum_{u \in \mathcal{U}} \mathds{1}_{\{ T(u)>|u| \}}\cdot F(\chi^{(p)}(u))\bigg],
\end{equation}
{where 
\[
F(q) = \mathbb{E}^{(q)}\bigg[\mathds{1}_{\big\{\sum^{\infty}_{i=1} g\big(\frac{\chi^{(q)}(i)}{q}\big) > A \left(\frac{B}{q}\right)^\mu\big\}} \cdot \sum^{\infty}_{k=1} \overline{v}_{B,M}(\chi^{(q)}(k)) {\mathds{1}_{\{\chi^{(q)}(k)\le B\}}}\bigg].
\]
}
{
\bigskip
\noindent \emph{\underline{Step 3}: Further bound on the terms $\overline{v}_{B,M}(\chi^{(q)}(k))$.}
}
{We now want to simplify the inequality \eqref{eq:_bound_E[V1]} by further estimating $\overline{v}_{B,M}(r)$ when $r\in[1,B]$.}
Given the asymptotic behaviour \eqref{eq: mean volume} of $\overline{V}(r)$, {when $n<2$} there exists a constant $C>1$ such that $\frac{1}{C} r^{\theta_\alpha} \le \overline{V}(r) \le C r^{\theta_{\alpha}}$ for all $r\ge 1$. When $n = 2$ and $g=\frac{h}{2}$, we see from \eqref{eq: mean volume 2} that there exists a constant $C>1$ such that $\frac{1}{C} r^{\theta_\alpha}\le \overline{V}(r) \le C r^{\theta_{\alpha}}$ for all $r\ge 1$. Thus by definition of $\overline{v}_{B,M}$ in \eqref{def:overline small v}, for all ${q}, r \in [1,B]$, 
\[
\overline{v}_{B,M}(r)\le C r^2 \frac{1+\ln(B)-\ln(r)}{1+\ln(B)-\ln(M)}
\le C^2 \overline{v}_{B,M}(q) \frac{1+\ln(B)-\ln(r)}{1+\ln(B)-\ln(q)}\left( \frac{r}{q}\right)^2 .
\]

\noindent When $n=2$ and $g<\frac{h}{2}$, by \eqref{eq: mean volume 2} there exists a constant $C>1$ such that $\frac{1}{C} r^{2}(1 + \ln r)^{-1} \le \overline{V}(r) \le C r^{2}(1 + \ln r)^{-1}$ for all $r\ge 1$. Thus in that case, for all $q, r\in [1,B]$,
\[
\overline{v}_{B,M}(r)\le C \frac{r^2}{1 + \ln (r)} \cdot \frac{1+\ln(B)-\ln(r)}{1+\ln(B)-\ln(M)}
\le C^2 \overline{v}_{B,M}(q) \frac{1+\ln(B)-\ln(r)}{1+\ln(B)-\ln(q)} \cdot \frac{\ln (q) + 1}{\ln (r) + 1} \left( \frac{r}{q}\right)^2 .
\]

\noindent 
{Recall the parameter $\delta$ introduced in Step 1.}
Note that 
\[\frac{1+\ln(B)-\ln(r)}{1+\ln(B)-\ln(q)} = 1 + \frac{\ln(q/r)}{1+\ln(B)-\ln(q)} \leq c\left(\frac{q}{r}\right)^{\delta},\]
\noindent and 
\[\frac{1 + \ln (q)}{1 + \ln (r)} = 1 + \frac{\ln(q/r)}{1 + \ln (r)} \leq c\left(\frac{q}{r}\right)^{\delta}. \]
We conclude that in any case, for all $q, r\in [1, B]$,
\[
\overline{v}_{B,M}(r) \le c \overline{v}_{B,M}(q) \left( \frac{r}{q}\right)^{\theta_\alpha} \left(\frac{B}{r}\right)^{2\delta} { = c \overline{v}_{B,M}(q) \left( \frac{r}{q}\right)^{\theta_\alpha - 2\delta} \left(\frac{B}{q}\right)^{2\delta}} . 
\]
\noindent  {Therefore, 
\[
F(q) \leq c\overline{v}_{B,M}(q) \left(\frac{B}{q}\right)^{2\delta} \mathbb{E}^{(q)}\bigg[\mathds{1}_{\big\{\sum^{\infty}_{i=1} g\big(\frac{\chi^{(q)}(i)}{q}\big) > A \left(\frac{B}{q}\right)^\mu \big\}} \cdot\sum^{\infty}_{k=1} \left(\frac{\chi^{(q)}(k)}{q}\right)^{\theta_\alpha-{2}\delta}\bigg]. 
\]
}

{
\bigskip
\noindent \emph{\underline{Step 4}: Conclusion using Step 1.}
}
{
Given the form of the above display, we are now in a position to use the moment estimates relative to the two parameters $\delta$ and $\eta$ introduced in Step 1. Indeed, by} \eqref{proof moderate increments}, for all $q\ge 2$,
\[
\mathbb{E}^{(q)}\bigg[\mathds{1}_{\big\{\sum^{\infty}_{i=1} g\big(\frac{\chi^{(q)}(i)}{q}\big) > A \left(\frac{B}{q}\right)^\mu \big\}} \cdot \sum^{\infty}_{i=1} \left(\frac{\chi^{(q)}(i)}{q}\right)^{\theta_\alpha-{2}\delta}\bigg] 
\le
 C A^{-\eta} \left(\frac{B}{q}\right)^{-\mu\eta}.
\]
{Let us briefly discuss according to the cases $n<2$ or $n=2$:}
\begin{itemize}
\item For $n<2$, {the previous two displays and \eqref{eq:_bound_E[V1]} entail}
\[
\mathbb{E}^{(p)}[V_{1}]
\le  
c A^{-\eta} 
\mathbb{E}^{(p)}\bigg[\sum_{u\in \mathcal{U}}\mathds{1}_{\{T(u)>|u| \}} \overline{V}(\chi^{(p)}(u)) \left(\frac{B}{\chi^{(p)}(u)}\right)^{2\delta -\mu\eta} \bigg].
\]
By the many-to-one formula (\cref{prop: many to one}), this rewrites
\[
\mathbb{E}^{(p)}[V_{1}]
\le  
c  A^{-\eta} \overline{V}(p) {\tt E}_p\bigg[\sum_{n=0}^{{\tt T}_{B}^+\land {\tt T}_M-1} \left(\frac{B}{{\tt S_n}}\right)^{2\delta-\mu \eta} \bigg],
\]
and using $2\delta< \mu\eta$ we conclude by \cref{c:sum moments} (i) that 
\[
\mathbb{E}^{(p)}[V_{1}]
\le
c   A^{-\eta} \overline{V}(p) \bigg(\frac{B}{p} \bigg)^{-\gamma}
\le
c   A^{-\eta} \overline{V}(p).
\]

\item Likewise, for $n=2$,
\[
\mathbb{E}^{(p)}[V_{1}]
\le
C A^{-\eta} 
\mathbb{E}^{(p)}\left[\sum_{u\in \mathcal{U}}\mathds{1}_{\{{T(u)}>|u| \}} \overline{V}(\chi^{(p)}(u))\frac{1 + \ln B - \ln \chi^{(p)}(u)}{1 + \ln B - \ln M} \left(\frac{B}{\chi^{(p)}(u)}\right)^{2\delta -\mu\eta} \right].
\]

\noindent By another use of the many-to-one formula, 
\[
\mathbb{E}^{(p)}[V_{1}]
\le
C   A^{-\eta} \overline{v}_{B,M}(p) {\tt E}_p\bigg[\sum_{n=0}^{{\tt T}_{B}^+\land {\tt T}_M-1} \left(\frac{B}{{\tt S_n}}\right)^{2\delta-\mu \eta} \frac{1 + \ln B - \ln {\tt S}_n}{1 + \ln B - \ln p}\bigg].
\]

\noindent Using 
\[\frac{1 + \ln B - \ln {\tt S}_n}{1 + \ln B - \ln p} = 1 + \frac{\ln(p/{\tt S}_n)}{1 + \ln B - \ln p} \leq c\left(\frac{p}{{\tt S_n}}\right)^{\delta}, \]

\noindent $p\leq B$ and $3\delta \leq \mu\eta$, we conclude again by \cref{c:sum moments} (ii).
\end{itemize}

\noindent {In any case, this concludes the proof of \cref{lem: volume V2}.}
\end{proof}

\begin{proof}[Proof of \cref{lem: volume V3}]
We have, by the gasket decomposition (\cref{prop: spatial markov gasket}),
\begin{align*}
\mathbb{E}^{(p)}[V_{2}]&= \mathbb{E}^{(p)}\bigg[\sum_{u\in\mathcal{N}} V(u)\cdot \mathds{1}_{\{V(u)\geq L\}}\bigg]\\
&\leq \mathbb{E}^{(p)}\bigg[\sum_{u\in\mathcal{N}} \mathbb{E}^{(\chi^{(p)}(u))}[V\cdot \mathds{1}_{\{V\geq L\}}]\bigg]\\
&\leq \sup_{q\in [1,M]}\mathbb{E}^{(q)}[V\cdot \mathds{1}_{\{V\geq L\}}] \mathbb{E}^{(p)}[|\cal N|] \leq \sup_{q\in [1,M]}\mathbb{E}^{(q)}[V\cdot \mathds{1}_{\{V\geq L\}}] \mathbb{E}^{(p)}[V_{\cal N}],
\end{align*}

 \noindent since the definition of $u\in \cal N$ implies that $\chi^{(p)}(u)<M$. By \cref{eq: expected volume branch points}, $\mathbb{E}^{(p)}[V_\mathcal N]\le  \overline{V}(p) {\tt P}_p({\tt T}_M<{\tt T}_B^+)$ and we bound $\overline{V}(p) {\tt P}_p({\tt T}_M<{\tt T}_B^+)$ by $\overline{V}(p)$ if $n\in (0,2)$, and by $c \overline{v}_{B,M}(p)$ (thanks to \cref{prop: hitting time S}) if $n=2$.
\end{proof}

\subsection{Second moment estimate on the good region of the map}
\label{sec: estimates good region}

We complete the picture by now giving second moment estimates for the \emph{good} region of the map, corresponding to labels in $\cal G$ (see \cref{def: bad vertices}).  
Recall that $V_{\cal G}$ denotes the total volume of good loops (as in \eqref{eq: good and bad volumes}) and that it depends on a quadruplet $(B,M,A,L)$.
\begin{thm} \label{thm: second moment good volume}
(i) Suppose $n\in (0,2)$. Let $\gamma>0$ be as in Proposition \ref{prop: hitting S infinite} with $\gamma\in (0,\theta_\alpha -2\mu)$. There exists a constant $C>0$ such that for all $(B,M,A,L)$, and all $p\in [M,B]$,  
\[\mathbb{E}^{(p)}[V_{\cal G}^2]\leq \overline{V}(p) \Big( C A^2 B^{\theta_\alpha} \Big(\frac{B}{p}\Big)^{-\gamma}  + L^2 \Big).\]

(ii) Suppose $n=2$ and $g = \frac{h}{2}$. There exists a constant $C>0$ such that for all $(B,M,A,L)$, and all $p\in [1,B]$,  \[\mathbb{E}^{(p)}[V_{\cal G}^2]\leq C \overline{V}(p)   A^2 B^2 (1+\ln B -\ln M)^{-2}  + L^2\overline{V}(p).\]

(iii) Suppose $n=2$ and $g < \frac{h}{2}$. There exists a constant $C>0$ such that for all $(B,M,A,L)$, and all $p\in [1,B]$,  \[\mathbb{E}^{(p)}[V_{\cal G}^2]\leq C \overline{V}(p)   A^2 B^2 (1+\ln B -\ln M)^{-3}  + L^2\overline{V}(p).\] 
\end{thm}

{
We stress right away that the extra gain of the exponent $\gamma$ in \cref{thm: second moment good volume}(i) will be important in the proof of \cref{thm: main} to send the right-hand side of \eqref{eq:cvg sup ref} to $0$. On the other hand, the logarithmic terms in \cref{thm: second moment good volume}(ii) and (iii) will compensate the extra logarithmic corrections that appear in the case $n=2$ (see \eqref{eq: ref log corrections}).
}

\begin{proof}

{
We first split the expectation into squares and cross terms:
\begin{equation} \label{eq: good volume squares + cross}
\mathbb{E}^{(p)}[V_{\cal G}^2] = \mathbb{E}^{(p)}\bigg[\bigg(\sum_{u\in\cal G} V(u) \bigg)^2\bigg]\\
=
\mathbb{E}^{(p)}\bigg[\sum_{u\in \cal G}V(u)^2\bigg] + \bb E^{(p)} \bigg[ \sum_{u,w\in \cal G} V(u)V(w) \mathds{1}_{\{u\ne w\}}\bigg].
\end{equation}
Let us start with the first term of \eqref{eq: good volume squares + cross}. For this we only use the condition \eqref{eq: small map} that $u$ should contain a small map,
\[
\mathbb{E}^{(p)}\bigg[\sum_{u\in \cal G}V(u)^2\bigg]
\le
L^2 \mathbb{E}^{(p)}[|\cal G|],
\]
and we use the rough estimate $|\cal G| \le V$ to conclude that
\[ 
\mathbb{E}^{(p)}\bigg[\sum_{u\in \cal G}V(u)^2\bigg]
\le
 L^2 \overline{V}(p).
\]
{This term corresponds to the final term in all three estimates (i), (ii), (iii). Hence, the rest of the proof will focus on the other term of \eqref{eq: good volume squares + cross}. For convenience, let us denote this cross term by
\[
\times(p)
:=
\bb E^{(p)} \bigg[ \sum_{u,w\in \cal G} V(u)V(w) \mathds{1}_{\{u\ne w\}}\bigg].
\]
We stress that $\times(p)$ also depends on the set of parameters $(B,M,A,L)$ through $\cal G$, but we prefer to keep it implicit so as to lighten the notation.
}
We henceforth divide the proof into several steps: we first simplify $\times(p)$ (Step 0), and then distinguish according to the cases (i), (ii) and (iii) in the statement.
}

{
\bigskip
\noindent \emph{\underline{Step 0}: Expanding $\times(p)$ using the many-to-one formula.}
}
The main idea is to use the many-to-one formula combining the two conditions that make up a \emph{good} loop. 
We discuss on the last common ancestor $v$ of each $u,w\in \cal G$ in the sum. {Recall the set $\cal A$ of \cref{def:calA}, describing the set of possible ancestors of any two distinct good loops. Splitting over the possible ancestors $v\in\cal A$, we end up with the following sum:}
\[
{\times(p)}
=
\mathbb{E}^{(p)}\bigg[\sum_{v\in\cal A}
\sum^{\infty}_{m=1}\sum_{\abs{u}=m}\sum^{\infty}_{k=1}\sum_{\abs{w}=k}\mathds{1}_{\{u_1\ne w_1\}}V(vu) \mathds{1}_{\{vu\in\cal{G}\}}V(vw)\mathds{1}_{\{vw\in\cal{G}\}}\bigg].
\]
Introduce the event
\begin{equation} \label{eq: def G_q}
G_q:=\bigg\{\sum^{\infty}_{i=1} (\chi^{(q)}(i))^{\theta_{\alpha}}\bigg(1+\ln_+ \frac{q}{\chi^{(q)}(i)}\bigg)^{{2}} \leq A \left( \frac{B}{q} \right)^{\mu} q^{\theta_{\alpha}}\bigg\}.
\end{equation}
By the gasket decomposition (\cref{prop: spatial markov gasket}) and a crude bound, the above display yields
\begin{equation}
    {\times(p)}
    \le
\mathbb{E}^{(p)}\bigg[\sum_{v\in\cal A}
\mathbb{E}^{(q)}\bigg[{\mathds{1}_{G_q}\cdot}\sum^{\infty}_{m=1}\sum_{\abs{u}=m}\sum^{\infty}_{k=1}\sum_{\abs{w}=k}\mathds{1}_{\{u_1\ne w_1\}}V(u) \mathds{1}_{\{u\in\cal{N}\}}V(w)\mathds{1}_{\{w\in\cal{N}\}} \bigg]_{q=\chi^{(p)}(v)}\bigg]. \label{eq: cross term gasket decomposition interm}
\end{equation}
{To avoid cumbersome expressions, it will be convenient to set
\[
E_{m,k}(q)
:=
\mathbb{E}^{(q)}\bigg[\mathds{1}_{G_q}\cdot\sum_{\abs{u}=m}\sum_{\abs{w}=k}\mathds{1}_{\{u_1\ne w_1\}}V(u) \mathds{1}_{\{u\in\cal{N}\}}V(w)\mathds{1}_{\{w\in\cal{N}\}}\bigg],
\quad m,k\geq 1,
\]
and
\[
E(q)
:=
\sum_{m=1}^{\infty}\sum_{k=1}^{\infty} E_{m,k}(q)
=
\mathbb{E}^{(q)}\bigg[\mathds{1}_{G_q}\cdot\sum^{\infty}_{m=1}\sum_{\abs{u}=m}\sum^{\infty}_{k=1}\sum_{\abs{w}=k}\mathds{1}_{\{u_1\ne w_1\}}V(u) \mathds{1}_{\{u\in\cal{N}\}}V(w)\mathds{1}_{\{w\in\cal{N}\}}\bigg],
\]
so that \eqref{eq: cross term gasket decomposition interm} becomes
\begin{equation}
{\times(p)}
\le
\mathbb{E}^{(p)}\bigg[\sum_{v\in\cal A} E(\chi^{(p)}(v))\bigg]. \label{eq: cross term gasket decomposition}
\end{equation}
}

\begin{figure}
\bigskip
\begin{center}
\includegraphics[scale=0.9]{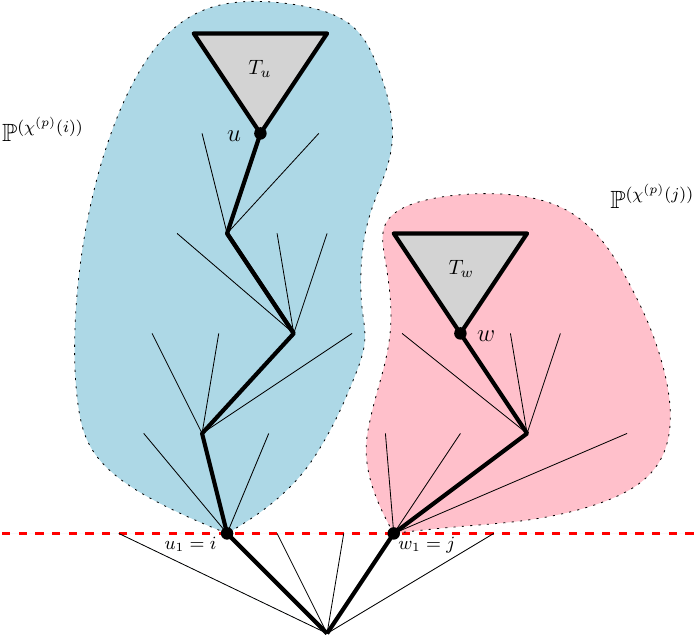}
\end{center}
\caption{{The conditional independence property of the maps inside the loops $u_1$ and $w_1$. Here $u$ and $w$ are nodes at generations $m$ and $k$ respectively, and we are calculating functionals of the subtrees $T_u$ and $T_w$ rooted at $u$ and $w$. Conditional on the first generation (red dashed line), the maps inside $u_1=i$ and $w_1=j$ (outlined by the blue and red regions) are independent with respective laws $\bb P^{(\chi^{(p)}(i))}$ and $\bb P^{(\chi^{(p)}(j))}$.}}
\label{fig: trees_idp_map}
\end{figure}

{Fix some half-perimeter $M\le q\le B$ and some integers \added{$m,k\ge 1$}.} Then using the independence of the maps inside the loops labelled by {$u_1$ and $w_1$} (\cref{prop: spatial markov gasket}) {as represented in \cref{fig: trees_idp_map}}, we get 
\[
{E_{m,k}(q)} 
\le \bb E^{(q)} \bigg[{\mathds{1}_{G_q}\cdot}\sum_{i\ne j}  \bb E^{(\chi^{(q)}(i))}\bigg[\sum_{\abs{u}=m-1} V(u)\mathds{1}_{\{u\in\cal{N}\}}\bigg] \cdot \bb E^{(\chi^{(q)}(j))}\bigg[\sum_{\abs{w}=k-1} V(w)\mathds{1}_{\{w\in\cal{N}\}}\bigg] \bigg].
\]

{\noindent Summing over all \added{$m,k\ge 1$}, we obtain}
\[
{E(q)
\le \bb E^{(q)} \bigg[\mathds{1}_{G_q}\cdot\sum_{i\ne j} \bb E^{(\chi^{(q)}(i))}\Big[\sum_{u\in\mathcal{N}} V(u)\Big] \cdot \bb E^{(\chi^{(q)}(j))}\Big[\sum_{w\in \cal N} V(w)\Big] \bigg].}
\]
\noindent We use {\eqref{eq: expected volume branch points}} to rewrite the two inner expectations: 
\begin{align}\label{eq: cross term T_M<T_B}
{E(q)} 
&\le \bb E^{(q)} \bigg[{\mathds{1}_{G_q}\cdot}\sum_{i\ne j} \overline{V}(\chi^{(q)}(i))\overline{V}(\chi^{(q)}(j)) \tt P_{\chi^{(q)}(i)}(\tt T_M<{\tt T}_{B}^+) \cdot \tt P_{\chi^{(q)}(j)}(\tt T_M<{\tt T}_{B}^+) \bigg]\nonumber\\
&{\le \bb E^{(q)}\bigg[\mathds{1}_{G_q}\cdot\Big(\sum^{\infty}_{i=1}\overline{V}(\chi^{(q)}(i))\tt P_{\chi^{(q)}(i)}(\tt T_M<{\tt T}_{B}^+)\Big)^2\bigg]. }
\end{align}
{From here, we divide the proof into several steps, according to the three items of \cref{thm: second moment good volume}.}

{
\bigskip
\noindent \emph{\underline{Step 1}: Proof of item (i), for $n\in(0,2)$.}
} 
{Bounding the probability by 1}, and using the asymptotic behaviour \eqref{eq: mean volume} of $\overline{V}$, the last inequality becomes
\begin{equation}
 \label{eq: cross term m,k} 
 {E(q)}
{\leq C  \bb E^{(q)} \bigg[ \mathds{1}_{G_q} \cdot \bigg(\sum^{\infty}_{i = 1}  (\chi^{(q)}(i))^{\theta_{\alpha}}\bigg)^2  \bigg]}
\le C \left(\frac{B}{q}\right)^{2\mu} A^2 q^{2\theta_\alpha}
\le C \left(\frac{B}{q}\right)^{2\mu} A^2 q^{\theta_\alpha}  \overline{V}(q), 
\end{equation}

\noindent where in the {second-to-last inequality we used the definition of $G_q$ by ignoring the logarithmic terms in \eqref{eq: def G_q}.} We now plug the latter inequality \eqref{eq: cross term m,k} back into \eqref{eq: cross term gasket decomposition}. Then
\[
 {\times(p)}
\le
C  B^{2\mu} A^2 \mathbb{E}^{(p)}\bigg[\sum_{v\in\cal A} ( \chi^{(p)}(v))^{\theta_\alpha-2\mu}
\overline{V}(\chi^{(p)}(v))\bigg]. 
\]

\noindent By definition of $\cal A$, a node $v\in \cal A$ stays below $B$, so that
\[
 {\times(p)}
\le
C  B^{2\mu} A^2 \mathbb{E}^{(p)}\bigg[\sum_{v\in\cal U} \mathds{1}_{\{T^+_B(v)>|v|\}} \cdot ( \chi^{(p)}(v))^{\theta_\alpha-2\mu}
\overline{V}(\chi^{(p)}(v))\bigg]. 
\]

\noindent By the many-to-one formula,  we can bound the right-hand side from above as
\[
{\times(p)}
\le 
C \overline{V}(p) B^{2\mu}A^2 {\tt E}_p\bigg[\sum_{n=0}^{{\tt T}_{B}^+-1}  {\tt S}_n^{\theta_\alpha-2\mu}\bigg],
\]
and we apply  \cref{c:sum moments} (i) to complete the proof of (i).

{
\bigskip
\noindent \emph{\underline{Step 2}: Proof of item (ii), for $n=2$ and $g = \frac{h}{2}$.}
} 
In this case, we use \cref{prop: hitting time S} and the asymptotic behaviour \eqref{eq: mean volume 2} of $\overline{V}$ to deduce the bound, for all $r\ge 1$, 
\[\overline{V}(r) {\tt P}_{r}({\tt T}_M<{\tt T}_{B}^+)\le c   \frac{1+\ln_+\big(\frac{B}{r}\big)}{1+\ln(B)-\ln(M)} r^2.\]

\noindent Starting from \eqref{eq: cross term T_M<T_B}, this yields
\[
{E(q)}
\le  c  \bb E^{(q)} \Bigg[{\mathds{1}_{G_q}\cdot\Bigg(\sum^{\infty}_{i=1}(\chi^{(q)}(i))^2\bigg( \frac{1+\ln_+\Big(\frac{B}{\chi^{(q)}(i)}\Big)}{1+\ln(B)-\ln(M)}\bigg)\Bigg)^2}\bigg]. 
\]
In order to use the condition \eqref{eq: def G_q} in $G_q$, we now use the inequality, for $q\in [1,B]$
\begin{equation} \label{eq: inequality log Gq}
1+\ln_+\bigg(\frac{B}{\chi^{(q)}(i)}\bigg)
\le 
(1+\ln(B)-\ln q)  \left(1+\ln_+ \frac{q}{\chi^{(q)}(i)}\right).
\end{equation}
We end up with 
\begin{equation*}
{E(q)} \le \\
c  \left(\frac{1+\ln B -\ln q}{1+\ln B-\ln M}\right)^2 \bb E^{(q)} \bigg[{\mathds{1}_{G_q}\cdot\Bigg(\sum^{\infty}_{i=1}(\chi^{(q)}(i))^2\Big(1+\ln_+ \frac{q}{\chi^{(q)}(i)}\Big)\Bigg)^2}\bigg].
\end{equation*}
Finally, we use the event $G_q$ to obtain the bound
\[
{E(q)}
\le c \left(\frac{B}{q}\right)^{2\mu} A^2 q^4 \left(\frac{1+\ln B -\ln q}{1+\ln B-\ln M}\right)^2 \le  C \left(\frac{B}{q}\right)^{2\mu} A^2 q^2 \left(\frac{1+\ln B -\ln q}{1+\ln B-\ln M}\right)^2 \overline{V}(q).
\]

\noindent Going back to \eqref{eq: cross term gasket decomposition}, we deduce
\[
{\times(p)}
\le
C  B^{2\mu} A^2 \mathbb{E}^{(p)}\bigg[\sum_{v\in\cal A} \left(\frac{1+\ln B -\ln  \chi^{(p)}(v)}{1+\ln B-\ln M}\right)^2 \left( \chi^{(p)}(v)\right)^{2-2\mu}
\overline{V}(\chi^{(p)}(v))\bigg].
\]

\noindent The many-to-one formula then yields the upper-bound 
\begin{align*}
{\times(p)}
&\le
C  B^{2\mu} A^2 \overline{V}(p) {\tt E}_p\bigg[\sum_{n=0}^{\tt T_{B}^+ \wedge \tt T_M-1}
  \left(\frac{1+\ln B -\ln  \tt S_n}{1+\ln B-\ln M}\right)^2  {\tt S}_n^{2-2\mu}\bigg] \\
&\le
\frac{C  B^{2\mu} A^2 \overline{V}(p)}{(1+\ln B-\ln M)^2} {\tt E}_p\bigg[\sum_{n=0}^{ \tt T_{B}^+ \wedge \tt T_M-1} \bigg(\frac{B}{{\tt S}_n}\bigg)^\delta  {\tt S}_n^{2-2\mu}\bigg],
\end{align*}

\noindent where $\delta$ is some constant in $(0, 2-2\mu)$ and we used that $1+\ln x\le c x^{\frac{\delta}{2}}$ for $x \ge 1$. We conclude by Corollary \ref{c:sum moments} (ii) that
\begin{equation} \label{eq:cross terms final}
{\times(p)}
\le
C  B^{2} A^2 \frac{\overline{V}(p)}{(1+\ln B-\ln M)^2},
\end{equation}
{which concludes the proof of (ii).}

{
\bigskip
\noindent \emph{\underline{Step 3}: Proof of item (iii), for $n=2$ and $g < \frac{h}{2}$.}
} 
Likewise, using \cref{prop: hitting time S} and \eqref{eq: mean volume 2} we deduce the bound, for all $r\ge 1$, 
\[\overline{V}(r) {\tt P}_{r}({\tt T}_M<{\tt T}_{B}^+)\le c   \frac{1+\ln_+\big(\frac{B}{r}\big)}{1+\ln(B)-\ln(M)} \cdot \frac{r^2}{1 + \ln r}.\]

\noindent Starting from \eqref{eq: cross term T_M<T_B}, this yields
\begin{equation} \label{eq: second moment bound (iii)}
{E(q)}
\le  c  \bb E^{(q)} \Bigg[{\mathds{1}_{G_q}\cdot\Bigg(\sum^{\infty}_{i=1}\frac{(\chi^{(q)}(i))^2}{1+\ln\chi^{(q)}(i)}\bigg( \frac{1+\ln_+\Big(\frac{B}{\chi^{(q)}(i)}\Big)}{1+\ln(B)-\ln(M)}\bigg)\Bigg)^2}\Bigg].
\end{equation}
In order to use the condition \eqref{eq: def G_q} in $G_q$, we first combine \eqref{eq: inequality log Gq} with the inequality
\[
\frac{1}{1 + \ln \chi^{(q)}(i)}
\le 
\frac{1 + \ln_+ \Big(q/\chi^{(q)}(i)\Big)}{1 + \ln q},
\]
to get for $q\in [1,B]$,
\begin{equation}\label{eq: second order proof inequality iii}
\frac{1+\ln_+\Big(B/\chi^{(q)}(i)\Big)}{1 + \ln \chi^{(q)}(i)}
\le 
\frac{1+\ln(B)-\ln q}{1 + \ln q}  \Big(1 + \ln_+ \Big(q/\chi^{(q)}(i)\Big)\Big)^2.
\end{equation}

\noindent {Plugging the latter} inequality \eqref{eq: second order proof inequality iii} back into \eqref{eq: second moment bound (iii)}, we have
\begin{equation*}
{E(q)} \le \\
c  \left(\frac{1+\ln B -\ln q}{1+\ln B-\ln M}\right)^2{\frac{1}{(1+\ln q)^2}} \bb E^{(q)}\Bigg[{\mathds{1}_{G_q}\cdot\Bigg(\sum^{\infty}_{i=1}(\chi^{(q)}(i))^2\left(1 + \ln_+ \frac{q}{\chi^{(q)}(i)}\right)^2\Bigg)^2}\Bigg].
\end{equation*}
Finally, we use the event $G_q$ to obtain the bound
\begin{align*}
{E(q)}
&\le c \left(\frac{B}{q}\right)^{2\mu} A^2 \frac{q^4}{(1 + \ln q)^2} \left(\frac{1+\ln B -\ln q}{1+\ln B-\ln M}\right)^2 \\
&\le  c \left(\frac{B}{q}\right)^{2\mu} A^2 \frac{q^2}{1 + \ln q} \left(\frac{1+\ln B -\ln q}{1+\ln B-\ln M}\right)^2 \overline{V}(q).
\end{align*}

\noindent Going back to \eqref{eq: cross term gasket decomposition}, we deduce
\[
{\times(p)}
\le
C  B^{2\mu} A^2 \mathbb{E}^{(p)}\bigg[\sum_{v\in\cal A} \left(\frac{1+\ln B -\ln  \chi^{(p)}(v)}{1+\ln B-\ln M}\right)^2 \frac{\left(\chi^{(p)}(v)\right)^{2-2\mu}}{1 + \ln \chi^{(p)}(v)}
\overline{V}(\chi^{(p)}(v))\bigg].
\]

\noindent The many-to-one formula (\cref{prop: many to one}) then yields the upper-bound 
\[
{\times(p)}
\le
C  B^{2\mu} A^2 \overline{V}(p) {\tt E}_p\bigg[\sum_{n=0}^{\tt T_{B}^+ \wedge \tt T_M-1}
  \left(\frac{1+\ln B -\ln  \tt S_n}{1+\ln B-\ln M}\right)^2 \frac{{\tt S}_n^{2-2\mu}}{1 + \ln {\tt S}_n}\bigg]. 
\]
Now choose $\delta\in(0,2-2\mu)$, and further bound the above display using the two inequalities: $1+\ln x\le c x^{{\delta/3}}$ for $x\ge 1$, as well as 
\[\frac{1}{1 + \ln {\tt S}_n} = \frac{1}{1 + \ln B}\bigg(1 + \frac{\ln(B/{\tt S}_n)}{1 + \ln {\tt S_n}}\bigg)\leq \frac{c}{1+\ln B - \ln M}\bigg(\frac{B}{{\tt S}_n}\bigg)^{{\delta/3}}. \]
We end up with
\[
{\times(p)}
\le
\frac{C  B^{2\mu} A^2 \overline{V}(p)}{(1+\ln B-\ln M)^3} {\tt E}_p\bigg[\sum_{n=0}^{ \tt T_{B}^+ \wedge \tt T_M-1} \bigg(\frac{B}{{\tt S}_n}\bigg)^{\delta} {\tt S}_n^{2-2\mu}\bigg].
\]

\noindent We conclude by Corollary \ref{c:sum moments} (ii) that
\begin{equation} \label{eq:cross terms final 2}
{\times(p)}
\le
C  B^{2} A^2 \frac{\overline{V}(p)}{(1+\ln B-\ln M)^3}.
\end{equation}
{This proves (iii).}
\end{proof}

%
%
\section{Proof of the scaling limit (\cref{thm: main})}
\label{sec: proof main result}

We establish the scaling limit result for the volume. The strategy is to approximate the volume by its conditional expectation at some smaller generation $\ell$, and then relate it to either the additive martingale when $n\in (0,2)$ or to the derivative martingale when $n=2$.  

\bigskip

Recall that we defined in \cref{sec: good or bad} a set of branch points $\cal N$ and good points $\cal G$ associated to some parameters $(B,M,A,L)$. {Fix a $\gamma\in (0, \theta_{\alpha} - 2\mu)$ in \cref{prop: hitting S infinite} such that \cref{thm: second moment good volume} holds. } We will write from now $B$ as $bp$ for some $b\ge 1$. {Let $V_{\mathcal G}(u)$ stand for the volume of the maps associated to the 
good points $v\in \mathcal G$ in the lineage of $u$, namely 
\[
V_{\mathcal G}(u)
:=
\sum_{w\in\mathcal{U}: \; uw\in \mathcal{G}} V(uw).
\]
} 
For $\ell\ge 1$, let 
\begin{align*}
V^\ell &:= \sum_{|u|=\ell} V(u) && V_{\mathcal G}^\ell:= \sum_{|u|=\ell} V_{\mathcal G}(u) \\ 
\overline{V}^\ell &:= \bb E^{(p)}[V^\ell | \mathcal F_\ell]  &&  \overline{V}_{\mathcal G}^\ell:= \bb E^{(p)}[V_{\mathcal G}^\ell | \mathcal F_\ell]. 
\end{align*}

\noindent By the gasket decomposition,
\[
\overline{V}^\ell = \sum_{|u|=\ell} \overline{V}(\chi^{(p)}(u)), \, \qquad \overline{V}_{\mathcal G}^\ell=\sum_{|u|=\ell, \, u\in \mathcal A} \bb E^{(\chi^{(p)}(u))}[V_{\mathcal G}]
\]

\noindent where, in agreement with the notation in the proof of \cref{thm: second moment good volume}, $\mathcal A$ denotes the set of labels $u$ which have moderate increments \eqref{eq: size increments} and   $\chi^{(p)}(u_k) \in [M,bp]$ for all $1\le k\le |u|$.  In the following theorem, we say that $f(b,M,A,L,\ell,p)$ with values in a metric space $({\mathcal E},\mathrm{d})$ converges to some element $x_f \in {\mathcal E}$ as $p,\ell,L,A,M$ and $b$ go successively to $\infty$ if 
\[
\limsup_{b\to\infty}  \limsup_{M\to\infty} \limsup_{A\to\infty} \limsup_{L\to\infty} \limsup_{\ell\to\infty} \limsup_{p\to\infty} \;  \mathrm{d}(f(b,M,A,L,\ell,p),x_f)=0.
\]

\noindent We will use that convergence in distribution is metrizable, for example with the Prohorov metric. Note that under this metric, if some r.v. $(X_n,Y_n)$ satisfy  $X_n-Y_n\overset{(\mathrm{d})}{\longrightarrow}  0$, then the distance between $\bb P_{X_n}$ and $\bb P_{Y_n}$ goes to $0$, see \cite[Section 6]{billingsley2013convergence}. Theorem \ref{thm: main} is then a direct consequence of the following theorem.

\begin{thm}\label{thm:tuning parameters}
For simplicity, write $\overline{v}(p):= \overline{V}(p)$ if $n\in (0,2)$ and $\overline{v}(p)=\frac{\overline{V}(p)}{\ln(p)}$ if $n=2$.  The following convergences hold as $p,\ell,L,A,M$ and $b$ go successively to $\infty$:

 \begin{align}
& \frac{1}{\overline{v}(p)} (V - V_{\cal G}^{\ell})  \begin{cases} \overset{(L^1)}{\longrightarrow}  0 & \textrm{if } n\in (0,2), \\
\overset{(\mathrm{d})}{\longrightarrow}  0 & \textrm{if } n=2,
\end{cases}
\label{eq: tuning 1}\\
& \frac{1}{\overline{v}(p)} (V_{\cal G}^{\ell} - \overline{V}_{\cal G}^{\ell})  \overset{(\mathrm{d})}{\longrightarrow}  0, \label{eq: tuning 2}\\
& \frac{1}{\overline{v}(p)} \overline{V}_{\cal G}^{\ell}  \overset{(\mathrm{d})}{\longrightarrow}  \begin{cases} W_\infty & \textrm{if } n\in (0,2), \\
 D_\infty & \textrm{if } n=2.
\end{cases}  \label{eq: tuning 3}
 \end{align}
\end{thm}

{
\noindent The proof of the three scaling limits are sketched in Figures~\ref{fig: proof diagram 1}, \ref{fig: proof diagram 2} and \ref{fig: proof diagram 3}, where we outline how we combine all our estimates from Sections~\ref{sec: markov chain estimates} and \ref{sec: classification}.}

\begin{proof}
{We prove the three scaling limits one after another.}

{
\bigskip
\noindent \emph{\underline{Step 1}: Proof of \eqref{eq: tuning 1}.}
} 
Recall from \cref{sec: gasket estimates} the notation $\frak g_{B,M}$ standing for the gasket outside loops exiting $[M,B]$, {and that we have set $B=bp$ with $b\ge 1$}. Observe that 
\begin{align*}
    V & =V_{\cal N}+  \sum_{u \in \cal U} V(u)\mathds{1}_{\{T_{bp}^+(u)=|u|\} \cap \{T_M(u)>|u|\}}  + |\frak{g}_{{bp},M}| \\
    &= V_{\cal G}+ V_{\cal B} +  \sum_{u\in \cal U} V(u)\mathds{1}_{\{T_{bp}^+(u)=|u|\} \cap \{T_M(u)>|u|\}} + |\frak{g}_{bp,M}|.
\end{align*}
{We now distinguish according to the cases $n\in(0,2)$ or $n=2$.}

\bigskip
{
\emph{Let us start with the case $n\in(0,2)$.} We first take a look at $\frac{1}{\overline{V}(p)}(V-V_{\cal G})$ using the above decomposition and show that it goes to $0$ in $L^1$. The first term goes to $0$ in $L^1$ by \cref{thm: volume bad}. 
So does the last term by \cref{thm: gasket estimate} item (i), since $|\frak{g}_{bp,M}| \le |\frak{g}_{M}|$. Finally, the second one can be bounded by forgetting the second condition in the indicator and using \cref{prop: EB}:
\[
\frac{1}{\overline{V}(p)} \bb E^{(p)}\Big[\sum_{u\in \cal U} V(u)\mathds{1}_{\{T_{bp}^+(u)=|u|\} \cap \{T_M(u)>|u|\}}\Big]
\leq
{\tt P}_p({\tt T}_{bp}^+<\infty).
\]
We conclude by \cref{prop: hitting S infinite} that this also goes to $0$. 
Hence $\frac{1}{\overline{V}(p)}(V-V_{\cal G})$ goes to $0$ in $L^1$.
}

Moreover,  $V_{\cal G}-V_{\cal G}^\ell \le \sum_{|u|<\ell} V(u)\mathds{1}_{\{u\in \cal N\}}$. The expectation of the latter is, by the many-to-one formula in \eqref{eq: many-to-one},
\begin{equation}\label{proof tuning expectation N ell}
\bb E^{(p)}\bigg[\sum_{|u|<\ell} V(u)\mathds{1}_{\{u\in \cal N\}} \bigg]=\overline{V}(p){\tt P}_p({\tt T}_M<\min(\ell,{\tt T}_{bp}^+)) \le \overline{V}(p){\tt P}_p({\tt T}_M<\ell).
\end{equation}

\noindent We conclude by \cref{prop: scaling limits S} that $\frac{1}{\overline{V}(p)}(V - V_{\cal G}^{\ell})$ converges to $0$ in $L^1$. 

\bigskip
{\emph{The case $n=2$}} is similar, except that we reason under ${\mathcal E}(B)={\mathcal E}(bp)$ (recall the definition of this event in \eqref{def:EB}). {Indeed,} to prove convergence in distribution, we are allowed to restrict to this event by {the second claim of} \cref{prop: EB}. On ${\mathcal E}(bp)$, we have $V=V_{\cal G}+ V_{\cal B}+|\frak{g}_{bp,M}|$ so that $\frac{\ln(p)}{\overline{V}(p)}(V-V_{\cal G})$ converges to $0$ in distribution by \cref{thm: volume bad} and \cref{thm: gasket estimate} item (ii). It remains to control the expectation of $V_{\cal G}-V_{\cal G}^\ell\le \sum_{|u|<\ell} V(u)\mathds{1}_{\{u\in \cal N\}}$. By \cref{coupling 2}, for $p$ large enough,
\[
{\tt P}_p({\tt T_M<\ell}) 
\le c p^{-\frac{1}{32}} + {\tt P}_p\Big(\{\tt T_M<\ell\} \cap \Big\{\forall 0\le k\le \sigma_{\sqrt{p}}, \; \frac{1}{2} \le \frac{\tt S_k}{Y_k} \le 2 \Big\}\Big)
\le c p^{-\frac{1}{32}} + {\tt P}_p(\sigma_{\sqrt{p}}<\ell).
\]

\noindent Moreover, by a union bound, ${\tt P}_p(\sigma_{\sqrt{p}}<\ell) \le  \ell \, \bb P(\xi < p^{-\frac{1}{2\ell}}) = \frac{2}{\pi} \ell  \arctan \big(p^{-\frac{1}{4\ell}}\big)$, hence ${\tt P}_p({\tt T_M<\ell})$ decays polynomially fast as $p\to\infty$.
\cref{proof tuning expectation N ell} then implies that

\begin{equation}\label{proof tuning N ell}
\lim_{p\to\infty} \frac{\ln(p)}{\overline{V}(p)}\bb E^{(p)}[ V_{\cal G}-V_{\cal G}^\ell]=\lim_{p\to\infty} \frac{\ln(p)}{\overline{V}(p)}\bb E^{(p)}\bigg[\sum_{|u|<\ell} V(u)\mathds{1}_{\{u\in \cal N\}} \bigg]=0.
\end{equation}

\noindent It implies the convergence in distribution of $\frac{1}{\overline{v}(p)}(V - V_{\cal G}^{\ell})$ to $0$.

\begin{figure}[t]
\bigskip
\begin{center}
\includegraphics[scale=0.8]{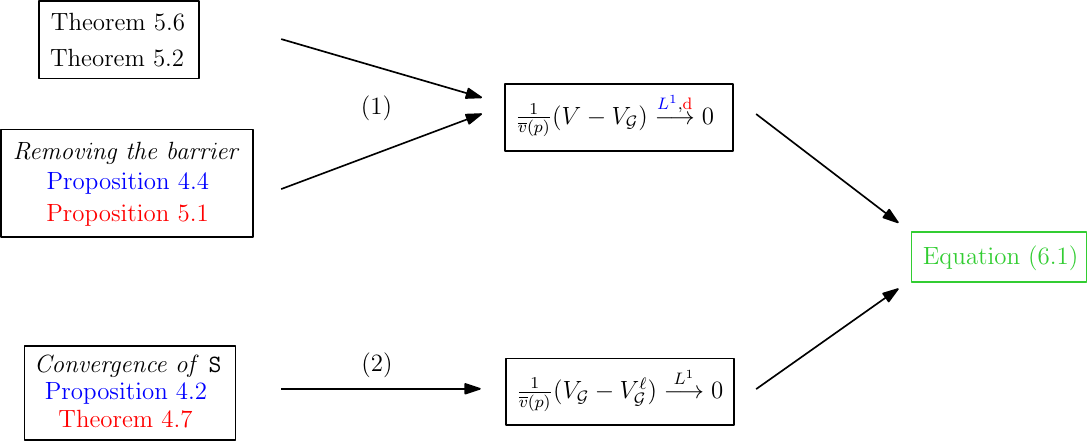}
\end{center}
\caption{
{A diagram for the proof of \eqref{eq: tuning 1}. Blue and red colours correspond to the cases $n\in(0,2)$ and $n=2$ respectively. The statement easily follows from the two boxes in the middle. 
(1) To prove the top one, we first notice that the difference $V-V_{\cal G}$ accounts for the volume of the bad region, the volume of the gasket $\mathfrak{g}_{bp,M}$ and that inside loops that hit the upper barrier $B=bp$. We handle the first two using \cref{thm: volume bad} and \cref{thm: gasket estimate}, respectively. For the latter one, we need to remove the barrier. When $n\in(0,2)$, the volume can be controlled using \cref{prop: hitting S infinite}. However, for $n=2$, because of the logarithmic asymptotics in \cref{prop: hitting time S}, we need to restrict to the nice event $\mathcal{E}(bp)$ where no loop reaches the barrier (\cref{prop: EB}), which is why we only get a convergence in distribution. 
(2) For the bottom box, we first bound the difference $V_{\cal G}-V_{\cal G}^\ell$ by the volume carried by branch points at generation smaller than $\ell$. This translates into hitting time estimates for $\tt S$ after applying the many-to-one formula, which are then handled using the convergence of $\tt S$. Note that for $n=2$ we need refined estimates, coming from our coupling in \cref{sec:hitting n=2}.}
}
\label{fig: proof diagram 1}
\end{figure}

{
\bigskip
\noindent \emph{\underline{Step 2}: Proof of \eqref{eq: tuning 2}.}
} 
It suffices to show that 
\begin{equation}\label{proof tunin 2 ne 2}
 \frac{1}{\overline{v}(p)} \bb E^{(p)}[|V_{\mathcal G}^\ell-\overline{V}^{\ell}_{\mathcal G}|\land \overline{v}(p)],
\end{equation}

\noindent goes to $0$.  By the Cauchy--Schwarz inequality, $\bb E^{(p)}[|V^\ell_{\mathcal G}-\overline{V}^{\ell}_{\mathcal G}|\land \overline{v}(p) | \mathcal{F}_{\ell}] \le \overline{v}(p)\land \bb E^{(p)}[|V^\ell_{\mathcal G}-\overline{V}^{\ell}_{\mathcal G}|^2| \mathcal{F}_{\ell}]^{1/2}$. Hence
\begin{equation}\label{eq: proof V^ell_G minus mean}
\bb E^{(p)}\big[|V^\ell_{\mathcal G}-\overline{V}^{\ell}_{\mathcal G}|\land \overline{v}(p)\big]
\le
\bb E^{(p)} \Big[ \overline{v}(p)\land \bb E^{(p)}[|V^\ell_{\mathcal G}-\overline{V}^{\ell}_{\mathcal G}|^2| \mathcal{F}_{\ell}]^{1/2} \Big].
\end{equation}

\begin{figure}
\begin{center}
\includegraphics[scale=0.8]{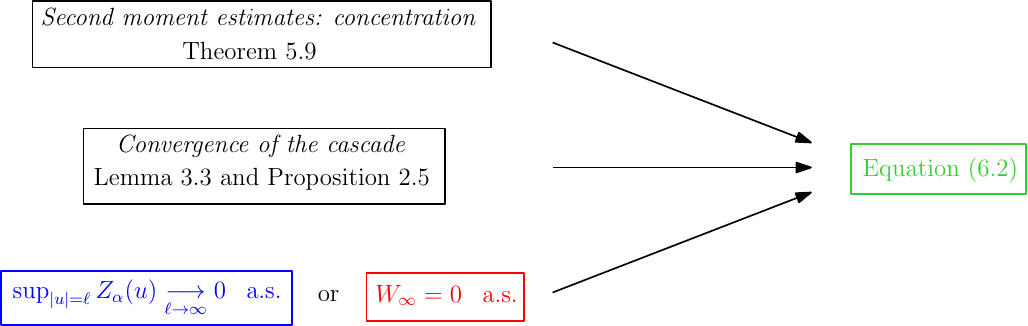}
\end{center}
\caption{
{A diagram for the proof of \eqref{eq: tuning 2}. To prove \eqref{eq: tuning 2}, we need to show that the good volume $V_{\cal G}$ concentrates around its mean. Our second moment estimates in \cref{thm: second moment good volume} allows us to bound the fluctuations around the mean. We then send $p\to \infty$ using \cref{l: passing to infinity} (and \cref{prop: CCM convergence fixed gen}) so that we are left with the continuous multiplicative cascade. The fact that $\sup_{|u|=\ell} Z_{\alpha}(u) \to 0$ ($n\in(0,2)$) or $W_\infty=0$ ($n=2$) finally shows that the upper bound goes to $0$.}
}
\label{fig: proof diagram 2}
\end{figure}

\noindent {Recall the set $\cal A$ in \cref{def:calA} listing the properties of any ancestor of two distinct good loops.} By independence of the submaps in the gasket decomposition (\cref{prop: spatial markov gasket}), 
\[
    \mathbb{E}^{(p)}[|V^\ell_{\mathcal G}-\overline{V}^{\ell}_{\mathcal G}|^2 | \mathcal{F}_{\ell}] 
 =\sum_{|u| = \ell, \, q = \chi^{(p)}(u)} \mathds{1}_{\{u\in\mathcal{A}\}} \mathbb{E}^{(q)}\Big[ \big(V_{\mathcal{G}} - {\bb E}^{(q)}[V_{\mathcal G}] \big)^2\Big] 
  \le 
  \sum_{|u| = \ell, \, q=\chi^{(p)}(u)} \mathds{1}_{\{u\in\cal A\}} \mathbb{E}^{(q)}[V_{\mathcal{G}}^2].
\]
{Given that our estimates in \cref{thm: second moment good volume} depend on the different regimes for $(n;g,h)$, we now sub-divide the proof.}
 
\medskip
{\emph{Let us first discuss the case $n\in (0,2)$.}} By \cref{thm: second moment good volume}, we have for all $q\in [M,bp]$ that  $ \mathbb{E}^{(q)}[V_{\mathcal{G}}^2]\le   \overline{V}(q) ( C A^2 b^{\theta_\alpha-\gamma} p^{\theta_\alpha} (q/p)^\gamma + L^2)$.  Hence the last display is smaller than 
\[
 \sup_{|u|=\ell,\, q = \chi^{(p)}(u)} (C_{A,b}p^{\theta_\alpha}(q/p)^\gamma  + L^2)\sum_{|u| = \ell, \, q = \chi^{(p)}(u)} \overline{V}(q),
\]

\noindent where $C_{A,b} = C A^2 b^{\theta_\alpha-\gamma}$.
{Recall from \eqref{eq: def martingale W_ell} the definition of the additive martingale $W_{\ell}$. In view of \cref{prop: CCM convergence fixed gen} and \cref{l: passing to infinity}, for some other constant $C_{A,b}'$ depending only on $A$ and $b$, we have
\begin{equation} \label{eq:cvg sup ref}
    \frac{1}{\overline{V}(p)^2}\sup_{|u|=\ell,\, q = \chi^{(p)}(u)} (C_{A,b}p^{\theta_\alpha}(q/p)^\gamma  + L^2)\sum_{|u| = \ell, \, q = \chi^{(p)}(u)} \overline{V}(q) \xrightarrow[p\to\infty]{(\mathrm{d})} C_{A,b}' W_\ell \sup_{|u|=\ell}   Z_{\alpha}(u)^\gamma. 
\end{equation}
Recalling \eqref{eq: extremal asympt}, we see that $\sup_{|u|=\ell} Z_{\alpha}(u)$ goes to $0$ almost surely when $\ell\to \infty$, and therefore so does the right-hand side in the above display.} 
{We stress once more that $p$ is sent to infinity before all other parameters.} 
Hence \eqref{proof tunin 2 ne 2} goes to $0$ by dominated convergence. 

\medskip
{\emph{The case $n=2$}} follows analogous lines. When $g=\frac{h}{2}$, \cref{thm: second moment good volume} (ii) gives that $\mathbb{E}^{(q)}[V_{\mathcal{G}}^2]\le  \overline{V}(q)(C_{A,b}  p^2 (\ln p)^{-2} + L^2)$ for $p\ge M^2$ {and all $q\in [1,B]$}. When $g<\frac{h}{2}$, \cref{thm: second moment good volume} (iii) gives that $\mathbb{E}^{(q)}[V_{\mathcal{G}}^2]\le  \overline{V}(q)(C_{A,b}  p^2 (\ln p)^{-3} + L^2)$ for $p\ge M^2$ {and all $q\in [1,B]$}. By the asymptotics \eqref{eq: mean volume 2}, in both cases, {there are constants $C''_{A, b}$ depending only on $(A, b)$ and $C_L$ depending only on $L$ such that} for $p\geq M^2$, 

\begin{equation} \label{eq: ref log corrections}
\frac{1}{\overline{v}(p)^2}\sum_{|u| = \ell, \, q=\chi^{(p)}(u)} \mathds{1}_{\{u\in\cal A\}} \mathbb{E}^{(q)}[V_{\mathcal{G}}^2]\leq \sum_{|u| = \ell, \, q=\chi^{(p)}(u)}\frac{\overline{V}(q)}{\overline{V}(p)}\bigg({C''_{A, b}} + {C_L}\frac{{(\ln p)^3}}{p^2}\bigg). 
\end{equation}

\noindent \cref{l: passing to infinity} {and \eqref{eq: proof V^ell_G minus mean}} yield that {there exists a constant $C_{A,b}$ depending only on $(A,b)$ such that} 
\[
\limsup_{p\to\infty} \frac{1}{\overline{v}(p)} \bb E^{(p)}[|V_{\mathcal G}^\ell-\overline{V}^{\ell}_{\mathcal G}|\land \overline{v}(p)] \le {C_{A,b}} \bb E\Big[ 1\land W_\ell^{1/2} \Big],
\]

\noindent which goes to $0$ when $\ell\to\infty$ {by dominated convergence,} since $W_\ell\to 0$ almost surely (see  \cref{sec: intro prev known}). 
{This again says that \eqref{proof tunin 2 ne 2} goes to $0$.}

\begin{figure}[t]
\begin{center}
\includegraphics[scale=0.8]{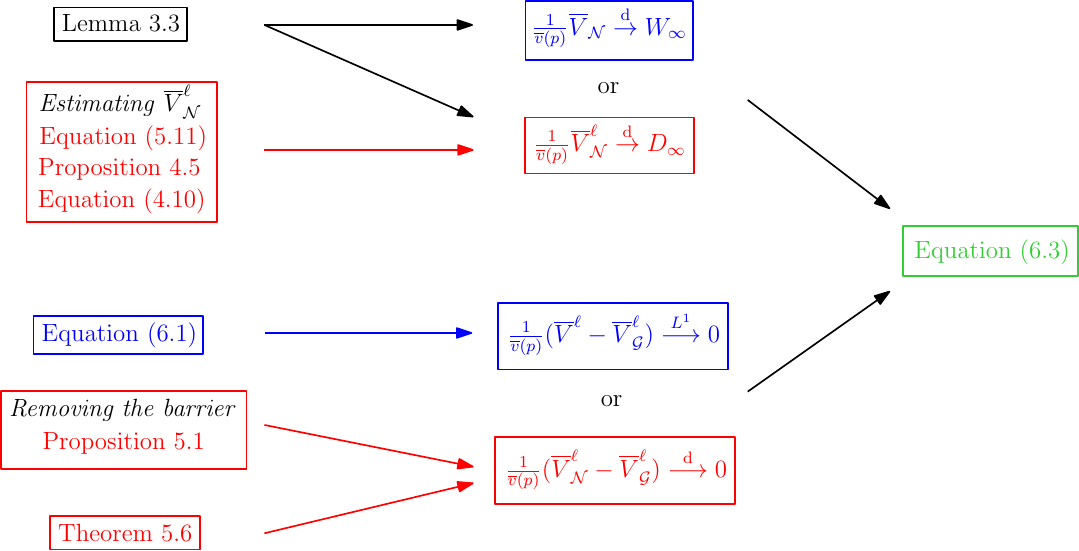}
\end{center}
\caption{
{A diagram for the proof of \eqref{eq: tuning 3}. The result easily follows provided we show the scaling limits in the middle. In the case when $n\in (0,2)$, depicted in blue, the proof is almost immediate: the top box comes for free using \cref{l: passing to infinity}, while the bottom one is a consequence of the $L^1$ convergence in \eqref{eq: tuning 1}. Because $W_\infty=0$ almost surely, this approach is not precise enough when $n=2$ (red). We introduce the volume $V_{\cal N}^{\ell}$ of the branch points \emph{inside} the maps at generation $\ell$. For the convergence towards the derivative martingale (top box), we compute its mean $\overline{V}_{\cal N}^{\ell}$ using \eqref{eq: expected volume branch points}. This involves the hitting time probability that we have estimated in \cref{prop: hitting time S}, up to the renewal asymptotics \eqref{def:c0}. From the latter, a logarithmic correction emerges, and we end up with the derivative martingale. To obtain the red box at the bottom, we first need to restrict to the nice event $\mathcal{E}(bp)$ where no loop reaches the upper barrier $B=bp$ (\cref{prop: EB}), which explains why we only have a convergence in distribution in this case. On this event, the difference $V_{\cal N}^\ell-V_{\cal G}^\ell$ is essentially given by the bad volume, that scales to $0$ in $L^1$ by \cref{thm: volume bad}.}
}
\label{fig: proof diagram 3}
\end{figure}

{
\bigskip
\noindent \emph{\underline{Step 3}: Proof of \eqref{eq: tuning 3}.}
} 
{We need to split the two cases $n\in (0,2)$ and $n=2$ right away.}

\medskip
{\emph{In the case $n\in (0,2)$}}, by \cref{l: passing to infinity}, the limit in distribution of $\frac{\overline{V}^\ell}{\overline{V}(p)}$ as $p\to\infty$ then $\ell\to\infty$ is $W_\infty$. So we want to show  that $\frac{1}{\overline{V}(p)}(\overline{V}^\ell-\overline{V}_{\cal G}^\ell)$ goes to $0$. But its expectation is
\[
\frac{1}{\overline{V}(p)}{\bb E}^{(p)}[V^\ell-V_{\cal G}^\ell]\le \frac{1}{\overline{V}(p)}{\bb E}^{(p)}[V-V_{\cal G}^\ell],
\]

\noindent which goes to $0$ by \eqref{eq: tuning 1}. 

\medskip
{\emph{We now consider the case $n=2$.}} 
For $u\in \cal U$, let ${\frak m}_u$ denote the (loop-decorated) map inside the loop labelled by $u$. We may define as in \cref{sec: good or bad} (and relative to the same constants $B$ and $M$) the set of branch points $\cal N({\frak m}_u) \subset u\cdot\cal U$ for the map $\frak m_u$: beware that $\cal N({\frak m}_u)$ may include labels which are not in ${\cal N}={\cal N}(\frak m_{\varnothing})$, for example in the case $T_M(u)<\ell$. Introduce
\[
V_{\cal N}({\frak m}_u) := \sum_{v\in \cal N(\frak m_u)} V(v),
\]
and $V_{\cal N}^\ell:=\sum_{|u|=\ell} V_{\cal N}({\frak m}_u)$. Let $\overline{V}_{\cal N}^\ell:= {\bb E}^{(p)}[V_{\cal N}^\ell | {\mathcal F}_\ell]$. We have by the gasket decomposition and \eqref{eq: expected volume branch points},
\begin{equation} \label{eq: scaled V_calN^l}
\frac{\ln(p)}{\overline{V}(p)}\overline{V}_{\cal N}^\ell = \frac{\ln(p)}{\overline{V}(p)} \sum_{|u| = \ell} \overline{V}(\chi^{(p)}(u)) {\tt P}_{\chi^{(p)}(u)}({\tt T}_M<{\tt T}_{bp}^+, \tt S_{\tt T_M} \ne 0).
\end{equation}

\noindent By \eqref{eq: expectation gasket P(S_1=0)} and \cref{thm: gasket estimate} (ii), {for all $q\ge 1$,}
\[
\overline{V}(q) {\tt P}_{q}({\tt T}_M<{\tt T}_{bp}^+, \tt S_{\tt T_M} = 0) = \bb E^{(q)}[|{\frak g}_{bp,M}|]\le C M^{-\beta_\alpha} \overline{v}_{bp,M}(q) {\mathds{1}_{\{q\le bp\}}}.
\]

\noindent Using this inequality in $\sum_{|u| = \ell} \overline{V}(\chi^{(p)}(u)) {\tt P}_{\chi^{(p)}(u)}({\tt T}_M<{\tt T}_{bp}^+, \tt S_{\tt T_M} = 0)$ and applying \cref{l: passing to infinity} with $h_p(x)=\ln(p){\frac{\overline{v}_{bp,M}(px)}{\overline{V}(px)}\mathds{1}_{\{x\le b\}}}$ 
shows that the limit of 
\[
\frac{\ln(p)}{\overline{V}(p)} \sum_{|u| = \ell} \overline{V}(\chi^{(p)}(u)) {\tt P}_{\chi^{(p)}(u)}({\tt T}_M<{\tt T}_{bp}^+, \tt S_{\tt T_M} = 0),
\]

\noindent is $0$. This proves that we can get rid of the event $\{\tt S_{\tt T_M}\ne 0\}$ in \eqref{eq: scaled V_calN^l}.
Recall that $R$ denotes the renewal function defined in \eqref{def: renewal function R}. Using, in view of \cref{prop: hitting time S}, the convergence stated in \cref{l: passing to infinity}  with $h_p(x)= \ln(p) {\tt P}_{px}({\tt T}_M<{\tt T}_{bp}^+)$ we thus get, as $p\to \infty$,
\[
\frac{\ln(p)}{\overline{V}(p)}\overline{V}_{\cal N}^\ell 
\overset{(\mathrm{d})}{\longrightarrow}
{\frac{1}{c_0}}\sum_{|u|=\ell} R\bigg(\ln \frac{b}{Z_\alpha(u)}\bigg) Z_\alpha(u)^2.
\]
Since $\sup_{|u|=\ell} Z_\alpha(u) \to 0$ as $\ell \to 0$, we can then use the asymptotics \eqref{def:c0} and the fact that $W_\infty = 0$ to conclude that $\frac{\ln(p)}{\overline{V}(p)}\overline{V}_{\cal N}^\ell$ goes to {$D_{\infty}$} upon letting first $p\to\infty$ and then $\ell\to\infty$.
Finally, it remains to prove that $\frac{\ln(p)}{\overline{V}(p)}(\overline{V}_{\cal N}^\ell-\overline{V}_{\cal G}^\ell)$ goes to $0$ {in distribution}. For this, \cref{prop: EB} enables us to restrict to the event $\cal E(bp)$. Now on the event $\cal E(bp)$, we notice that a node $v \in \cal N(\frak m_u)$ for some $u\in \cal U$ with $|u|=\ell$ is either in $\cal N = \cal N(\frak m_{\varnothing})$ or the descendant of some branch point $w\in \cal N$ with $|w|< \ell$. Thus
\[
\bb E^{(p)}\big[\mathds{1}_{{\cal E}(bp)} (\overline{V}_{\cal N}^\ell-\overline{V}_{\cal G}^\ell)\big]
\le 
\bb E^{(p)}\bigg[V_{\cal B}+\sum_{|w|<\ell} \mathds{1}_{\{ w \in \cal N\}} V(w)\bigg],
\]
and we conclude by \cref{thm: volume bad} and \eqref{proof tuning N ell}. 
\end{proof}

\appendix

\section{Partition function of the non-generic critical $O(n)$ loop model}
\label{s:partition function O(2)}

We review results on the rigid $O(n)$ loop model on quadrangulations for $n\in (0,2)$, {referring to} \cite{borot2011recursive} and the Appendix of \cite{budd2018peeling} for details. 

\bigskip

Let $n\in (0,2)$ and $b:=\frac1\pi \arccos \frac{n}2$. Recall that the loop--$O(n)$ model depends on some parameters $(g,h)$ which satisfy \eqref{non-generic line} in the non-generic critical regime. Recall that $F_p(n;g,h)$ in \eqref{eq: weights O(n)} denotes the partition function of the model, and satisfies \eqref{eq: asymptotics partition function} {in that regime}. The resolvent function  is defined by
\begin{equation}{\label{resolvent}}
W(\xi;n) := \sum_{p\geq 0} \frac{F_p(n;g,h)}{\xi^{2p+1}}, \qquad |\xi|> \gamma:= \frac1{\sqrt{h}}.
\end{equation}

\noindent It has an analytic continuation for $\xi \notin [-\gamma,\gamma]$ and has the explicit expression $W = W_{part} + W_{hom}$ where
\[
W_{part}(\xi;n) = 
\frac{2(\xi-g\xi^3) - n(\frac{1}{h^2\xi^3} - \frac{g}{h^4\xi^5})}{4-n^2} + \frac{n}{(2+n)\xi},
\]
and
\[
W_{hom}(\xi;n) = \bigg(B(\xi) - \frac{\gamma^2}{\xi^2} B\bigg(\frac{\gamma^2}{\xi}\bigg)\bigg)\bigg(\frac{\xi-\gamma}{\xi+\gamma}\bigg)^b - \bigg(B(-\xi) - \frac{\gamma^2}{\xi^2} B\bigg(-\frac{\gamma^2}{\xi}\bigg)\bigg)\bigg(\frac{\xi+\gamma}{\xi-\gamma}\bigg)^b,
\]

\noindent with \[
B(\xi) = \frac{g}{4-n^2}\bigg(\xi^3 + 2b\gamma\xi^2 + 2b^2\gamma^2\xi + \frac{2}{3}(b+2b^3)\gamma^3\bigg) - \frac{1}{4-n^2}(\xi+2b\gamma).
\]

We introduced the gasket decomposition in \cref{sec: gasket decomposition}, where it was recalled that the gasket of a loop--$O(n)$ decorated quadrangulation  is a Boltzmann map associated to the weight sequence  $\bf{\hat{g}} = (\hat{g}_k,k\ge 1)$  given by \cref{eq: fixed point equation}. 
This weight sequence is admissible for the pointed Boltzmann planar map model, so that it gives rise to a probability distribution on the space of maps. It is equivalent to \cref{eq:admissible} having a positive solution.  It is also critical meaning that the probability distribution $\mu_{JS}$ defined in \eqref{def:muJS} has mean $1$, \textit{i.e.} \
\begin{equation}\label{critical}
\sum_{k=1}^\infty k \dbinom{2k-1}{k-1} \hat{g}_k   Z_{\hat{\bf g}}^{k-1}=1. 
\end{equation}

\noindent In general, the following result is a criterion for a weight sequence to be admissible and critical.

\begin{prop}[\cite{budd2016peeling}, {Proposition 3}]\label{p: admissible critical}
A weight sequence $({\hat{g}}_k,\,k\ge 1)$ is admissible {and critical} for a pointed Boltzmann planar map model iff there exists a probability law $\nu$ on $\mathbb{Z}$ such that ${\hat{g}}_k = (\frac{\nu(-1)}{2})^{k-1}\nu(k-1)$, and $h^{\uparrow}$ is $\nu$--harmonic on $\mathbb{Z}_{>0}$, where $h^{\uparrow}(k) = 2k\, 4^{-k}\dbinom{2k}{k}$. 
\end{prop}

\noindent In our case,  $h^{\uparrow}$ is $\nu(\, \cdot\, ;n)$--harmonic on $\mathbb{Z}_{>0}$, for the probability measure $\nu(\, \cdot\, ;n)$ defined by, see  Equation (16)  in \cite{budd2016peeling},
\begin{equation}\label{nu n}
\nu(k;n) = \begin{cases}
{\hat{g}}_{k+1}\gamma^{2k},\quad & k\geq 0, \\
2F_{-k-1}(n;g,h)\gamma^{2k}, \quad & k\leq -1.
\end{cases}
\end{equation}

\noindent The spectral density function $\rho(u;n)$, $u \in[-\gamma, \gamma]$, is given by
\begin{align}\label{rho n}
\rho(u;n) &= \frac{1}{2\pi i}(W(u-i0;n)-W(u+i0;n))  \\
&= -\frac{\sin\pi b}{\pi} \left[\bigg(B(u)-\frac{\gamma^2}{u^2} B\bigg(\frac{\gamma^2}{u}\bigg)\bigg) \bigg(\frac{\gamma-u}{\gamma+u}\bigg)^b + \bigg(B(-u)-\frac{\gamma^2}{u^2} B\bigg(-\frac{\gamma^2}{u}\bigg)\bigg) \bigg(\frac{\gamma+u}{\gamma-u}\bigg)^b\right]\nonumber.
\end{align}

\noindent {As seen from \cite[Equation (6.1)]{borot2011recursive},} we have for any $|r|<\frac{1}{4h}$,
$$
\sum^{\infty}_{k=1}  k \dbinom{2k}{k} F_k(n;g,h) r^{k-1} =  \int_{-\gamma}^\gamma \rho(u;n) \frac{u^2}{\left(1-4r u^2\right)^{3/2}}  \mathrm{d} u.
$$

\noindent Take the limit $r\to \frac{1}{4h}$ to see that equation \eqref{critical} now reads
\begin{equation}\label{critical n}
    \frac{nh^2}{2} \int_{-\gamma}^\gamma \rho(u) \frac{u^2}{\left(1-h u^2\right)^{3/2}} \mathrm{d} u=1-\frac32 h^3g.
\end{equation}

\section{Construction of a non-generic critical $O(2)$ loop model}

\label{s:construction O(2)}
We suppose that the parameters $(g,h)$ satisfy \eqref{n=2 line}. The parameter $h$ then satisfies
$\frac{4}{3\pi^2} \leq h \leq \frac{2}{\pi^2}$, and we can check that for any such $h$, for any $n<2$, the parameter $g_n$ obtained from  \eqref{non-generic line}  exists. Therefore, we are allowed to apply the results of the previous section to any such $h$. From now on, we fix $h\in [\frac{4}{3\pi^2},\frac{2}{\pi^2}]$, $\gamma=\frac{1}{\sqrt{h}}$ and $g$ the solution of \eqref{n=2 line}. Beware that the parameter $g$, call it $g_n$, of the previous section depends on $n$ and satisfies $\lim_{n\to 2} g_n=g$.  The next proposition follows from computations since $W(\xi;n)$ and $\rho(u;n)$ have explicit expressions.

\begin{prop}\label{p:conv_resolvent}
 Uniformly over compact sets of $\mathbb{C}\backslash[-\gamma,\gamma]$,  $W(\xi; n)$  converges as $n\to 2$ to 
\begin{align*}
W(\xi;2):=& \frac1{4\pi^2} \bigg(\ln \bigg(\frac{\xi-\gamma}{\xi+\gamma}\bigg)\bigg)^2    \bigg(g\xi^3 - \frac{g}{h^4\xi^5} - \xi+\frac{1}{h^2\xi^3} \bigg) \\
&+
\frac{\gamma}{\pi^2} \ln \bigg(\frac{\xi-\gamma}{\xi+\gamma}\bigg) \bigg( g \bigg(\xi^2 - \frac{1}{h^3 \xi^4} \bigg)+ \Big(\frac{1}{3h} -1\Big)\bigg(1-\frac{\gamma^2}{\xi^2}\bigg)  \bigg)\\
& + \frac{g}{\pi^2 h} \left(\xi- \frac{1}{h^2\xi^3}\right)
+ \frac{1}{4}\left( \frac1{h^2\xi^3} - \frac{g}{h^4\xi^5}\right) +\frac{1}{2\xi}.
\end{align*}
Similarly,  $\rho(u;n)$ converges pointwise on $[-\gamma,\gamma]$ to some continuous function $\rho(u;2)$ and there exists a constant $c=c(\gamma)>0$ such that $\sup_{u\in [-\gamma,\gamma]} \frac{|\rho(u,n)|}{(\gamma^2-u^2)}\le c$ for all $n$ close enough to $2$.
\end{prop}

\bigskip

\noindent Taking the limit $n\to 2$ in \cref{critical n}, we obtain by the dominated convergence theorem,
\begin{equation}
  h^2 \int_{-\gamma}^\gamma \rho(u;2) \left(1-h u^2\right)^{-3/2} u^2 \mathrm{d} u  =1-\frac32 h^3g. \label{critical n=2}
\end{equation}

\noindent Let $(F_k(2;g,h),\,k\ge 1)$ be the coefficients of the Taylor expansion of $W(\, \cdot\, ;2)$ at $\infty$, \textit{i.e.}
\begin{equation}\label{n=2 Fk}
  W(\xi;2) = \sum_{k\ge 0} \frac{F_k(2;g,h)}{\xi^{2k+1}},\, |\xi|>\gamma.
\end{equation}

\noindent  By Proposition \ref{p:conv_resolvent} and the Cauchy integral formula, $F_k(2;g,h)$ is the limit of $F_k(n;g,h)$, which shows that $F_k(2;g,h)\ge 0$ for all $k\ge 0$. The expression of $F_k(2;g,h)$ can be computed explicitly
\[F_1 = \frac{13}{30}\gamma^2 + \Big(\frac{1}{4} - \frac{43}{15}\frac{1}{\pi^2}\Big)\gamma^4, \]
\[F_2 = \Big(\frac{433}{90\pi^2}-\frac{3}{8}\Big)\gamma^6 + \Big(\frac{\pi^2}{8}-\frac{10}{9}\Big)\gamma^4. \]

\noindent For $F_k$, $k \geq 3$, we have
\begin{eqnarray*}
4\pi^2 F_k &=& \bigg(\frac{4}{2k+4}g\gamma^{2k+4}c_{2k+4} - \frac{8}{2k+3}g\gamma^{2k+4}-\frac{1}{2k+1}\frac{8}{3}g\gamma^{2k+4} \\
& &
-\frac{4}{2k-4}g\gamma^{2k+4}c_{2k-4} + \frac{8}{2k-3}g\gamma^{2k-4}+\frac{1}{2k-1}\frac{8}{3}g\gamma^{2k+4}\bigg)\\
& &+\bigg(-\frac{4}{2k+2}\gamma^{2k+2}c_{2k+2} + \frac{4}{2k-2}\gamma^{2k+2}c_{2k-2} + \frac{8}{2k+1}\gamma^{2k+2}-\frac{8}{2k-1}\gamma^{2k+2}\bigg).
\end{eqnarray*}

\noindent Here $c_{2m} = 2(1 + \frac{1}{3} + \frac{1}{5} +\cdots+ \frac{1}{2m-1}) \sim \ln m$. We can check that \eqref{asymptotics} holds. Define
\begin{equation}\label{weightseq}
{\hat{g}}_k := g\delta_{2,k} + 2h^{2k}F_k(2;g,h),\, k\ge 1,
\end{equation}

\noindent and the measure $\nu(\, \cdot\, ;2)$ on $\mathbb{Z}$ by
\begin{equation}\label{nu n=2}
\nu(k;2) = \begin{cases}
{\hat{g}}_{k+1}\gamma^{2k},\quad & k\geq 0, \\
2F_{-k-1}(2;g,h)\gamma^{2k}, \quad & k\leq -1.
\end{cases}
\end{equation}

\noindent {By monotone convergence, one sees that $W(\gamma;2)$ is well-defined as a limit along real values of $\xi$. Moreover} we observe that $\sum_{k\in \mathbb{Z}} \nu(k;2) = g\gamma^2-2\gamma^{-2}+\frac{4}{\gamma}W(\gamma;2)=1$ from the expression of $W(\cdot;2)$ in Proposition \ref{p:conv_resolvent}, hence $\nu(\, \cdot\, ;2)$ is a probability measure. Recall the notation $h^{\uparrow}$ in \cref{p: admissible critical}.

\begin{lem}
For $\ell>0$, we have $\sum^{\infty}_{k=-\ell+1} h^{\uparrow}(\ell+k)\nu(k;2) = h^{\uparrow}(\ell)$.
\end{lem}

\begin{proof}
We know that for all $n<2$,
\begin{equation}\label{harmonic n}
\sum^{\infty}_{k=-\ell+1} h^{\uparrow}(\ell+k)\nu(k; n) = h^{\uparrow}(\ell). 
\end{equation}
\noindent We want to show it also holds for $n=2$. The case $\ell=1$ is  \eqref{critical n=2}.  Since $\lim_{n\to 2} \nu(k;n)=\nu(k;2)$, Scheff\'e's lemma implies that
$$
\sum_{k=0}^\infty h^{\uparrow}(1+k) |\nu(k;n)-\nu(k;2)|\to 0,\quad n\to 2.
$$

\noindent Hence, using that $h^{\uparrow}(\ell+k)\le 4^\ell h^{\uparrow}(k)$ {for $k\ge 1$} we are allowed to take the limit $n\to 2$ again in \eqref{harmonic n} for any $\ell\ge 1$.
\end{proof}

\bigskip
\noindent By Proposition \ref{p: admissible critical}, the weight sequence $({\hat{g}}_k,\, k\ge 1)$ is admissible and critical.  Theorem 1 in \cite{budd2018peeling} implies that $(g, h)$ defines a {valid} loop--$O(2)$ model on quadrangulations as in the setting of \cite{borot2011recursive}. The asymptotics \eqref{asymptotics} show that it is  non-generic critical. The expected volume $\overline{V}(p)$ of the $O(2)$ planar map with half-perimeter $p$ is 
\begin{equation}\label{expected volume Budd}
\overline{V}(p) = \frac{\gamma^{2p}}{F_p},
\end{equation}
\noindent as given by {\cite[Equation (42)]{budd2018peeling}}. In view of \eqref{asymptotics}, as $p\to \infty$, 
\begin{equation}\label{expected volume O(2)}
\overline{V}(p) \sim \Lambda \begin{cases}
  p^2, & \text{if } g=\frac{h}{2}, \\
 \frac{p^2}{\ln(p)}, & \text{if }g< \frac{h}{2}.
\end{cases}
\end{equation}

\noindent For the gasket of loop--$O(2)$ quadrangulations, the pointed partition function is given in \cite[Theorem 6]{curien2019peeling} by 
\[\dbinom{2k-1}{k-1} Z_{\hat{g}}^{k}\sim Ch^{-k}k^{-\frac12}, \quad \text{as } k\to\infty.\]

\noindent Since the expected volume is the ratio of the pointed partition function over the non-pointed partition function $F_p{(2;g,h)}$, we obtain the asymptotics \eqref{eq: mean volume boltzmann n=2}.

{Finally, we discuss properties of the measure} $\mu_{\mathrm{JS}}$ introduced in (\ref{def:muJS}). Since $Z_{\hat{g}} = \frac{1}{4h}$,  $\mu_{\mathrm{JS}}(k) = \dbinom{2k}{k}4^{-k+1}h^{k+1}F_k{(2;g,h)}$. Note that $\sqrt{k} \dbinom{2k}{k} 4^{-k} \to \frac{1}{\sqrt{\pi}}$ and that
\begin{multline*}
\sqrt{n+1}{\dbinom{2n+2}{n+1}} 4^{-(n+1)} - \sqrt{n} {\dbinom{2n}{n}} 4^{-n} \\
= \bigg(\frac{2n+1}{2\sqrt{n}\sqrt{n+1}} - 1\bigg)\sqrt{n} {\dbinom{2n}{n}} 4^{-n}
= \frac{\sqrt{n} 4^{-n}}{2\sqrt{n}\sqrt{n+1}(\sqrt{n+1}+\sqrt{n})^2}{\dbinom{2n}{n}} = O(n^{-2}).
\end{multline*}

\noindent Hence
\[\dbinom{2k}{k}4^{-k} = \frac{1}{\sqrt{\pi}\sqrt{k}} + O(k^{-\frac32}). \]

\noindent Together with the asymptotics of $F_k(2; g, h)$ {in \eqref{asymptotics}} and $h = \frac{1}{\gamma^2}$, we deduce that
\begin{equation}\label{tail muJS}
\mu_{\text{JS}}(k)  \sim C_{\text{JS}} \begin{cases}
  k^{-\frac52}, & \text{if } g=\frac{h}{2}, \\
  k^{-\frac52}\ln k, & \text{if }g< \frac{h}{2},
\end{cases}
\end{equation}
for some constant $C_{\text{JS}}>0$.

\section{Tail distribution of the Markov chain ${\tt S}$ in the case $n=2$}\label{s:tail distribution S}
{We fix $n=2$ throughout this section. We first gather some preliminary estimates on the tail asymptotics of $T_p$ and $L_p$ (recall \eqref{eq: def Tp and Lp}). The following statement is an extension of \cite[Lemma 5]{chen2020perimeter}. We will use the function $f(p)$ defined in \eqref{def: fp}.}
{
\begin{lem}\label{lem: exp tail T_p small}
There exists $c>0$ such that
\begin{equation}
\mathbb{P}(T_p\leq \varepsilon f(p))\leq \mathbb{P}(L_p\leq \varepsilon f(p))\leq 
\begin{cases}
    \exp(-c\varepsilon^{-2}), & \text{ \textnormal{in Case~\ref{caseA}},}\\
    \exp(-c\varepsilon^{-2}(|\ln\varepsilon|+1)^{-2}), & \text{ \textnormal{in Case~\ref{caseB}}.}
\end{cases}
\end{equation}
\end{lem}
}

\begin{proof}
{
We feel free to omit the proof as it is the same as in \cite[Lemma 5]{chen2020perimeter}. The only difference is that we are now in the regime where either $\mathbb{P}(X_1\geq x)\sim Cx^{-3/2}$ (in Case~\ref{caseA}) or $\mathbb{P}(X_1\geq x)\sim Cx^{-3/2}\ln x$ (in Case~\ref{caseB}), but the standard Abelian theorems still apply in both cases.
}
\end{proof}

\begin{lem}\label{lem: Tp/Lp}
Let $r > 1$. There exists some constant $C>0$ such that for all $p\ge 1$
\begin{equation}
\mathbb{E}\left[\abs{\mu_{\textnormal{JS}}(0)\frac{T_p-1}{1+L_p}-1}^r\right] \leq {Cp^{-2r}f(p)^r}. 
\end{equation}
\end{lem}

\begin{proof}
Set {$F(p) = p^{-2}f(p)$}. 
{
Note that since $\alpha>1$, for $p\ge 2$ large enough, we have $F(p)\ge \frac{1}{p-1}$ and thus
\[
\abs{\frac{1+L_p}{T_p-1} - \mu_{\textnormal{JS}}(0)}
\le 
\abs{\frac{L_p}{T_p} - \mu_{\textnormal{JS}}(0)}+ \frac{2}{p-1}
\le
\abs{\frac{L_p}{T_p} - \mu_{\textnormal{JS}}(0)} + 2F(p).
\]
Therefore, by a union bound,
\begin{align*}
&\mathbb{P}\left(\abs{\frac{1+L_p}{T_p-1} - \mu_{\textnormal{JS}}(0)}\geq 3F(p)\right) \\
&\leq  \mathbb{P}(T_p<p^{5/4})+
\mathbb{P}\left(\abs{\frac{L_p}{T_p} - \mu_{\textnormal{JS}}(0)}\geq F(p),T_p\geq p^{5/4} \right)\\
&\leq \mathbb{P}(T_p<p^{5/4})+
\sum^{\infty}_{n=p^{5/4}} \mathbb{P}\left(\abs{\frac1{n} \sum^n_{i=1}\mathds{1}_{\{X_i = -1\}} -\mu_{\text{JS}}(0)} \geq F(p)\right).
\end{align*}
}
By Hoeffding's inequality, {the second term can be further bounded as
\[
\sum^{\infty}_{n=p^{5/4}} \mathbb{P}\left(\abs{\frac1{n} \sum^n_{i=1}\mathds{1}_{\{X_i = -1\}} -\mu_{\text{JS}}(0)} \geq F(p)\right)
\leq \sum^{\infty}_{n=p^{5/4}} 2\mathrm{e}^{-2nF(p)^2}
\leq 4F(p)^{-2}\mathrm{e}^{-2f(p)^2p^{-11/4}}.
\]
}
{
We conclude by definition of $f(p)$ in \eqref{def: fp} and \cref{lem: exp tail T_p small} that there exists a constant $C>0$ such that for all $p$,}
\begin{equation*}
\mathbb{P}\left(\abs{\frac{1+L_p}{T_p-1} - \mu_{\textnormal{JS}}(0)}\geq 3{F(p)}\right) \leq C\mathrm{e}^{-2p^{{1/8}}}.
\end{equation*}

\noindent By \cref{eq: exponential tail Tp/Lp}, for any $r'>1$,  
\[\sup_{p\ge 1} \mathbb{E}\left[ \bigg| \mu_{\textnormal{JS}}(0) \frac{T_p-1}{1+L_p}-1 \bigg|^{r'}\right]<\infty. \]

\noindent {The Cauchy-Schwarz inequality} implies that {there is a constant $C>0$ such that for all $p$,}
\[\mathbb{E}\left[ \bigg|\mu_{\textnormal{JS}}(0)\frac{T_p-1}{1+L_p}-1\bigg|^r\mathds{1}_{\{|{(1+L_p)/(T_p-1)} - {\mu_{\textnormal{JS}}(0)}|\geq 3{F(p)}\}}\right]\leq C\mathrm{e}^{-p^{{1/8}}}.\]

\noindent We finally observe that {on the event that $|{(1+L_p)/(T_p-1)} - \mu_{\textnormal{JS}}(0)|< 3F(p)$,
\[
\abs{\mu_{\textnormal{JS}}(0)\frac{T_p-1}{1+L_p}-1}
<
\frac{3 F(p)}{\mu_{\mathrm{JS}}(0)-3F(p)}.
\]
}
{Hence for $p$ large enough so that $3F(p)<\mu_{\mathrm{JS}}(0)/2$, } 
\[\mathbb{E}\left[ \abs{\mu_{\textnormal{JS}}(0)\frac{T_p-1}{1+L_p}-1}^r\mathds{1}_{\{|{(1+L_p)/(T_p-1)} - {\mu_{\textnormal{JS}}(0)}|< 3{F(p)}\}}\right]\leq { \frac{3^r F(p)^{r}}{(\mu_{\mathrm{JS}}(0)-3F(p))^r}\leq 6^r\mu_{\mathrm{JS}}(0)^{-r}F(p)^r},\]

\noindent which completes the proof.

\end{proof}

{We now extend the convergence of the perimeter cascade to the general case.}
\begin{prop}[{Extension of \cite[Proposition 3]{chen2020perimeter}}]
\label{prop: CCM convergence gen 1}
For any $\varphi: \ell^{\infty}(\mathbb{N}^*) \to \mathbb{R}$ bounded and continuous, we have {as $p\to\infty$,}
\[\mathbb{E}^{(p)}[\varphi(p^{-1}\chi^{(p)}(i), i\geq 1)]\longrightarrow\mathbb{E}[\varphi(Z_{\alpha}(i), i\geq 1)].\]
\end{prop}

\begin{proof}
Recall that faces of the gasket correspond to loops in $(\frak{q}, \boldsymbol{\ell})$, except for some faces of degree $4$. By \eqref{eq: key formula maps to RW}, we thus have as $p\to \infty$,
\[\mathbb{E}^{(p)}[\varphi(p^{-1}\chi^{(p)}(i), i\geq 1)] = \frac{1}{\mathbb{E}[1/L_p]}\mathbb{E}\bigg[\frac{\varphi(p^{-1} \mathbf{X}^{(p)})}{1+L_p}\bigg] + o(1).\]

\noindent {Recall the function $f(p)$ defined in \eqref{def: fp}.} The standard invariance principle gives that 
\[\bigg(\frac{S_{t f(n)}}{n}\bigg)_{t\geq 0} \xrightarrow{\text{(d)}} (\zeta_t)_{t\geq 0}. \]

\noindent Hence 
\[f(p)^{-1} T_p \xrightarrow{\text{(d)}} \tau, \quad f(p)^{-1} L_p \xrightarrow{\text{(d)}}{\mu_{\text{JS}}(0)}\tau \quad  \textnormal{and} \quad p^{-1}\mathbf{X}^{(p)} \xrightarrow{\text{(d)}} (\Delta \zeta_t)_{t\leq \tau}. \]

\noindent 
{\cref{eq: L1 convergence Lp}} 
indicates that $\mathbb{E}[f(p)/L_p] \to \mu_{\textnormal{JS}}(0)^{-1}\mathbb{E}[1/\tau]$ and 
\[f(p)\mathbb{E}\left[\frac{\varphi(p^{-1} \mathbf{X}^{(p)})}{1+L_p}\right]\to \mu_{\textnormal{JS}}(0)^{-1} \mathbb{E}\left[\frac{\varphi(\Delta\zeta_t, t\leq \tau)}{\tau}\right]. \]

\noindent We conclude the desired convergence by definition of $(Z_{\alpha}(i), i\leq 1)$. 
\end{proof}

Now we turn to the speed of convergence of the Markov Chain $\tt S$ relative to the $O(2)$ model, {as defined in \cref{sec: markov chain}.}
\begin{lem}\label{l:convergence speed S1}
There exists a constant $C>0$ such that, for all large enough $p$, and for all $n\geq 0$,
\begin{equation} \label{eq: distribution fn n=2}
\abs{\mathbb{P}_p({\tt S_1} \leq n) - \frac{2}{\pi}\arctan\sqrt\frac{n}{p}} \leq \frac{C}{\sqrt{p}}.
\end{equation}
\end{lem}

\begin{proof}
In this proof the $O(\cdot)$ estimates will always be independent of $n$. Note that in the gasket $\frak g$, a face of degree $2k$ when $k\ne 2$ will correspond to a loop. Also note that all faces of degree $4$ of the gasket will independently have probability $a = \frac{g}{\hat{g}_2}$ to be a plain quadrangle and $1-a$ to stem from a loop. Let {$g(k) = \overline{V}(k)(1-a\mathds{1}_{\{k=2\}})$}. By definition of ${\tt S}$, (\ref{eq: key formula maps to RW}) and (\ref{eq: volume ratio}), for any $n\geq 3$, 
\begin{eqnarray*}
{\tt P}_p({\tt S}_1\leq n) 
&=& \frac{1}{\overline{V}(p)}\mathbb{E}^{(p)}\left[\sum^{\infty}_{i=1} \overline{V}(\chi^{(p)}(i))\mathds{1}_{\{\chi^{(p)}(i)\leq n\}}\right] + \frac{1}{\overline{V}(p)}\mathbb{E}^{(p)}\left[|\frak g| \right]\\
&=& \frac{1}{\overline{V}(p)} \frac{1}{\mathbb{E}[\frac{1}{1+L_p}]} \mathbb{E}\bigg[\frac{1}{1+L_p} \sum_{i=1}^{T_p} g(X_i + 1)\mathds{1}_{\{X_i + 1\leq n\}}\bigg]+ O(p^{-1/2}).
\end{eqnarray*}

\noindent By the same arguments on the faces of degree $4$ and the gasket decomposition we have 
\begin{eqnarray*}
\overline{V}(p) &=& \mathbb{E}^{(p)}[|\frak g|] + \mathbb{E}^{(p)}\bigg[\sum^{\infty}_{i=1}\overline{V}(\chi^{(p)}(i))\bigg]\\
&=& \frac{1}{\mathbb{E}[\frac1{1+L_p}]} + \frac{1}{\mathbb{E}[\frac1{1+L_p}]}\mathbb{E}\bigg[\frac{1}{1+L_p} \sum_{i=1}^{T_p} g(X_i + 1)\bigg]. 
\end{eqnarray*}

\noindent We may then replace $\overline{V}(p)\mathbb{E}[\frac{1}{1+L_p}]$ using the display above:
\[{\tt P}_p({\tt S}_1 \leq n) = \frac{\mathbb{E}\left[\frac{1}{1+L_p} \sum_{i=1}^{T_p} g(X_i+1)\mathds{1}_{\{X_i+1} \leq n\}\right]}{1+\mathbb{E}\left[\frac{1}{1+L_p} \sum_{i=1}^{T_p} g(X_i + 1)\right]} + O(p^{-1/2}).\]

\noindent We will see that both expectations are of order $\sqrt{p}$ later. 

To apply \cref{prop: kemperman}, we first need to replace $1/(L_p+1)$ by $1/(T_p-1)$ inside the expectation. By H\"older's inequality with coefficients $\frac1r+\frac1{r'} = 1$, we have 
\begin{eqnarray*}
&\abs{\mathbb{E}\left[\frac{1}{1+L_p} \sum_{i=1}^{T_p} g(X_i + 1)\right] - \mathbb{E}\left[\frac{1}{(T_p-1)\mu_{\textnormal{JS}}(0)} \sum_{i=1}^{T_p} g(X_i + 1)\right] } \\
\leq &\mathbb{E}\left[\abs{\mu_{\textnormal{JS}}(0)\frac{T_p-1}{1+L_p}-1}^r\right]^{\frac1r}\mathbb{E}\Big[\Big(\frac{1}{(T_p - 1)\mu_{\textnormal{JS}}(0)}\sum^{T_p}_{i=1} g(X_i+1)\Big)^{r'}\Big]^\frac{1}{r'}. 
\end{eqnarray*}

\noindent Fix $r>1$ large enough such that $r'-1<\gamma_0$ in \cref{thm: tail biggins}. {Taking another set of coefficients $\frac{1}{s}+\frac{1}{s'} = 1$ with $s'r'-1<\gamma_0$, by H\"older's inequality again, there exists a constant $C>0$ such that, for all $p$,
\[\mathbb{E}\bigg[\bigg(\frac{1}{(T_p - 1)\mu_{\textnormal{JS}}(0)}\sum^{T_p}_{i=1} g(X_i+1)\bigg)^{r'}\bigg]^\frac{1}{r'}\leq C\mathbb{E}\bigg[\frac{1}{(T_p-1)^{sr'-s+1}}\bigg]^{\frac{1}{sr'}}\mathbb{E}\bigg[\frac{1}{T_p - 1}\bigg(\sum^{T_p}_{i=1} g(X_i+1)\bigg)^{s'r'}\bigg]^\frac{1}{s'r'}. \]

\noindent By \eqref{eq:tail biggins proof eta}, the second term is bounded by $Cp^2f(p)^{-1/s'r'}$. Since $f(p)/T_p\to 1/\tau$ in any $L^r$ for $r>1$ (see \cref{sec: estimtate r.w.}), the first term is bounded by $Cf(p)^{-(1-1/r'+1/sr')}$. We conclude that 
\[\mathbb{E}\bigg[\bigg(\frac{1}{(T_p - 1)\mu_{\textnormal{JS}}(0)}\sum^{T_p}_{i=1} g(X_i+1)\bigg)^{r'}\bigg]^\frac{1}{r'}\leq Cp^2f(p)^{-1}. \]

\noindent Moreover, using }\cref{lem: Tp/Lp}, 
\[\mathbb{E}\left[\abs{\mu_{\textnormal{JS}}(0)\frac{T_p-1}{1+L_p}-1}^r\right]^{\frac1r} \leq Cp^{-2}f(p). \]
{
\noindent Hence we have shown that} 
\[\bigg|\mathbb{E}\bigg[\frac{1}{1+L_p} \sum_{i=1}^{T_p} g(X_i + 1)\bigg] - \mathbb{E}\bigg[\frac{1}{(T_p-1)\mu_{\textnormal{JS}}(0)} \sum_{i=1}^{T_p} g(X_i + 1)\bigg] \bigg| \leq {C}. \]

\noindent The inequality still holds when adding the indicator $\mathds{1}_{\{X_i+1\leq n\}}$, by the same argument. We have now replaced $1/(L_p+1)$ by $1/(T_p-1)$ inside the expectation {with an error term $O(1)$.} By \cref{prop: kemperman}, 
\[
 \mathbb{E}\bigg[\frac{1}{T_p-1} \sum_{i=1}^{T_p} g(X_i+1)\bigg]
=
\bb E\bigg[ \frac{p}{p+X_1} g(X_1+1)]\bigg].
\]

\noindent It remains to show that 
\[
\Bigg|
\frac{\mathbb{E}\Big[\frac{p}{p + X_1}g(X_1 + 1)\mathds{1}_{\{X_i + 1\leq n\}}\Big] {+O(1)}}
{1+\mathbb{E}\Big[\frac{p}{p+X_1} g(X_1 + 1)\Big]{+O(1)}}
-
\frac{2}{\pi}\arctan\sqrt{\frac{n}{p}}
\Bigg|
\leq \frac{C}{\sqrt{p}}. 
\]

\noindent Recall from \cref{s:construction O(2)} that $\overline{V}(p) = \frac{\gamma^{2p}}{F_p}$ with $\gamma=1/\sqrt{h}$, and $\mu_{\mathrm{JS}}(k) = \dbinom{2k}{k}4^{-k+1}h^{k+1}F_k$. Therefore, using $\dbinom{2k}{k}4^{-k} = \frac{1}{\sqrt{\pi}\sqrt{k}} + O(k^{-\frac32})$ again,
\begin{equation} \label{eq: vol*mu}
\overline{V}(k)\mu_{\mathrm{JS}}(k) = h \cdot \dbinom{2k}{k} 4^{-k+1}  = \frac{4h}{\sqrt{\pi}\sqrt{k}} + O(k^{-3/2}). 
\end{equation}

\noindent From \eqref{eq: vol*mu}, we get to
\begin{eqnarray*}
& & \bb E\bigg[ \frac{p}{p+X_1} g(X_1 + 1)\mathds{1}_{\{X_i + 1 \leq n\}}\bigg] = \sum^{n}_{k=3} \overline{V}(k)\mu_{\textnormal{JS}}(k)\frac{p}{p+k-1} + O(1)\\
&=& 4h\sum^{n}_{k=3} \frac{1}{\sqrt{\pi}\sqrt{k}}\frac{p}{p+k-1} + O(1) = \sqrt{p}\frac{4h}{\sqrt{\pi}}\frac{1}{p}\sum^{n}_{k=2}\sqrt{\frac{p}{k}}\frac{p}{p+k-1} + O(1). 
\end{eqnarray*}

\noindent For each $k$, 
\[\int^{\frac{k+1}{p}}_{\frac{k}{p}}\frac{\mathrm{d}x}{\sqrt{x}(1+x)}\leq \frac{1}{p}\sqrt{\frac{p}{k}}\frac{1}{1 + \frac{k-1}{p}}\leq \int^{\frac{k-1}{p}}_{\frac{k-2}{p}}\frac{\mathrm{d}x}{\sqrt{x}(1+x)}.\]

\noindent Summing over $2\leq k\leq n-1$ (or $2\leq k\leq n$) leads to 
\[\int^{\frac{n}{p}}_{\frac{2}{p}} \frac{\mathrm{d} x}{\sqrt{x}(1+x)} + \frac{1}{\sqrt{np}}\frac{p}{p+n-1} \leq \sum^{n}_{k=2}\frac{1}{p}\sqrt{\frac{p}{k}}\frac{1}{1+\frac{k-1}{p}}
\leq \int^{\frac{n-1}{p}}_{0} \frac{\mathrm{d}x}{\sqrt{x}(1+x)}.\]

\noindent The difference between $\int^{\frac{n}{p}}_{0} \frac{\mathrm{d} x}{\sqrt{x}(1+x)}$ and the series is of order $\frac{1}{\sqrt{p}}$. Thus we conclude that
\[\bb E\bigg[ \frac{p}{p+X_1} g(X_1 + 1)\mathds{1}_{\{X_i + 1\leq n\}}\bigg] = \frac{4h}{\sqrt{\pi}}\sqrt{p}\int^{\frac{n}{p}}_{0}\frac{\mathrm{d}x}{\sqrt{x}(1+x)} + O(1) 
= \frac{8h}{\sqrt{\pi}} \sqrt{p} \arctan \sqrt{\frac{n}{p}} + O(1).\]
The same argument leads to
\[\mathbb{E}\bigg[\frac{p}{p+X_1} g(X_1 + 1)\bigg] = \frac{4h}{\sqrt{\pi}}\sqrt{p}\int^{\infty}_0 \frac{\mathrm{d}x}{\sqrt{x}(1+x)} + O(1)
= \frac{\pi}{2} \frac{8h}{\sqrt{\pi}} \sqrt{p}  + O(1).\]
\noindent This completes the proof. 

\end{proof}



\bibliographystyle{alpha}
\bibliography{biblio}

\end{document}